\documentclass[11pt]{article}

\usepackage{amsmath}
\usepackage{amsthm}
\usepackage{amsfonts}
\usepackage{amssymb}
\usepackage{mathtools}
\usepackage{hyperref}

\usepackage[disable]{todonotes} 

\usepackage{cases}
\usepackage{xcolor}

\newcommand{\vech}[1]{\vec{#1}_h}

\newcommand{\dt}{\partial_t}
\newcommand{\dv}{\mathrm{div}\,}
\newcommand{\dvh}{\mathrm{div}_h\,}
\newcommand{\nablah}{\nabla_h}
\newcommand{\deltah}{\Delta_h}
\newcommand{\dz}{\partial_z}
\newcommand{\inth}{\int_{\Omega_h}}
\newcommand{\intw}{\int_{\Omega}}
\newcommand{\idxh}{\,d\vec{x}_h}
\newcommand{\idx}{\,d\vec{x}}
\newcommand{\subeqref}[2]{$\eqref{#1}_{#2}$}
\newcommand{\abs}[2]{\bigl| #1 \bigr|^{#2}}
\newcommand{\norm}[2]{\bigl\Arrowvert #1 \bigr\Arrowvert_{#2}}
\newcommand{\hnorm}[2]{\bigl| #1 \bigr|_{#2}}
\newcommand{\Lnorm}[1]{L^{#1}}
\newcommand{\Hnorm}[1]{H^{#1}}
\newcommand{\hHnorm}[1]{H^{#1}_h}

\newtheorem{proposition}{Proposition}
\newtheorem{lemma}{Lemma}
\newtheorem{theorem}{Theorem}

\theoremstyle{remark}
\newtheorem{remark}{Remark}

\numberwithin{equation}{section}
\allowdisplaybreaks

\title{Asymptotic stability of the equilibrium for the free boundary problem of a compressible atmospheric primitive model with physical vacuum}

\author{
Xin Liu
\footnote{Department of Mathematics, Texas A\&M University, College Station, TX 77843-3368.
\href{mailto:xliu23@tamu.edu}{xliu23@tamu.edu}
}, 
\quad 
Edriss S. Titi
\footnote{Department of Mathematics, Texas A\&M University, College Station, TX 77843-3368, USA; Department of Applied Mathematics and Theoretical Physics, University of Cambridge, Cambridge CB3 0WA UK; also Department of Computer Science and Applied Mathematics, Weizmann Institute of Science, Rehovot 76100, Israel.
\href{mailto:titi@math.tamu.edu}{titi@math.tamu.edu} and \href{mailto:Edriss.Titi@maths.cam.ac.uk}{Edriss.Titi@maths.cam.ac.uk}
},
\quad and \quad
Zhouping Xin
\footnote{
	The Institute of Mathematical Sciences, The Chinese University of Hong Kong, Shatin, New Territories, Hong Kong, \href{mailto:zpxin@ims.cuhk.edu.hk}{zpxin@ims.cuhk.edu.hk}
}
}

\date{\today}

\begin{document}
\maketitle
\begin{abstract}
	This paper concerns the large time asymptotic behavior of solutions to the free boundary problem of the compressible primitive equations in atmospheric dynamics with physical vacuum. Up to second order of the perturbations of an equilibrium, we have introduced a model of the compressible primitive equations with a specific viscosity and shown that the physical vacuum free boundary problem for this model system has a global-in-time solution converging to an equilibrium exponentially, provided that the initial data is a small perturbation of the equilibrium. More precisely, we introduce
	a new coordinate system by choosing the enthalpy (the square of sound speed) as the vertical coordinate, and 
	thanks to the hydrostatic balance, the degenerate density at the free boundary admits a representation with separation of variables in the new coordinates. Such a property allows us to establish horizontal derivative estimates without involving the singular vertical derivative of the density profile, 
	which plays a key role in our analysis. 
	\bigskip
	{\par\noindent\bf Keywords: } Hydrostatic approximations, Compressible flows, Asymptotic stability, Free boundary problem, Physical vacuum, Primitive equations, Modified viscosity, Atmospheric dynamics
	{\par\noindent\bf MSC2020: } 35Q35, 76E20, 76N10, 86A10
\end{abstract}
	\tableofcontents

\section{Introduction}

\subsection{Compressible flows with physical vacuum}
The dynamics of compressible flows is governed by the compressible Euler or Navier-Stokes system, describing the conservations of mass, momentum, and energy. See, e.g., \cite{Feireisl2009a,Feireisl2004,Lions1996}. In particular, for an isentropic flow, under the influence of gravity, one has that
\begin{equation*}\tag{CNS}\label{CNS}
	\begin{cases}
		\dt \rho + \dv (\rho u) = 0,\\
		\dt (\rho u) + \dv (\rho u\otimes u) + \nabla P(\rho) = \mathcal D_v - \rho g \vec{e}_3,
	\end{cases}
\end{equation*}
where $ \rho \in [0,\infty), u \in \mathbb R^3 $ are the scalar and vector-valued functions, representing the density and the velocity field of the compressible flow, respectively. Here $ P(\rho) = \rho^\gamma $ is the pressure in isentropic flows, for $ \gamma > 1 $. $ \mathcal D_v $ represents the viscous term in the Navier-Stokes equations, and $ \vec{e}_3 = (0,0,1)^\top $. 
We consider the stationary equilibrium state of system \eqref{CNS}, $ (\rho_e, u_e) \equiv (\bar\rho, 0) $ with $ \bar\rho = \bar \rho (z) $ satisfying
\begin{equation*}
	\nabla P(\bar\rho) = - \bar\rho g \vec{e}_3.
\end{equation*}
Then it is clear that in the subset $ \lbrace \bar\rho > 0\rbrace $,
\begin{equation*}
	- \infty < \dz P'(\bar\rho) = \gamma \dz \bar\rho^{\gamma-1} < 0.
\end{equation*}
In particular, when vacuum occurs, i.e., $ \lbrace \bar\rho = 0 \rbrace \neq \emptyset $, $ \bar\rho $ is only H\"older continuous across the gas-vacuum interface. 

In this work, the stationary density profile is taken to be
\begin{equation}\label{def:rho-bar}
	\bar\rho^{\gamma-1} = \dfrac{\gamma-1}{\gamma} g (1 - z),
\end{equation}
where $ \lbrace z=0 \rbrace $ and $ \lbrace z=1 \rbrace $ is the solid ground and the stationary gas-vacuum interface, respectively.

Such a kind of singularity/degeneracy is referred to as the physical vacuum and it appears widely in compressible flows connected to vacuum, e.g., the flows in porous media \cite{Barenblatt1953,Liu1996} and configurations of the gaseous stars \cite{Chandrasekhar1958,KUAN1996,Fu1998}. We also refer to related works in \cite{Chanillo1994,Chanillo2012,Luo2008a,Luo2009b,Caffarelli1980}. Such a singularity makes the study of the evolutionary problem more challenging. As pointed out by T.-P. Liu in  \cite{Liu1996}, the standard symmetrizing method of hyperbolic systems does not apply in the case of compressible Euler equations with physical vacuum density. Essentially, on the gas-vacuum interface, the compressible Euler system loses the hyperbolicity and {\it the characteristic surfaces hit the vacuum surface tangentially and then bounce back}. The local well-posedness of the free boundary problem with physical vacuum is developed in \cite{Jang2009b,Jang2015,Jang2010,LuoXinZeng2014,Gu2012,Gu2015,Coutand2012,Coutand2011a,Coutand2010}. 
In the case of one-dimensional motion or spherically symmetric motion, the long time dynamics has been investigated for the compressible Euler equations with damping \cite{ZengHH2017,ZengHH2015a}, the compressible Navier-Stokes equations \cite{ZengHH2015,Luo2000a,XL2018}, and the gaseous star problem \cite{Hadzic2016,Jang2014,Jang2013a,LuoXinZeng2016,LuoXinZeng2015,XL2018a, XL2018a-2}. Without the symmetry assumption, see, e.g., \cite{Hadzic2018,Shkoller2019,Sideris2017} for the study of the compressible Euler equations, and \cite{ZengHH2021a,ZengHH2021b,ZengHH2022} for the study of the compressible Euler equations with damping.

On the other hand, the asymptotic stability of the physical vacuum profile \eqref{def:rho-bar} with general perturbation is still open, for either the compressible Euler or Navier-Stokes equations. The main difficulty is that any spatial tangential/horizontal derivative of the momentum equations \subeqref{CNS}{2} will involve derivative of $ \rho $ along the singular/gradient direction ($ \nabla \rho $), regardless how small the perturbation is. {In this paper, we consider a compressible model in the context of atmospheric dynamics, and establish the asymptotic stability result for the underlying equilibrium. }

\bigskip

Specifically, we study system \eqref{CNS} with the {\bf hydrostatic approximation}, which can be obtained by performing a small aspect ratio (between vertical and horizontal scales) limit. 
The resulting equations are referred to as the compressible primitive equations (CPE) of atmospheric dynamics (c.f. \cite{Ersoy2011a}, or see \eqref{isen-CPE-fb}, below). Similarly, as in \eqref{CNS}, the physical vacuum profile $(\rho_e, u_e)  = (\bar\rho, 0) $ with $ \bar\rho $ as in \eqref{def:rho-bar} is a stationary solution to the CPE. We will investigate the nonlinear stability of such a profile and establish a global solution for the compressible flow with the physical vacuum singularity in multi-dimensional space without imposing symmetry. 

The hydrostatic approximation, i.e., formally replacing the equation of the vertical momentum in \eqref{CNS} by the hydrostatic balance equation (or quasi-static equilibrium equation)
\begin{equation}\label{eq:hydrostatic}
	\partial_z P(\rho) = - \rho g,
\end{equation}
is used to model flows with a much larger horizontal scale than the vertical scale such as the case of atmospheric flows. Indeed, the hydrostatic balance approximation turns out to be highly accurate for atmosphere and is a fundamental equation in the atmospheric science (see, e.g., \cite{Lions1992,Richardson1965,Washington2005}). 

The major stabilizing feature of the CPE in comparison to \eqref{CNS} is that, the hydrostatic balance \eqref{eq:hydrostatic} indicates that the density is always monotone in the vertical direction. This inspires us to introduce the change of coordinates \eqref{new-coordinates}, below, which is the key starting point of this work. 

\bigskip

The first mathematical study of the CPE can be tracked back to Lions, Temam, and Wang \cite{Lions1992}. By formulating the equations in the pressure coordinates ($ p $-coordinates), the resulting equations are in the form of the incompressible primitive equations (referred to as the primitive equations, or the PE in the following). In yet another work \cite{JLLions1992}, the same authors established the PE as the hydrostatic approximation of the Boussinesq equations. They showed the existence of global weak solutions and therefore indirectly studied the CPE. See \cite{Lions2000,JLLions1994} for more works by the authors. 
With fixed boundaries, Gatapov, Kazhikhov, Ersoy, and Ngom constructed a global weak solution to the two-dimensional CPE in \cite{Ersoy2012,Gatapov2005}. The uniqueness of such weak solutions is shown in \cite{JiuLiWang2018}. Meanwhile, Ersoy, Ngom, Sy, Tang, and Gao studied the compactness/stability of sequences of weak solutions to the CPE with degenerate viscosities in \cite{Ersoy2011a,Tang2015}.
In \cite{LT2018b}, we show the existence of global weak solutions of the CPE with degenerate viscosities satisfying the entropy conditions. The same result is independently shown in \cite{WangDouJiu2020}. Also, in \cite{LT2018a}, we show the local well-posedness of strong solutions to the CPE.


\subsection{The change of coordinates and a viscous compressible primitive system}

This work is a contribution to our program of studying the compressible primitive equations. We focus here on investigating the long time dynamics of the CPE with the free boundary condition, 
and the stability of the equilibrium state with physical vacuum. 
The major benefit given by the hydrostatic approximation is that the enthalpy (the square of the sound speed) $ \partial_\rho P(\rho) $ is determined by the height of the moving boundary \eqref{Square-SS}, below. Consequently, by taking the ratio between the enthalpy and its value on the solid ground as the vertical coordinate (see \eqref{new-coordinates}, below) 
together with the equation for the square of the sound speed and the new ``vertical velocity'' \eqref{def:new-vertical-v}, below, we will arrive at a reformulation of the free boundary problem in a fixed domain (see \eqref{rfeq:isen-CPE-fb}, below). Such a formulation introduces a system that describes explicitly the physical vacuum in the degenerate direction, a separation of variables of the density, and a simple representation of the dynamical boundary condition (see \eqref{new-bc-01}, below). 
However, the loss of evolutionary equation for the vertical velocity, due to the hydrostatic balance, brings new degeneracy to the system. We overcome this difficulty through a complicated, but straightforwards ``vertical velocity'' formula \eqref{eq:verticalvelocity}, below. Moreover, as we will explain later, since we only focus on the global stability theory, we will {\bf linearize the viscosities and the boundary condition} in our equations. This is done by neglecting the $ \mathcal O (\varepsilon^2) $ terms, with $ \varepsilon $ representing the size of perturbation. However, 
this disrupts the conservation of momentum. To avoid this shortcoming, we add some momentum correction 
term on the right-hand side of the momentum equations. 
Such a correction is important for establishing the Poincar\'e type inequality (Lemma \ref{lm:poincare-ineq}) and showing the asymptotic stability (Proposition \ref{prop:decay-est}). In fact, we show that the perturbation around the equilibrium state will decay exponentially for $ \gamma > 4/3 $, which incidentally is the same stability threshold established for the gaseous star problem (see \cite{LuoXinZeng2015,LuoXinZeng2016,Jang2014,Lin1997}). We refer to such a system with the momentum correction term as the conservative system. In section \ref{sec:non-conservative}, we also investigate a non-conservative system with only the linear viscosity terms.

\bigskip

We remark that the aforementioned reformulation of the free boundary problem of the compressible primitive equations has been introduced and presented earlier in \cite{LT2018a}. Below, we will recall it for the sake of completeness. Indeed, by denoting $ v, w $ as the horizontal and vertical components of the velocity $ u $ (i.e. $ u = (v, w) \in \mathbb R^2 \times \mathbb R$), the compressible primitive system with moving free boundary is given by:
\begin{equation*}\tag{FBCPE}\label{isen-CPE-fb}
	\begin{cases}
		\dt \rho + \dvh (\rho v) + \dz (\rho w) = 0 & \text{in} ~ \Omega(t), \\
	\dt (\rho v) + \dvh (\rho v \otimes v) + \dz (\rho w v) + \nablah P = \mathcal D_v & \text{in} ~ \Omega(t),\\
	\dz P + \rho g = 0 & \text{in} ~ \Omega(t),
	\end{cases}
\end{equation*}
where $ P = \rho^\gamma, \gamma > 1 $ and $ \mathcal D_v = \mathcal D_v(\nabla v) $ is the eddy viscosity term given by
\begin{equation*}
	\mathcal D_v := \mu \dvh (\nablah v+ \nablah v^\top) + \lambda \nablah \dvh v + \mu' \partial_{zz} v,
\end{equation*}
with the constant viscosity coefficients $ \mu, \lambda, \lambda' $ satisfying the Lam\'e conditions $ \mu > 0, \mu' > 0, \mu + \lambda > 0 $. 
Here $ \Omega(t) = \lbrace \vec x = (\vech{x}, z) \in \mathbb T^2 \times (0 , Z(\vech{x},t)) \rbrace $ is the evolving domain, with $ z = 0 $ being the bottom solid ground and $ z = Z(\vech{x},t) $ being the interface (free boundary) connecting the flows and the vacuum. 

System \eqref{isen-CPE-fb} should be complemented with 
the following boundary conditions:
\begin{equation}\label{bc-fb}
	\begin{gathered}
	- \dt Z - v|_{z =  Z(\vech{x},t)} \cdot \nablah Z + w|_{z =  Z(\vech{x},t)}=0 ,\\
	 P|_{z =  Z(\vech{x},t)} = 0,
	w|_{z = 0} = 0, 
	\dz v|_{z=0} = 0, \\ 
	\biggl( \begin{array}{cc}
		\mu (\nablah v+\nablah v^\top) + \lambda \dvh v \mathbb I_{2} & \mu' \dz v \\
		0 & 0
	\end{array}\biggr)\Big|_{z=Z}\vec{N}(\vech{x},t) = 0, \\
	\text{where} ~ \vec{N}(\vech{x},t) = \biggl( \begin{array}{c}
		- \nablah Z \\
		1
	\end{array}\biggr). 
	\end{gathered}
\end{equation}
In this work, we assume that $ \mu = \mu' > 0 $ and $ \lambda > 0 $.
%

The hydrostatic balance equation \subeqref{isen-CPE-fb}{3} yields that the enthalpy (the square of sound speed) can be calculated as 
\begin{equation}\label{Square-SS}
\begin{aligned}
	\rho^{\gamma-1}(\vech x, z, t) & = \dfrac{\gamma-1}{\gamma} g( Z(\vech{x},t)-z) + \rho^{\gamma-1}(\vech{x}, Z(\vech{x},t),t)\\
	& = \dfrac{\gamma-1}{\gamma} g(Z(\vech{x},t)-z),
\end{aligned}
\end{equation}
where we use in the above that $ \gamma > 1 $ and that the density at the free boundary $ z = Z(\vech x, t) $ vanishes. 
On the other hand, from \subeqref{isen-CPE-fb}{1}, we have
\begin{equation}\label{CPE-fb-001}
	\dt \rho^{\gamma-1} + v \cdot \nablah \rho^{\gamma-1} + w \dz  \rho^{\gamma-1} + (\gamma-1) \rho^{\gamma-1}(\dvh v + \dz w) = 0.
\end{equation}
Now, we define the new coordinates as
\begin{equation}\label{new-coordinates}
	\begin{gathered}
		 \vech{x}' = (x',y') := \vech{x} = (x,y), ~ t' = t, \\
		 z' := \dfrac{\rho^{\gamma-1}}{\rho^{\gamma-1}|_{z=0}} =  \dfrac{Z(\vech{x},t)-z}{Z(\vech{x},t)}= 1 - \dfrac{z}{Z(\vech{x},t)} .
	 \end{gathered}
\end{equation}
In this new coordinate system, $ \lbrace z' = 0 \rbrace $ and $ \lbrace z' = 1 \rbrace $ correspond to the upper moving boundary connecting the fluid to the vacuum and the bottom fixed boundary, 
respectively. 
It is easy to verify that $ \nabla_{\vech{x}'} Z = \nabla_{\vech{x}} Z, \dt Z = \partial_{t'} Z $ and the following chain rules: 
\begin{align*}
	\partial_{x} & = \partial_{x'} + \biggl(\dfrac{z}{Z^2} \partial_x Z \biggr)  \partial_{z'} = \partial_{x'} + \biggl( \dfrac{1- z'}{Z} \partial_{x'} Z \biggr) \partial_{z'}, \\
	\partial_{y} & = \partial_{y'} + \biggl(\dfrac{z}{Z^2} \partial_y Z\biggr)  \partial_{z'} = \partial_{y'} + \biggl(\dfrac{1-z'}{Z} \partial_{y'} Z\biggr) \partial_{z'}, \\
	\partial_z & = - \dfrac{1}{Z} \partial_{z'} ,\\
	\partial_t & = \partial_{t'} + \biggl(\dfrac{z}{Z^2} \partial_t Z\biggr)  \partial_{z'} = \partial_{t'} + \biggl(\dfrac{1-z'}{Z} \partial_{t'} Z\biggr) \partial_{z'}, \\
	\vec x' = (\vech{x}', z') & \in \Omega_h \times(0,1) =: \Omega'_e.
\end{align*}
Recall that $ \Omega_h = \mathbb T^2 $ is the horizontal domain. 
Now we rewrite equations \eqref{isen-CPE-fb} and \eqref{CPE-fb-001} in the new coordinates $ (\vech{x}',z',t') $ as follows:
\begin{equation*}
	\begin{cases}
		\partial_{t'} Z + v \cdot \nabla_{\vech{x}'} Z - w + (\gamma-1) z' \bigl( Z \dv_{\vech{x}'} v \\
		~~~~ ~~~~ ~~~~ - (z'-1) \nabla_{\vech{x}'} Z \cdot \partial_{z'} v - \partial_{z'} w\bigr) = 0 & \text{in} ~ \Omega'_e,\\
		\rho \bigl(\partial_{t'} v + v \cdot \nabla_{\vech{x}'} v + \dfrac{1}{Z}( -(z'-1) \partial_{t'}Z - (z'-1) v \cdot \nabla_{\vech{x}'} Z\\
		~~~~ ~~~~ ~~~~ - w )\partial_{z'} v  + g \nabla_{\vech{x}'} Z  \bigr) = \mathcal D_v'& \text{in} ~ \Omega'_e,\\
		\partial_{z'} Z = 0 & \text{in} ~ \Omega'_e,
	\end{cases}
\end{equation*}
where \begin{equation}\label{def:density} \rho = \biggl( \dfrac{\gamma-1}{\gamma} g z' Z \biggr)^{1/(\gamma-1)},\end{equation}
and $ \mathcal D_v' $ represents the viscosity $ \mathcal D_v $ in the $ (x',y',z',t') $-coordinates. 

Moreover, we define the new ``vertical velocity''
\begin{equation}\label{def:new-vertical-v}
	W(\vech x', z', t') := \dfrac{-(z'-1) \partial_{t'}Z - (z'-1) v \cdot \nabla_{\vech{x}'} Z - w }{Z}.
\end{equation}
Then dropping the prime sign yields the following system
\begin{equation}\label{rfeq:isen-CPE-fb}
	\begin{cases}
		\gamma z(\dt Z + v \cdot \nablah Z) + ( \gamma-1) z Z ( \dvh v + \partial_z W) + Z W = 0 & \text{in} ~ \Omega_e,\\
		\rho (\dt v + v \cdot \nablah v + W \partial_z v + g \nablah Z) = \mathcal D_v& \text{in} ~ \Omega_e,\\
		\partial_z Z = 0 & \text{in} ~ \Omega_e,
	\end{cases}
\end{equation}
with $ \Omega_e = \Omega_h \times (0,1) = \mathbb T^2 \times (0,1) $ and $ \rho = \bigl( \frac{\gamma-1}{\gamma}gz Z \bigr)^{1/(\gamma-1)} $ given by \eqref{def:density}. 
Meanwhile, from the definitions in \eqref{new-coordinates} and \eqref{def:new-vertical-v}, the boundary conditions in \eqref{bc-fb} yield
\begin{equation}\label{new-bc-01}
	W|_{z = 0 , 1}=0.
\end{equation}
We recall that $ Z $ is the height of the free boundary and is also proportional to the enthalpy on the ground, from which the pressure (potential) on the ground can be recovered as $$ P_s = \bigl( \dfrac{\gamma-1}{\gamma} g Z\bigr)^{\frac{\gamma}{\gamma-1}}. $$ $ W $ serves as if it is the new ``vertical velocity'' in this formulation. 
The equilibrium state $ (\rho_e,u_e) = (\bar\rho, 0) $ is represented by \begin{equation}\label{equilibrium} (Z_e,v_e,W_e) \equiv (1,0,0) \qquad \text{or} \qquad  ( Z_e, v_e) \equiv (1,0) .\end{equation}

\bigskip

Next, multiplying \subeqref{rfeq:isen-CPE-fb}{1} with $  z^{1/(\gamma-1) - 1} $ leads to 
\begin{equation}\label{eq:25Feb2022-1}
	\gamma z^{\frac{1}{\gamma-1}}(\dt Z + v \cdot \nablah Z) + (\gamma-1) z^{\frac{1}{\gamma-1}} \dvh v Z + (\gamma-1) Z \partial_z(z^{\frac{1}{\gamma-1}} W) = 0.
\end{equation}
Noticing that since $ \partial_z Z = 0 $, integrating the above equation with respect to the vertical variable $ Z $ over the interval $ (0,1) $ and using \eqref{new-bc-01} yields
\begin{equation}\label{new-interface}
	(\gamma-1) \dt Z + \gamma \biggl(\overline{z^{\frac{1}{\gamma-1}} v}\biggr) \cdot \nablah Z + (\gamma-1) \biggl(\overline{ z^{\frac{1}{\gamma-1}} \dvh v}\biggr) Z =0,
\end{equation}
where we have used and will use the notation
\begin{equation}\label{averaging}
	\overline{f} := \int_0^1 f(z) \,dz.
\end{equation}
Next, we eliminate $ \dt Z $ by employing \eqref{eq:25Feb2022-1} and \eqref{new-interface} to obtain
\begin{align*}
	& (\gamma-1) Z \partial_z (z^{\frac{1}{\gamma-1}} W) = \dfrac{\gamma}{\gamma-1} z^{\frac{1}{\gamma-1}} \bigl\lbrack \gamma \bigl(\overline{z^{\frac{1}{\gamma-1}} v}\bigr) \cdot\nablah Z + ( \gamma-1) \bigl(\overline{z^{\frac{1}{\gamma-1}} \dvh v}\bigr) Z \bigr\rbrack \\
	& ~~~~ - \gamma z^{\frac{1}{\gamma-1}} v \cdot\nablah Z - (\gamma-1) z^{\frac{1}{\gamma-1}} \bigl(\dvh v\bigr) Z.
\end{align*}
Therefore, we have the following representation of $ W $, after applying the fundamental theorem of calculus in the $ z $-variable and using \eqref{new-bc-01},
\begin{equation}\label{new-verticle-v}
	\begin{aligned}
		& (\gamma-1) z^{\frac{1}{\gamma-1}} Z W= z^{\frac{\gamma}{\gamma-1}}\bigl\lbrack \gamma \bigl(\overline{z^{\frac{1}{\gamma-1}} v}\bigr) \cdot\nablah Z + ( \gamma-1) \bigl(\overline{z^{\frac{1}{\gamma-1}} \dvh v}\bigr) Z \bigr\rbrack \\
		& ~~ - \gamma \biggl(\int_0^z \xi^{\frac{1}{\gamma-1}}v(\vech{x},\xi,t) \,d\xi\biggr) \cdot\nablah Z \\
		& ~~ - (\gamma-1) Z  \biggl(\int_0^z \xi^{\frac{1}{\gamma-1}}\dvh v(\vech{x},\xi,t) \,d\xi \biggr).
	\end{aligned}
\end{equation}
Hence, the above relation recovers the ``vertical velocity''. 

Next, we make a comment on the derivation of \eqref{rfeq:isen-CPE-fb}. In \cite[Chapter 3]{Washington2005}, the $ \sigma $-coordinate system is introduced to reformulate \eqref{isen-CPE-fb}. Our choice of coordinates \eqref{new-coordinates} shares the same philosophy as the $ \sigma $-coordinates as well as the $ p $-coordinates in \cite{Lions1992}. 
Denote by $ \alpha : = \frac{1}{\gamma-1} $. On the other hand, the viscosity $ \mathcal D_v $ becomes very complicated in the new coordinates and brings additional difficulty to our analysis. To avoid this, we consider a model whose viscosity is derived by linearizing the $ \mathcal D_v $ and taking into account the momentum correction term (see Remark \ref{rm:viscosity}, below). 
Then system \eqref{rfeq:isen-CPE-fb} can be modified as 
\begin{equation}\label{FB-CPE}
	\begin{cases}
		(\alpha + 1) z(\dt Z + v \cdot \nablah Z) + z Z ( \dvh v + \partial_z W) \\
		~~~~ ~~~~ + \alpha Z W = 0 & \text{in} ~ \Omega_e,\\
		\rho (\dt v + v \cdot \nablah v + W \partial_z v + g \nablah Z) = \mu \deltah v + (\mu+\lambda)\nablah \dvh v \\
		~~~~ ~~~~ 
		+ \mu \partial_{zz} v + \mu \bigl(\nablah \log Z\bigr) \cdot \nablah v +(\mu+\lambda) \bigl(\dvh v\bigr) \nablah \log Z & \text{in} ~ \Omega_e,\\
		\partial_z Z = 0 & \text{in} ~ \Omega_e,
	\end{cases}
\end{equation}
with \eqref{def:density}, \eqref{new-interface}, and \eqref{new-verticle-v} being written as
\begin{align}
	& \Omega_e = \Omega_h \times (0,1) , ~~ \rho = \biggl( \dfrac{1}{\alpha + 1} gzZ \biggr)^\alpha =: A_\alpha z^\alpha Z^\alpha ,  {\nonumber}\\
	& \dt Z + (\alpha + 1) \overline{z^\alpha v} \cdot \nablah Z + \overline{z^\alpha \dvh v} Z = 0, \label{eq:movingboundary}  \\
	& z^\alpha Z W = z^{\alpha+1} \bigl( (\alpha+1) \overline{z^\alpha v} \cdot \nablah Z + \overline{z^\alpha \dvh v} Z\bigr)  \nonumber \\
	& ~~~~ ~~~~ - (\alpha+ 1) \biggl( \int_0^z \xi^\alpha v\,d\xi\biggr) \cdot \nablah Z - \biggl( \int_0^z \xi^\alpha \dvh v \,d\xi\biggr) \cdot Z, \label{eq:verticalvelocity}
\end{align} 
where  the viscosity has the form
\begin{equation}\label{def:artificial-viscosity}\begin{gathered} 
\underbrace{\mu \deltah v + (\mu+\lambda)\nablah \dvh v + \mu \partial_{zz} v}_\text{linearized viscosity} \\
+  \underbrace{\mu \bigl( \nablah \log Z \bigr) \cdot \nablah v +(\mu+\lambda) \bigl( \dvh v \bigr) \nablah \log Z}_\text{momentum correction}. \end{gathered}\end{equation}
The following boundary conditions for \eqref{FB-CPE} are imposed:
\begin{equation}\label{FB-BC-CPE}
	W|_{z=0,1} = 0, \qquad \dz v|_{z=0,1} = 0.
\end{equation}

\bigskip

Without loss of generality, we assume $ A_\alpha \equiv 1 $. We will study \eqref{FB-CPE} with initial data $ (Z,v)|_{t=0} = (Z_0, v_0) \in H^2(\Omega_e) \times H^2(\Omega_e) $ satisfying 
\begin{gather}\label{vanish-of-initial-momentum}
	\int_{\Omega_e} \rho v Z\idx \big|_{t=0} = \int_{\Omega_e} \bigl( z^\alpha Z_0^\alpha v_0 Z_0 \bigr) \idx = 0, \\
	\int_{\Omega_e} Z^{\alpha+1} \idx \big|_{t=0} = \int_{\Omega_e} Z_0^{\alpha+1} \idx = 1. \nonumber
\end{gather}
\begin{remark}
Following the derivation of \eqref{rfeq:isen-CPE-fb}, the `real' physical viscosity term and the boundary conditions should be given by applying the change of coordinates to the momentum equations in \eqref{isen-CPE-fb} and \eqref{bc-fb}. In particular, the viscosity $\mathcal D_v $ in the new coordinates is much more complicated than the viscosity \eqref{def:artificial-viscosity}. However, we aim at studying the stability of the steady state given by $ (Z_e, v_e) \equiv (1,0) $. Therefore, the change of coordinates will only bring some terms which are at least quadratic and $ o (1) $ as the perturbation size goes to zero, although these terms may involve higher order derivatives. For this reason, we study here the modified viscosity and boundary conditions as a plausible model, which will substantially simplify the calculations and presentation. The original case will be treated in future study. 
\end{remark}
\begin{remark}\label{rm:viscosity}
The momentum correction term $$ \mu \bigl( \nablah \log Z \bigr) \cdot \nablah v +(\mu+\lambda) \bigl(\dvh v \bigr) \nablah \log Z, $$ is used to recover the physical principle of the conservation of momentum. In fact, since we have simplified the viscosity term and the boundary conditions by dropping the $ o (1) $ terms, the conservative structure of the momentum is violated. We fix this by putting this correction term at the right-hand side of \subeqref{FB-CPE}{2} and we will have the conservation of momentum,
\begin{equation*}
	\dfrac{d}{dt} \intw \rho v Z\idx = 0.
\end{equation*}
Also, from \eqref{eq:movingboundary} we have the conservation of mass,
\begin{equation*}
	\dfrac{d}{dt} \inth Z^{\alpha+1} \idxh = 0.
\end{equation*}
\end{remark}
\bigskip
The rest of this work will be organized as follows. In the next section, we introduce basic notations and frequently used inequalities. Furthermore, we will also state our main theorems. In section \ref{sec:aprior}, we show the global {\it a priori} estimates, which together with the local well-posedness theory in section \ref{sec:local}, show the global well-posedness and stability in section \ref{sec:global-stability}. In section \ref{sec:regularity-1} and section \ref{sec:regularity-2}, we will explore the regularity of the solutions. This will finish the proof of Theorem \ref{thm:global}, below. The asymptotic stability theory is established in section \ref{sec:decay}, which is summarized as Theorem \ref{thm:stability}, below. In section \ref{sec:non-conservative}, we introduce and show the global stability theory for a non-conservative system.

\section{Preliminaries and main theorems}
In section \ref{sec:notations}, below, we introduce some notations and inequalities that will be used in this paper. With the energy and dissipation functionals defined in section \ref{sec:functional}, we will state the main theorems in section \ref{sec:main-thm}. 
\subsection{Notations and Hardy's inequality}\label{sec:notations}
As we stated before, we will use the following notations to denote the differential operators in the horizontal directions:
\begin{gather*}
	\nablah := ( \partial_x, \partial_y)^\top, \qquad  \partial_h \in \lbrace \partial_x, \partial_y \rbrace,\\
	\dvh  :=  \nablah \cdot, \qquad \deltah := \dvh \nablah.
\end{gather*}
We will use $ \hnorm{\cdot}{}, \norm{\cdot}{} $ to denote norms in $ \Omega_h \subset \mathbb R^2 $ and $ \Omega \subset \mathbb R^3 $, respectively. 
Correspondingly, the associated Sobolev norm in the horizontal variables is denoted as,
\begin{equation*}
	\norm{\cdot}{\hHnorm{s}} := \norm{\cdot}{\Lnorm{2}} + \norm{\nablah \cdot}{\Lnorm{2}} + \cdots + \norm{\nablah^s \cdot}{\Lnorm{2}}, ~ \text{for $ s \in \mathbb Z^+ $}.
\end{equation*}
For instance, given $ f \in L^2(\Omega;\mathbb R) = L^2(\Omega_h\times(0,1); \mathbb R) $, denote by
\begin{gather*}
	\hnorm{f}{\Lnorm{2}}(z) := \bigl( \inth \abs{f(\vech x, z)}{2} \idxh \bigr)^{1/2} \qquad \text{and} \qquad \norm{f}{\Lnorm{2}} := \bigl( \int \abs{f}{2} \idx \bigr)^{1/2},
\end{gather*}
where
\begin{gather*}
	\inth \cdot \idxh = \inth \cdot \,dxdy \qquad \text{and} \qquad \int \cdot \idx = \intw \cdot \idx = \intw \cdot \,dxdydz,
\end{gather*}
represent the integrals with respect to the horizontal variables and the full spatial variables, respectively. Then, with the notation defined in \eqref{averaging}, it is easy to verify,
\begin{gather*}
	\hnorm{\overline{f}}{\Hnorm{s}} \leq C \norm{f}{\hHnorm{s}},
\end{gather*}
for some constant $ 0 < C < \infty $ provided the right-hand side is finite. In addition, we have the following property:
	\begin{lemma}
		Consider $ p_i \geq 1, \alpha_i > 0 $ satisfying
		\begin{equation*}
			\dfrac{\alpha_1}{p_1} + \dfrac{\alpha_2}{p_2} + \cdots \dfrac{\alpha_n}{p_n} \leq 1.
		\end{equation*}
		The following H\"older's inequality holds:
		\begin{equation*}
			\int_0^1 \hnorm{f_1}{\Lnorm{p_1}}^{\alpha_1}\hnorm{f_2}{\Lnorm{p_2}}^{\alpha_2} \cdots \hnorm{f_n}{\Lnorm{p_n}}^{\alpha_n} \,dz  \leq C \norm{f_1}{\Lnorm{p_1}}^{\alpha_1}\norm{f_2}{\Lnorm{p_2}}^{\alpha_2} \cdots \norm{f_n}{\Lnorm{p_n}}^{\alpha_n},
		\end{equation*}
		provided the right-hand side is finite.
	\end{lemma}
	After applying the
	Gagliardo-Nirenberg (or Ladyzhenskaya) inequality 
	in $ \Omega_h $ (c.f. \cite{GN-ineq}), one has
	\begin{equation}\label{ineq-supnorm}
	\begin{gathered}
		\hnorm{f}{\Lnorm{4}} \leq C \hnorm{f}{\Lnorm{2}}^{1/2} \hnorm{f}{\Hnorm{1}}^{1/2}, \\
		\hnorm{f}{\infty} \leq C \hnorm{f}{W^{1,4}}\leq C \hnorm{f}{\Hnorm{1}}^{1/2} \hnorm{f}{\Hnorm{2}}^{1/2},
	\end{gathered}
	\end{equation}
	for the function $ f $ with bounded right-hand sides. 

	The standard Hardy inequalities read:	
	\begin{lemma}[Hardy's inequalities] \label{lm:hardy} Let $ k $ and $ L $ be given real and positive numbers, respectively, and let $ g:[0,L] \mapsto \mathbb R $ be a measurable function satisfying $ \int_0^L s^k ({\abs{g(s)}{2}}/{L^2} + \abs{g'(s)}{2})\,ds < \infty $.
	\begin{enumerate}
		\item[a)] If $ k > 1 $, then we have
		\begin{equation*}
			\int_0^L s^{k-2} \abs{g(s)}{2}\,ds \leq C \int_{0}^{L} s^k (\abs{g(s)}{2}/L^2 + \abs{g'(s)}{2})\,ds;
		\end{equation*}
		\item[b)] if $ k < 1 $, then $ g $ has a trace at $ x=0 $ and, moreover,
		\begin{equation*}
			\int_0^L s^{k-2} \abs{g(s)- g(0)}{2}\,ds\leq C\int_0^L s^k \abs{g'(s)}{2} \,ds;
		\end{equation*}
		\item[c)] if $ k> 1 $, then we have, for $ 0<\omega < 1$,
		\begin{align*}
			\int_0^\omega s^{k-2} \abs{g(s)}{2} \,ds 
			\leq C\omega^{-2} \int_0^\omega s^k \abs{g(s)}{2} \,ds + C  \omega^2 \int_0^\omega s^{k-2} \abs{g'(s)}{2}\,ds ,
		\end{align*}
		provided the right-hand side is finite. 
	\end{enumerate}
	Here $ 0 < C < \infty $ is a positive constant independent of $ L $ and $ \omega $.
\end{lemma}
\begin{proof}
We refer $ a), b) $ to \cite{Jang2014} and \cite{Kufner}. $ c) $ is a special case of $ a) $. 
\end{proof}
A straight forward calculation shows that Hardy's inequalities in Lemma \ref{lm:hardy} hold with $ s $ replaced by $ z $ and $ \abs{g(s)}{}, \abs{g'(s)}{} $ replaced by $ \hnorm{f}{\Lnorm{2}}(z), \hnorm{\dz f}{\Lnorm{2}}(z) $, respectively. Additionally, we will apply the special case when $ L = 1 $ throughout this paper. 

For any two quantities $ A, B $, we use $ A \lesssim B $ to denote that there is a scale invariant constant $ 0 < C < \infty $ such that
\begin{equation*}
	A \leq C B. 
\end{equation*}
We use the notation
	$ \mathcal H(\cdot) $
to represent a polynomial of its arguments, which may be different from line to line, satisfying
\begin{equation*}
	\mathcal H(0) = 0.
\end{equation*}


\subsection{Energy and dissipation functionals}\label{sec:functional}
In this section, we introduce the main functionals to be used in this work.
Recalling that $ \alpha = \frac{1}{\gamma-1} $, we introduce first, the following energy functionals, 
\begin{equation}\label{STB-total-energy}
	\begin{aligned}
		& \mathcal E(t) := \norm{z^{\alpha/2} v}{\hHnorm{2}}^2 + \norm{z^{\alpha/2} v_t}{\Lnorm{2}}^2 + \norm{z^{\alpha/2}v_z}{\hHnorm{1}}^2 + \norm{v}{\Hnorm{1}}^2 \\
		& ~~~~ + \norm{Z - 1}{\Hnorm{2}}^2 + \norm{Z_t}{\Lnorm{2}}^2.
	\end{aligned}
\end{equation}
Moreover, we denote the total initial energy as
\begin{equation}\label{STB-ttl-ini-energy}
	\begin{aligned}
		& \mathcal E_0 := \norm{z^{\alpha/2} v_0}{\hHnorm{2}}^2 + \norm{z^{\alpha/2} v_1}{\Lnorm{2}}^2 + \norm{z^{\alpha/2}v_{0,z}}{\hHnorm{1}}^2 + \norm{v_0}{\Hnorm{1}}^2 \\
		& ~~~~ + \norm{Z_0 - 1}{\Hnorm{2}}^2 + \norm{Z_1}{\Lnorm{2}}^2,
	\end{aligned}
\end{equation}
where $v_0 = v|_{t=0}, v_1 = v_t|_{t=0}, Z_1 = Z_t|_{t=0} $ are given by
\begin{align*}
	& z^\alpha Z^\alpha v_1 = \mu \deltah v_0 + (\mu+\lambda) \nablah \dvh v_0 + \mu \partial_{zz} v_0 + \mu \nablah \log Z_0 \cdot \nablah v_0 \\
	& ~~~~ + (\mu+\lambda) \dvh v_0 \nablah \log Z_0 - g z^\alpha Z^\alpha_0 \nablah Z_0 - z^\alpha Z^\alpha v_0 \cdot \nablah v_0 \\
	& ~~~~ - z^{\alpha+1} \bigl( \dfrac{\alpha+1}{\alpha} \overline{z^\alpha v_0} \cdot\nablah Z_0^{\alpha} + \overline{z^\alpha \dvh v_0} Z_0^{\alpha} \bigr)  \dz v_0 \\
	& ~~~~ + \dfrac{\alpha+1}{\alpha} \int_0^z \xi^\alpha v_0 \,d\xi \cdot\nablah Z_0^\alpha \dz v_0 + \int_0^z \xi^\alpha \dvh v_0 \,d\xi Z_0^\alpha \dz v_0, \\
	& Z_1 = - (\alpha+1) \overline{z^\alpha v_0} \cdot \nablah Z_0 - \overline{z^\alpha \dvh v_0} Z_0.
\end{align*}
For some constant $ 0 < c < \infty $, we denote the total relative energy as
\begin{equation}\label{STB-inst-energy}
	\begin{aligned}
		& \mathfrak E_c(t) := \dfrac{1}{2} \int z^\alpha Z^{\alpha+1} \abs{v}{2} \idx + \dfrac{g}{\alpha+2} \int \biggl( \bigl( Z^{\alpha+2} - 1 \bigr) - \dfrac{\alpha+2}{\alpha+1} \bigl( Z^{\alpha+1} - 1 \bigr) \biggr) \idx \\
		& ~~~~ + c \int \biggl( \dfrac{\mu}{2} \abs{\nablah v}{2} + \dfrac{\mu+\lambda}{2} \abs{\dvh v}{2} + \dfrac{\mu}{2}\abs{\dz v}{2} \\
		& ~~~~ ~~~~ - \dfrac{g}{\alpha+1} z^\alpha \bigl( Z^{\alpha+1} - 1\bigr) \dvh v \biggr) \idx + \dfrac{1}{2} \int \biggl( z^\alpha Z^{\alpha+1} \bigl(\abs{v_t}{2}+ \abs{\nablah v}{2} \\
		& ~~~~ ~~~~ + \abs{ \nablah^2 v}{2} + \abs{v_z}{2} + \abs{\nablah v_{z}}{2} \bigr) \biggr) \idx  + \dfrac{g}{2} \int \biggl( \bigl( Z^\alpha  \abs{Z_t}{2} + Z^\alpha \abs{\nablah Z}{2} \\
		& ~~~~ ~~~~ + \abs{\nablah^2 Z}{2}\bigr) \biggr) \idx.
	\end{aligned}
\end{equation}
We observe that if $ \mathcal E(t) $ is sufficiently small, $ \mathfrak E_c(t) $ is equivalent to $ \mathcal E(t) $ for some fixed constant $ c > 0 $:
\begin{proposition}\label{lm:equal-of-energies}
	There is a constant $ \varepsilon_1 \in (0,1) $ small enough, such that if $ \mathcal E(t) < \varepsilon_1 $ for $ t \in [0,T] $, the following inequalities hold for some positive constant $ C < \infty $, depending on $ \varepsilon_1 $ and independent of $ t $, 
	\begin{gather}
		\abs{Z(\vech x, t) - 1}{} \leq C \varepsilon_1^{1/2} < 1/2, \label{aprasm:Z}\\
		C^{-1} \mathcal E(t) \leq \mathfrak E_c(t) \leq C \mathcal E(t), \label{aprasm:energy}
	\end{gather}
	where $$ c := \dfrac{\mu}{8g},$$
	in equation \eqref{STB-inst-energy}.
\end{proposition}
\begin{proof}
	Notice first that, $ C \mathcal E(t) \geq \norm{Z(\cdot, t)-1}{\Lnorm{\infty}}^2 $, which implies \eqref{aprasm:Z}. Furthermore, for $ \varepsilon_1 $ sufficiently small, 
	\begin{equation*}\tag{\ref{BE-002}}
	\begin{aligned}
		& \dfrac{\alpha+2}{4(\alpha+1)^2} \abs{Z^{\alpha+1} - 1}{2} \leq (Z^{\alpha+2} - 1) - \dfrac{\alpha+2}{\alpha+1}(Z^{\alpha+1} - 1)  \\
		& ~~~~ ~~~~ \leq \dfrac{\alpha+2}{(\alpha+1)^2} \abs{Z^{\alpha+1} - 1}{2}. 
	\end{aligned}
	\end{equation*}
	Meanwhile, after applying Young's inequality, one has
	\begin{align*}
		& \abs{\dfrac{g}{\alpha+1} z^\alpha \bigl( Z^{\alpha+1} - 1\bigr) \dvh v}{} \leq \dfrac{\mu}{4} \abs{\nablah v}{2} + \dfrac{g^2}{\mu (\alpha+1)^2}\abs{Z^{\alpha+1}-1}{2}.
	\end{align*}
	Hence for $ 0 < c \leq \frac{\mu}{8g} $, one has
	\begin{align*}
		& C \mathcal E(t)  \geq  \mathfrak E_c(t) \geq \dfrac{1}{2} \int z^\alpha Z^{\alpha+1} \abs{v}{2} \idx + \dfrac{g}{8(\alpha+1)^2} \int \abs{Z^{\alpha+1}-1}{2} \idx \\
		& ~~~~ + c\int \dfrac{\mu}{4} \abs{\nabla v}{2} \idx + \dfrac{1}{2} \int \biggl( z^\alpha Z^{\alpha+1} \bigl(\abs{v_t}{2}+ \abs{\nablah v}{2} \\
		& ~~~~ ~~~~ + \abs{ \nablah^2 v}{2} + \abs{v_z}{2} + \abs{\nablah v_{z}}{2} \bigr) \biggr) \idx  + \dfrac{g}{2} \int \biggl( \bigl( Z^\alpha  \abs{Z_t}{2} + Z^\alpha \abs{\nablah Z}{2} \\
		& ~~~~ ~~~~ + \abs{\nablah^2 Z}{2}\bigr)\biggr) \idx \geq C^{-1} \mathcal E(t),
	\end{align*}
	for some positive constant $  1 < C < \infty $. Here we have applied Hardy's inequality, Lemma \ref{lm:hardy}, and, thanks to \eqref{aprasm:Z}, the fact that $ \abs{Z-1}{} < 1/2 $ 
	as follows, 
	\begin{align}\label{STB-001}
		& \norm{v}{\Lnorm{2}}^2 = \int_0^1 \hnorm{v}{\Lnorm{2}}^2 \,dz \lesssim \int_0^1 z^2 (\hnorm{v}{\Lnorm{2}}^2 + \hnorm{\dz v}{\Lnorm{2}}^2 ) \,dz  \lesssim \int_0^1 z^{4} \hnorm{v}{\Lnorm{2}}^2 \nonumber \\
		& ~~~~ + \hnorm{\dz v}{\Lnorm{2}}^2\,dz \underbrace{\lesssim \cdots \lesssim}_{\max\lbrace[\alpha/2]-1,0\rbrace\text{ times}} \norm{z^{\alpha/2} v}{\Lnorm{2}}^2 + \norm{\dz v}{\Lnorm{2}}^2, \\
		& \abs{Z-1}{} = \abs{(Z^{\alpha+1})^{1/(\alpha+1)} - 1}{} \lesssim \abs{Z^{\alpha+1}-1}{} \lesssim \abs{Z - 1}{}. \label{STB-002}
	\end{align}
	This finishes the proof. 
\end{proof}
In addition, we introduce the 
relative dissipation functional,
\begin{equation}\label{STB-inst-dissipation}
	\begin{aligned}
		& \mathfrak D(t) : = \dfrac{\mu}{2} \int Z \bigl( \abs{\nabla v}{2} + \abs{\nabla v_t}{2} + \abs{\nabla^2 v}{2} + \abs{\nabla^2 \nablah v}{2}  \bigr) \idx\\
		& ~~~~ + \int z^\alpha Z^\alpha \abs{v_t}{2}  \idx
		+ \int \bigl( \abs{Z_t}{2} + \abs{\nablah Z}{2} + \abs{\nablah^2 Z}{2} \bigr) \idx.
	\end{aligned}
\end{equation}
We have the following lemma: 
\begin{proposition}\label{lm:equal-of-dissipation}
	In addition to the assumptions in Proposition \ref{lm:equal-of-energies}, suppose that $ (Z,v) $ is a solution to \eqref{FB-CPE}. Then there is a constant $ 0 < C < \infty $ such that,
	\begin{equation}\label{aprasm:dissipation}
		 \begin{aligned}
		 & C^{-1} \bigl( \norm{Z-1}{\Hnorm{2}}+ \norm{Z_t}{\Lnorm{2}}^2  + \norm{\nabla v}{\Hnorm{1}}^2 + \norm{\nablah v}{\Hnorm{2}}^2 \\
		 & ~~~~ + \norm{z^{\alpha/2} v_t}{\Lnorm{2}}^2 + \norm{\nabla v_t}{\Lnorm{2}}^2 \bigr) \leq \mathfrak D(t) \leq C \bigl( \norm{Z-1}{\Hnorm{2}}+ \norm{Z_t}{\Lnorm{2}}^2 \\
		 & ~~~~ ~~~~ + \norm{\nabla v}{\Hnorm{1}}^2 + \norm{\nablah v}{\Hnorm{2}}^2 + \norm{z^{\alpha/2} v_t}{\Lnorm{2}}^2 + \norm{\nabla v_t}{\Lnorm{2}}^2 \bigr).
		 \end{aligned}
	\end{equation} 
\end{proposition}
\begin{proof}
	As the consequence of \eqref{STB-002} and \eqref{STB-inst-dissipation}, it suffices to show 
	\begin{equation*}
		\norm{Z^{\alpha+1} - 1}{\Lnorm{2}} \lesssim \norm{\nablah Z}{\Lnorm{2}}.
	\end{equation*}
	But this is a direct consequence of the conservation of mass $ \int Z^{\alpha+1} \idx = \int Z_0^{\alpha+1} \idx = 1 $ (see \eqref{apriori:conserv-mass} in Proposition \ref{lm:basic-energy}, below) , the Poincar\'e inequality, and \eqref{aprasm:Z}.
\end{proof}


\subsection{Main theorems}\label{sec:main-thm}
Now we can state the main theorems in this contribution. 
\begin{theorem}[Global Stability]\label{thm:global} Recalling $ \alpha = \frac{1}{\gamma-1} $, consider initial data $ (Z,v)|_{t=0} = (Z_0, v_0) \in H^2(\Omega) \times H^2(\Omega) $ satisfying the compatible conditions,
\begin{gather*}
	\dz v_{0}|_{z=0,1} = 0, ~ \dz Z_0 = 0, ~ \int Z_0^{\alpha+1} \idx = 1, ~ W_0|_{z=0,1} = 0,
\end{gather*}
where $ W_0 $ is defined in \eqref{eq:verticalvelocity} by replacing $ v $ and $ Z $ with $ v_0 $ and $ Z_0 $, respectively. 
Suppose that $ \mathcal E_0 < \varepsilon' $ for some positive constant $ \varepsilon' $ given in Proposition \ref{prop:global-stability}, below. Then: 
{\par\noindent\bf I. }
There is a unique strong global solution $ (Z,v) $ to the simplified model equations \eqref{FB-CPE}, subject to the boundary condition \eqref{FB-BC-CPE}, for the free boundary problem of the compressible primitive equations. The global strong solution admits the following regularity:
\begin{gather*}
	Z - 1 \in L^\infty(0,\infty;H^2(\Omega))\cap L^2(0,\infty;H^2(\Omega)), \\
	\dt Z \in L^\infty(0,\infty;L^2(\Omega))\cap L^2(0,\infty;L^2(\Omega)), \\
	z^{\alpha/2} v, z^{\alpha/2} \nablah v, z^{\alpha/2}\nablah^2 v, z^{\alpha/2} \dz v, z^{\alpha/2} \nablah \dz v \in L^\infty(0,\infty;L^2(\Omega)), \\
	v \in L^\infty(0,\infty;H^2(\Omega)), \nabla v \in L^2(0,\infty;H^1(\Omega)), \nablah v \in L^\infty(0,\infty;H^2(\Omega)),\\
	z^{\alpha/2} \dt v\in L^\infty(0,\infty;L^2(\Omega)) \cap L^2(0,\infty;L^2(\Omega)), 
	\nabla v_t\in L^2(0,\infty;L^2(\Omega)),\\
	z^{1/2} \partial_{zzz} v \in L^2(0,\infty;L^2(\Omega)).
\end{gather*}
Moreover, for every $ T \in (0,\infty) $, the following estimates hold for the solution $ (Z,v) $,
\begin{align*}
	& \bullet \sup_{0\leq t\leq T} \mathcal E(t) + \int_0^T \biggl( \norm{Z-1}{\Hnorm{2}}^2+ \norm{Z_t}{\Lnorm{2}}^2  + \norm{\nabla v}{\Hnorm{1}}^2 + \norm{\nablah v}{\Hnorm{2}}^2 \\
	& ~~~~ + \norm{z^{\alpha/2} v_t}{\Lnorm{2}}^2 + \norm{\nabla v_t}{\Lnorm{2}}^2 \biggr) \,dt \leq C \mathcal E_0; \\
	& \bullet \sup_{0\leq t \leq T} \norm{v(t)}{\Hnorm{2}} \leq C \bigl( \mathcal E_0 + 1 \bigr) \mathcal E_0^{1/2}; \\
	& \bullet \text{if $ \alpha > 1/2 $, i.e., } 1 < \gamma < 3,  ~ \int_0^T \norm{\partial_{zzz} v}{\Lnorm{2}}^2 \,dt \leq C \mathcal E_0 + C \mathcal E_0^3; \\
	& \bullet \text{if $ 0 < \alpha\leq 1/2 $, i.e., } \gamma \geq 3 , ~ \int_0^T \norm{z^{\beta}\partial_{zzz} v}{\Lnorm{2}}^2 \,dt \leq C \mathcal E_0 + C \mathcal E_0^3, \\
	& ~~~~  \text{with $ \beta > \frac{1}{2} - \alpha $},
\end{align*}
for some constant $ 0 < C < \infty $. In particular,
		\begin{align*}
			&  \int_0^T \norm{z^{1/2}\partial_{zzz} v}{\Lnorm{2}}^2 \,dt \leq C \mathcal E_0 + C \mathcal E_0^3. 
		\end{align*}
{\par\noindent\bf II. }
If in addition, $ (Z_0, v_0) $ satisfies
\begin{equation}\label{high-regular-initial}
	\norm{Z_0}{\Hnorm{3}} + \norm{z^{\alpha/2}\nablah^3 v_0}{\Lnorm{2}} < \infty,
\end{equation}
then 
for any $  T \in (0, \infty) $, there is a positive constant $ 0 < C_T < \infty $ such that,
	\begin{align*}
	& \sup_{0\leq t\leq T} \bigl\lbrace \norm{z^{\alpha/2} \nablah^3 v}{\Lnorm{2}}^2 + \norm{\nablah^3 Z}{\Lnorm{2}}^2 +  \norm{\dt Z}{\Lnorm{\infty}}^2 +  \norm{\nablah Z}{\Lnorm{\infty}}^2 \bigr\rbrace \\
	& ~~~~ + \int_0^T \norm{\nabla\nablah^3 v}{\Lnorm{2}}^2\,dt \leq C_T \bigl( \norm{z^{\alpha/2} \nablah^3 v_0}{\Lnorm{2}}^2 + \norm{\nablah^3 Z_0}{\Lnorm{2}}^2 + 1 \bigr).
	\end{align*}
\end{theorem}
\begin{proof}[Proof of Theorem \ref{thm:global}]
	The existence and uniqueness of the strong solutions follow from Proposition \ref{prop:global-stability}, below. The estimates and regularity are given in Proposition \ref{prop:stability-theory}, Proposition \ref{prop:H^2-of-v}, and Proposition \ref{lm:diss-dzzz-v}, below. The estimate with initial data satisfying \eqref{high-regular-initial} is given in Proposition \ref{prop:Higher-order}, below.
\end{proof}

\begin{theorem}[Asymptotic Stability]\label{thm:stability}
Consider $ (Z_0, v_0) \in H^2(\Omega) \times H^2(\Omega) $. 
	Suppose that in addition to the conditions given in Theorem \ref{thm:global}, $ 0 < \alpha < 3 $ (i.e., $ \gamma > \frac{4}{3} $)  and \eqref{vanish-of-initial-momentum} holds for the initial data, i.e.,
	\begin{equation*}\tag{\ref{vanish-of-initial-momentum}}
		\int \bigl( z^\alpha Z_0^{\alpha+1} v_0 \bigr) \idx = 0.
	\end{equation*}
	Then, there exist constants $ 0 < \varepsilon'' < \varepsilon' , 0 < C_1, C_2 < \infty $, such that if $ \mathcal E_0 < \varepsilon'' $, we have the following decay estimate: for all $ t > 0 $,
	\begin{equation*}
		\mathcal E(t) \leq e^{-C_1t}C_2 \mathcal E_0.
	\end{equation*}
	In particular, 
	\begin{equation*}
		 \abs{Z(\vech x, t)-1}{} + \abs{v(\vec x, t)}{} \leq \norm{Z(t)-1}{\Hnorm{2}} + \norm{v(t)}{\Hnorm{2}} \leq C e^{-Ct},
	\end{equation*}
	for some constant $ 0 < C < \infty $. 
\end{theorem}

\begin{proof}[Proof of Theorem \ref{thm:stability}]
	This is the consequence of Proposition \ref{prop:decay-est} and Proposition \ref{prop:H^2-of-v}, below. 
\end{proof}

\begin{remark}
	The condition \eqref{vanish-of-initial-momentum} is only required in showing the asymptotic stability, which can be reached via a Galilean transformation. 
	The study of a non-conservative system, in which case such an assumption does not make sense, is discussed in section \ref{sec:non-conservative}.

	With more regular initial data satisfying \eqref{high-regular-initial}, Theorem \ref{thm:global} shows that the solution is more regular in the tangential direction. In particular, this will yield the change of coordinates \eqref{new-coordinates} is regular, and the reformulated equations \eqref{rfeq:isen-CPE-fb} are equivalent to the original equations \eqref{isen-CPE-fb}. 	
\end{remark}

\section{{\it \textbf{A priori}} estimates}\label{sec:aprior}
In this section, we derive the required global {\it a priori} estimates which guarantee the stability of the equilibrium. We assume that the strong solution $ (Z, v) $ to \eqref{FB-CPE} exists in $ \Omega \times (0,T) $ for some positive time $ T \in (0,\infty) $ and that $ (Z, v) $ is regular enough so that the quantities in the following relevant statements are finite and well-defined. This is the essential block of the continuity arguments later in this paper. 
Indeed, we will establish the following proposition:
\begin{proposition}\label{prop:stability-theory}
Suppose that $ (Z,v) $ is a strong solution to \eqref{FB-CPE} subject to the boundary condition \eqref{FB-BC-CPE}. 
There is a constant $ \varepsilon_2 \in (0,\varepsilon_1) $, with $ \varepsilon_1 $ given in Proposition \ref{lm:equal-of-energies}, such that if $ \mathcal E(t) < \varepsilon_2 $, for $ t \in (0,T) $, the following inequality holds:
\begin{equation}\label{apriori:total-energy}
	\begin{aligned}
		& \sup_{0<t<T} \mathcal E(t) + \int_0^T \biggl( \norm{Z-1}{\Hnorm{2}}^2+ \norm{Z_t}{\Lnorm{2}}^2  + \norm{\nabla v}{\Hnorm{1}}^2 + \norm{\nablah v}{\Hnorm{2}}^2 \\
		& ~~~~ + \norm{z^{\alpha/2} v_t}{\Lnorm{2}}^2 + \norm{\nabla v_t}{\Lnorm{2}}^2 \biggr) \,dt \leq C_1 \mathcal E_0,
	 \end{aligned}
\end{equation}
	where $ 0 < C_1 < \infty $ is a constant depending only on $ \varepsilon_2 $. 
\end{proposition}
Throughout the rest of this section, without loss of generality,
we assume that
\begin{equation}\label{apriori-boundness-of-Z}
	\dfrac{1}{2} < Z (\vech x, t) < 2,
\end{equation}
for all $ \vech x \in \Omega_h $ and $ t \in [0,T] $.

\subsection{Basic energy}
In this section, we will perform some $ L^2 $-energy estimates on \eqref{FB-CPE}. Our first lemma is concerning the basic energy of the system:
\begin{proposition}\label{lm:basic-energy}
	Let $ (Z, v) $ be a regular enough strong solution on $ [0,T] $ to \eqref{FB-CPE}. Then one has the following, for all $ t \in [0,T] $:
	\begin{align}
		& \bullet \text{Conservation of mass:} ~~~~ \int Z^{\alpha+1}(\vech x, t)\idx = \int Z_0^{\alpha+1}(\vech x, t) \idx; \label{apriori:conserv-mass}\\
		& \bullet \text{Conservation of energy:} \label{apriori:conserv-kinetic-energy} \nonumber \\
		&  \dfrac{1}{2} \int z^\alpha Z^{\alpha+1}(\vech x, t) \abs{v(\vec x, t)}{2}  \idx + \dfrac{g}{\alpha+2} \int \biggl( \bigl( Z^{\alpha+2}(\vech x, t) - 1 \bigr) \nonumber \\
		& ~~~~ ~~~~ ~~~~ - \dfrac{\alpha+2}{\alpha+1} \bigl( Z^{\alpha+1}(\vech x, t)-1\bigr) \biggr) \idx \nonumber  \\
		& ~~~~ ~~~~ + \int_0^T \int Z\bigl( \mu \abs{\nablah v}{2} + (\mu+\lambda) \abs{\dvh v}{2} + \mu \abs{\dz v}{2} \bigr) \idx \,dt  \\ 
		& =  \dfrac{1}{2} \int z^\alpha Z_0^{\alpha+1} \abs{v_0}{2}  \idx  + \dfrac{g}{\alpha+2} \int \biggl( \bigl( Z_0^{\alpha+2} - 1 \bigr) - \dfrac{\alpha+2}{\alpha+1} \bigl( Z_0^{\alpha+1}-1\bigr) \biggr) \idx.  \nonumber
	\end{align}
	Furthermore, the following inequalities hold for some constant $ 0 < C < \infty $:
	\begin{align}
		& \int_0^T \norm{Z_t}{\Lnorm{2}}^2 \,dt \leq C \int_0^T \norm{\nablah v}{\Lnorm{2}}^2 \,dt \nonumber \\
		& ~~~~ ~~~~ ~~~~ + C \int_0^T \norm{z^{\alpha/2} v}{\hHnorm{1}}^2 \norm{\nablah Z}{\Hnorm{1}}^2 \,dt, \label{apriori:dissipation-Z-t} \\
		& \int_0^T \norm{Z_t}{\Lnorm{3}}^3 \,dt \leq C \int_0^T \bigl( \norm{z^{\alpha/2} v}{\hHnorm{1}}^3 \norm{\nablah Z}{\Hnorm{1}}^3 + \norm{z^{\alpha/2}\nablah v}{\hHnorm{1}}^{3} \bigr) \,dt.  \label{apriori:dissipation-Z-t-L3}
	\end{align}

\end{proposition}
\begin{proof}
	Taking the $ L^2 $-inner produce of \subeqref{FB-CPE}{2} with $ Z v $ yields, after applying integration by parts in the resultant,
	\begin{equation}\label{BE-001}
		\begin{aligned}
			& \dfrac{d}{dt} \biggl\lbrace \dfrac{1}{2} \int z^\alpha Z^{\alpha+1} \abs{v}{2} \idx \biggr\rbrace + \int z^\alpha Z^{\alpha+1} g \nablah Z \cdot v \idx \\
			& ~~~~ + \int Z \bigl\lbrack  \mu \abs{\nablah v}{2}  + ( \mu+ \lambda)  \abs{\dvh v}{2} + \mu \abs{\dz v}{2} \bigr\rbrack \idx \\
			&  = 0,
		\end{aligned}
	\end{equation}
	where we have applied the fact that from \subeqref{FB-CPE}{1}, one has
	\begin{equation}\label{eq:density}
		\dt (z^\alpha Z^{\alpha+1} ) + \dvh (z^\alpha Z^{\alpha+1} v) + \dz ( z^\alpha Z^{\alpha+1} W) = 0.
	\end{equation}
	On the other hand, after applying integration by parts and \eqref{eq:density}, one will have
	\begin{align*}
		& \int z^\alpha Z^{\alpha+1} g \nablah Z \cdot v \idx = - g \int  Z \dvh ( z^{\alpha} Z^{\alpha+1} v) \idx  \\
		& = g \int \bigl\lbrack \dt (z^\alpha Z^{\alpha+1} ) Z + \dz (z^\alpha Z^{\alpha+1} W) Z \bigr\rbrack \idx = \dfrac{d}{dt} \biggl\lbrace  \dfrac{(\alpha+1)g}{\alpha+2} \int z^{\alpha} Z^{\alpha+2} \idx \biggr\rbrace\\
		& = \dfrac{d}{dt} \biggl\lbrace \dfrac{g}{\alpha+2} \int Z^{\alpha+2} \idx \biggr\rbrace,
	\end{align*}
	thanks to the fact that $ \dz Z = 0 $. 
	In addition, integrating \eqref{eq:movingboundary} leads to, 
	\begin{equation}\label{eq:25Feb2022-2}
		\dfrac{d}{dt} \int Z^{\alpha+1} \idx = 0.
	\end{equation}
	Thus we have
	\begin{equation}\label{eq:000001}
		\int z^\alpha Z^{\alpha+1} g \nablah Z \cdot v \idx = \dfrac{d}{dt} \biggl\lbrace \dfrac{g}{\alpha+2} \int \bigl\lbrack (Z^{\alpha+2} - 1) - \dfrac{\alpha+2}{\alpha+1}(Z^{\alpha+1} - 1) \bigr\rbrack \idx \biggr\rbrace.
	\end{equation}
	Moreover, observe that when $ |Z - 1| $ is small enough, one has
	\begin{equation}\label{BE-002}
	\begin{aligned}
		& \dfrac{\alpha+2}{4(\alpha+1)^2} \abs{Z^{\alpha+1} - 1}{2} \leq (Z^{\alpha+2} - 1) - \dfrac{\alpha+2}{\alpha+1}(Z^{\alpha+1} - 1)  \\
		& ~~~~ ~~~~ \leq \dfrac{\alpha+2}{(\alpha+1)^2} \abs{Z^{\alpha+1} - 1}{2}  
	\end{aligned}
	\end{equation}
	Therefore from \eqref{BE-001} and \eqref{eq:000001}, we have
	\begin{equation}\label{apriori-ineq-001}
	\begin{aligned}
		& \dfrac{d}{dt} \biggl\lbrace \dfrac{1}{2} \int z^\alpha Z^{\alpha+1} \abs{v}{2} \idx  +  \dfrac{g}{\alpha+2} \int \bigl\lbrack (Z^{\alpha+2} - 1) - \dfrac{\alpha+2}{\alpha+1}(Z^{\alpha+1} - 1) \bigr\rbrack \idx \biggr\rbrace \\
		& ~~~~ + \int Z \bigl(  \mu \abs{\nablah v}{2}  + ( \mu+ \lambda)  \abs{\dvh v}{2} + \mu \abs{\dz v}{2} \bigr) \idx = 0.
	\end{aligned}\end{equation}
	Integrating \eqref{eq:25Feb2022-2} and \eqref{apriori-ineq-001} with respect to time yields the conservations of mass and energy. 
	Moreover, from \eqref{eq:movingboundary}, we have
	\begin{align}
		& \norm{Z_t}{\Lnorm{2}} \lesssim \norm{z^{\alpha/2}v}{\hHnorm{1}} \norm{\nablah Z}{\Hnorm{1}} + \norm{z^{\alpha/2} \nablah v}{\Lnorm{2}}, {\nonumber} \\
		& \norm{Z_t}{\Lnorm{p}} \lesssim \hnorm{\overline{z^{\alpha/2} v}}{\Hnorm{1}} \hnorm{\nablah Z}{\Hnorm{1}} + \hnorm{\overline{z^{\alpha/2} \nablah v}}{\Lnorm{p}} \nonumber \\
		& ~~~~ \lesssim \hnorm{\overline{z^{\alpha/2} v}}{\Hnorm{1}} \hnorm{\nablah Z}{\Hnorm{1}} + \hnorm{\overline{z^{\alpha/2} \nablah v}}{\Lnorm{2}}^{2/p} \hnorm{\overline{z^{\alpha/2} \nablah v}}{\Hnorm{1}}^{1-2/p} \nonumber \\
		& ~~~~ \lesssim \norm{{z^{\alpha/2} v}}{\hHnorm{1}} \norm{\nablah Z}{\Hnorm{1}} + \norm{{z^{\alpha/2} \nablah v}}{\Lnorm{2}}^{2/p} \norm{{z^{\alpha/2} \nablah v}}{\hHnorm{1}}^{1-2/p},\label{BE-004}
	\end{align}
	for $ 2 < p < \infty $. Taking $ p = 3 $ and integrating the corresponding quantities with respect to time will complete the proof of the proposition. 
\end{proof}
Next lemma is concerning the estimate for $ \nabla v $: 
\begin{proposition}\label{lm:H^1-v-est}
	Under the same assumptions as in Proposition \ref{lm:basic-energy}, one has that
	\begin{equation}\label{apriori:H^1-v-est}
		\begin{aligned}
		& \int \bigl\lbrack \dfrac{\mu}{2} \abs{\nablah v}{2} + \dfrac{\mu+\lambda}{2} \abs{\dvh v}{2} + \dfrac{\mu}{2} \abs{\dz v}{2} \bigr\rbrack \idx \\
		& ~~~~ - \dfrac{1}{\alpha+1} \int g z^\alpha (Z^{\alpha+1}-1) \dvh v \idx + \int_0^T \int z^\alpha Z^\alpha \abs{v_t}{2} \idx \,dt \\
		& \leq \int \dfrac{\mu}{2} \abs{\nablah v_0}{2} + \dfrac{\mu+\lambda}{2} \abs{\dvh v_0}{2} + \dfrac{\mu}{2} \abs{\dz v_0}{2} \idx \\
		& ~~~~ - \dfrac{1}{\alpha+1} \int g z^\alpha (Z_0^{\alpha+1}-1) \dvh v_0 \idx  + C \int_0^T \norm{Z_t}{\Lnorm{2}}^2 \\
		& ~~~~ + \norm{\nablah v}{\Lnorm{2}}^2 \,dt  + C \int_0^T \bigl( (\norm{\nablah Z}{\Hnorm{1}} + 1) \norm{z^{\alpha/2}v}{\hHnorm{2}} \\
		& ~~~~ ~~~~ + \norm{v}{\Hnorm{1}} + \norm{\nablah Z}{\Hnorm{1}} \bigr) 
		 \bigl(\norm{\dz v}{\hHnorm{1}} + \norm{\nablah v}{\Hnorm{1}}\bigr)  \\
		& ~~~~ ~~~~ \times\bigl( \norm{z^{\alpha/2} v_t}{\Lnorm{2}} + \norm{\nabla v_t}{\Lnorm{2}} \bigr) \,dt,
		\end{aligned}
	\end{equation}
	for some constant $ 0 < C < \infty $. 
\end{proposition}

\begin{proof}
	Taking the $L^2$-inner product of \subeqref{FB-CPE}{2} with $ v_t $ implies, after applying integrations by parts,
	\begin{equation}\label{BE-003}
		\begin{aligned}
			& \dfrac{d}{dt} \biggl\lbrace \int \dfrac{\mu}{2} \abs{\nablah v}{2} + \dfrac{\mu+\lambda}{2} \abs{\dvh v}{2} + \dfrac{\mu}{2} \abs{\dz v}{2} \idx \biggr\rbrace + \int z^\alpha Z^\alpha \abs{v_t}{2} \idx  \\
			& ~~~~ - \dfrac{1}{\alpha+1} \int g z^\alpha Z^{\alpha+1} \dvh v_t \idx = - \int \biggl( z^\alpha Z^\alpha v\cdot\nablah v \cdot v_t \\
			& ~~~~ - \mu \nablah \log Z \cdot \nablah v \cdot v_t - (\mu+\lambda) \dvh v \nablah \log Z \cdot v_t \biggr) \idx \\
			& ~~~~ - \int z^\alpha Z^\alpha W \dz v \cdot v_t \idx =: I_1 + I_2 .
		\end{aligned}
	\end{equation}
	Also, one has
	\begin{align*}
		& - \dfrac{1}{\alpha+1} \int g z^\alpha Z^{\alpha+1} \dvh v_t \idx = \dfrac{d}{dt} \biggl\lbrace -  \dfrac{1}{\alpha+1} \int g z^\alpha (Z^{\alpha+1}-1) \dvh v \idx \biggr\rbrace\\
		& ~~~~ + \underbrace{\int g z^\alpha Z^\alpha Z_t \dvh v\idx}_{-I_3}.
	\end{align*}
	
	We estimate $ I_2 $ first. After substituting \eqref{eq:verticalvelocity}, one has
	\begin{align*}
		& I_2 = - \int  z^{\alpha+1} Z^{\alpha-1} \bigl( (\alpha+1) \overline{z^\alpha v} \cdot \nablah Z + \overline{z^\alpha\dvh v} Z\bigr) \dz v \cdot v_t \idx\\
		& ~~~~ +\int \biggl( Z^{\alpha-1} \bigl( (\alpha+1)  \int_0^z \xi^\alpha v \,d\xi \cdot\nablah Z + \int_0^z \xi^\alpha \dvh v \,d\xi Z \bigr) \dz v \cdot v_t \biggr) \idx\\
		& ~~~~ =: I_2' + I_2''.
	\end{align*}
	Notice that $ Z, \overline{z^\alpha v}, \overline{z^\alpha\dvh v} $ are independent of the $ z $-variable. Therefore after applying H\"older's, Minkowski's, the Sobolev embedding inequalities, and \eqref{ineq-supnorm}, one has that
	\begin{align*}
		& I_2' \lesssim \int_0^1 (\hnorm{\overline{z^\alpha v}}{\Lnorm{\infty}}  \hnorm{\nablah Z}{\Lnorm{4}} + \hnorm{\overline{z^\alpha \nablah v}}{\Lnorm{4}}) \hnorm{z^{\alpha/2+1} \dz v}{\Lnorm{4}} \hnorm{ z^{\alpha/2} v_t}{\Lnorm{2}} \,dz\\
		& ~~~~ \lesssim (\norm{z^\alpha v}{\hHnorm{2}} \norm{\nablah Z}{\Lnorm{2}}^{1/2} \norm{\nablah Z}{\Hnorm{1}}^{1/2} + \norm{z^\alpha\nablah v}{\hHnorm{1}} )\\
		& ~~~~ ~~ \times \int_0^1 \hnorm{z^{\alpha/2} \dz v}{\Lnorm{2}}^{1/2} \hnorm{z^{\alpha/2}\dz v}{\Hnorm{1}}^{1/2}\hnorm{z^{\alpha/2}v_t}{\Lnorm{2}}\,dz \lesssim \bigl( \norm{z^{\alpha/2} v}{\hHnorm{2}} \norm{\nablah Z}{\Hnorm{1}} \\
		& ~~~~ ~~ + \norm{z^{\alpha/2} \nablah v}{\hHnorm{1}}\bigr) \norm{z^{\alpha/2}\dz v}{\hHnorm{1}} \norm{z^{\alpha/2} v_t}{\Lnorm{2}},\\
		& I_2'' \leq \int_0^1 \bigl\lbrack \hnorm{\nablah Z}{\Lnorm{4}} \int_0^z \xi^\alpha \hnorm{v}{\infty} \,d\xi + \int_0^z \xi^\alpha \hnorm{\nablah v}{\Lnorm{4}} \,d\xi \bigr\rbrack \hnorm{\dz v}{\Lnorm{4}} \hnorm{v_t}{\Lnorm{2}} \,dz\\
		& ~~~~ \lesssim  \int_0^1 \bigl(\hnorm{\nablah Z}{\Lnorm{4}} (\int_0^z \xi^{\alpha/2} \hnorm{v}{\infty}^2 \,d\xi)^{1/2} + (\int_0^z \xi^{\alpha/2} \hnorm{\nablah v}{\Lnorm{2}}^2 \,d\xi)^{1/4}\\
		& ~~~~ ~~~~ \times (\int_0^z \xi^{\alpha/2}  \hnorm{\nablah v}{\Hnorm{1}}^2\,d\xi)^{1/4}   \bigr) z^{\alpha/2+1/2} \hnorm{\dz v}{\Lnorm{2}}^{1/2} \hnorm{\dz v}{\Hnorm{1}}^{1/2} \hnorm{v_t}{\Lnorm{2}}\,dz  \\
		& ~~~~ \lesssim (\norm{\nablah Z}{\Hnorm{1}} + 1) \norm{z^{\alpha/2}v}{\hHnorm{2}} \norm{\dz v}{\hHnorm{1}} \norm{z^{\alpha/2}v_t}{\Lnorm{2}}. 
	\end{align*}
	Therefore, we arrive at the inequality 
	$$ I_2 \lesssim \bigl(\norm{\nablah Z}{\Hnorm{1}} + 1\bigr) \norm{z^{\alpha/2}v}{\hHnorm{2}} \norm{\dz v}{\hHnorm{1}} \norm{z^{\alpha/2}v_t}{\Lnorm{2}}. $$ 
	On the other hand, employing similar arguments, one can obtain the following estimates for $ I_1, I_3 $:
	\begin{align*}
		& I_1 \lesssim \norm{z^{\alpha/2}v}{\Lnorm{3}} \norm{\nablah v}{\Lnorm{6}}\norm{z^{\alpha/2}v_t}{\Lnorm{2}} + \norm{\nablah Z}{\Lnorm{6}} \norm{\nablah v }{\Lnorm{3}} \norm{v_t}{\Lnorm{2}}\\
		& ~~~~ \lesssim 
		\norm{v}{\Hnorm{1}}
		\norm{\nablah v}{\Hnorm{1}}\norm{z^{\alpha/2}v_t}{\Lnorm{2}}
		+ \norm{\nablah Z}{\Hnorm{1}} \norm{\nablah v}{\Hnorm{1}} \bigl( \norm{z^{\alpha/2} v_t}{\Lnorm{2}} \\
		& ~~~~ + \norm{\nabla v_t}{\Lnorm{2}} \bigr) ,\\
		& I_3 \lesssim \norm{Z_t}{\Lnorm{2}} \norm{\nablah v}{\Lnorm{2}},
	\end{align*}
	where we have applied, above, Hardy's inequality as follows:
	\begin{equation}\label{TE-002}
	\begin{aligned}
		& \norm{v_t}{\Lnorm{2}}^2 = \int_0^1 \hnorm{v_t}{\Lnorm{2}}^2 \,dz \lesssim \int_0^1 z^2 (\hnorm{v_t}{\Lnorm{2}}^2 + \hnorm{\dz v_t}{\Lnorm{2}}^2 ) \,dz  \lesssim \int_0^1 z^{4} \hnorm{v_t}{\Lnorm{2}}^2 \\
		& ~~~~ + \hnorm{\dz v_t}{\Lnorm{2}}^2\,dz \underbrace{\lesssim \cdots \lesssim}_{\max\lbrace[\alpha/2]-1,0\rbrace\text{ times}} \norm{z^{\alpha/2} v_t}{\Lnorm{2}}^2 + \norm{\dz v_t}{\Lnorm{2}}^2.
	\end{aligned}
	\end{equation}
	Therefore, from \eqref{BE-003}, after summing up the inequalities above, we have
	\begin{equation}\label{apriori-ineq-002}
	\begin{aligned}
		& \dfrac{d}{dt} \biggl\lbrace \int \bigl\lbrack \dfrac{\mu}{2} \abs{\nablah v}{2} + \dfrac{\mu+\lambda}{2} \abs{\dvh v}{2} + \dfrac{\mu}{2} \abs{\dz v}{2} \bigr\rbrack \idx \\
		& ~~~~ - \dfrac{1}{\alpha+1} \int g z^\alpha (Z^{\alpha+1}-1) \dvh v \idx \biggr\rbrace + \int z^\alpha Z^\alpha \abs{v_t}{2} \idx\\
		& ~~~~ \lesssim \norm{Z_t}{\Lnorm{2}}^2 + \norm{\nablah v}{\Lnorm{2}}^2  + \bigl( (\norm{\nablah Z}{\Hnorm{1}} + 1) \norm{z^{\alpha/2}v}{\hHnorm{2}} \\
		& ~~~~ ~~~~ + \norm{v}{\Hnorm{1}} + \norm{\nablah Z}{\Hnorm{1}} \bigr) 
		 \bigl(\norm{\dz v}{\hHnorm{1}} + \norm{\nablah v}{\Hnorm{1}}\bigr)  \\
		& ~~~~ ~~~~ \times\bigl( \norm{z^{\alpha/2} v_t}{\Lnorm{2}} + \norm{\nabla v_t}{\Lnorm{2}} \bigr) .
	\end{aligned}
	\end{equation}
	Integrating the above inequality with respect to the time variable will conclude the proof of the proposition. 
\end{proof}

\subsection{The temporal and horizontal derivative estimates}
This section includes the temporal derivative estimate and the horizontal (tangential) derivative estimates.

After applying $ \dt $ to \subeqref{FB-CPE}{2}, \eqref{eq:movingboundary} and \eqref{eq:verticalvelocity}, the following equations hold
\begin{equation}\label{eq:dt-FBCPE}
	\begin{cases}
		z^\alpha Z^\alpha (\dt v_t + v \cdot \nablah v_t + W \dz v_t) + z^\alpha Z^\alpha ( v_t \cdot\nablah v + W_t \dz v) \\
		~~~~ ~~~~ + \alpha z^\alpha Z^{\alpha-1} Z_t ( v_t + v\cdot\nablah v + W\dz v) + g (z^\alpha Z^\alpha \nablah Z)_t \\
		~~~~  = \mu \deltah v_t + (\mu+\lambda) \nablah \dvh v_t + \mu \partial_{zz} v_t  + \mu \nablah \log Z \cdot \nablah v_t \\
		~~~~ ~~~~ + (\mu+\lambda) \dvh v_t \nablah \log Z + \mu \nablah (\log Z)_t \cdot \nablah v \\
		~~~~ ~~~~ + (\mu+\lambda) \dvh v \nablah (\log Z)_t , \\
		\dt Z_t + (\alpha+1) \overline{z^\alpha v} \cdot \nablah Z_t + \overline{z^\alpha \dvh v} Z_t + (\alpha+1) \overline{z^\alpha v_t} \cdot\nablah Z \\
		~~~~ ~~~~ + \overline{z^\alpha \dvh v_t} Z = 0 ,\\ 
		z^\alpha W_t = z^{\alpha+1} \bigl( (\alpha+1) \overline{z^\alpha v} \cdot \nablah (\log Z)_t + (\alpha + 1) \overline{z^\alpha v_t} \cdot \nablah \log Z \\
		~~~~ ~~~~ + \overline{z^\alpha \dvh v_t} \bigr) - (\alpha+1) \int_0^z \xi^\alpha v \,d\xi \cdot\nablah (\log Z)_t \\
		~~~~ ~~~~ - (\alpha + 1) \int_0^z \xi^\alpha v_t \,d\xi \cdot \nablah \log Z - \int_0^z \xi^\alpha \dvh v_t \,d\xi, 
		\\
		\dz Z_t = 0.
	\end{cases}
\end{equation}

Now we derive the temporal derivative estimate:
\begin{proposition}\label{lm:L^2-v-t-est}
Under the same assumptions as in Proposition \ref{lm:basic-energy}, one has that 
\begin{equation}\label{apriori:temporal-derivative}
	\begin{aligned}
		& \dfrac{1}{2} \int z^\alpha Z^{\alpha+1} \abs{v_t}{2} \idx + \dfrac{g}{2} \int Z^\alpha \abs{Z_t}{2}\idx  \\
		& ~~~~ + \int_0^T \int Z\bigl( \mu \abs{\nablah v_t}{2} + (\mu+\lambda) \abs{\dvh v_t}{2} + \mu \abs{\dz v_t}{2} \bigr) \idx \,dt \\
		& \leq  \dfrac{1}{2} \int z^\alpha Z_0^{\alpha+1} \abs{v_1}{2} \idx + \dfrac{g}{2} \int Z_0^\alpha \abs{Z_1}{2}\idx + \int_0^T \mathcal H(\norm{\nablah Z}{\Hnorm{1}}, \\
		& ~~~~ ~~~~ \norm{Z_t}{\Lnorm{2}}, \norm{z^{\alpha/2} v}{\hHnorm{2}}, \norm{v}{\Hnorm{1}}, \norm{z^{\alpha/2}v_t}{\Lnorm{2}} )\\
		& ~~~~ \times\bigl( \norm{z^{\alpha/2} v_t}{\Lnorm{2}}^2 + \norm{\nabla v_t}{\Lnorm{2}}^2 + \norm{\dz v}{\hHnorm{2}}^2 + \norm{\nablah v}{\Hnorm{2}}^2 \\
		& ~~~~ + \norm{\nablah Z}{\Hnorm{1}}^2 \bigr) 
		 + C \norm{Z_t}{\Lnorm{3}}^3 \,dt,
	\end{aligned}
\end{equation}
	for some constant $ 0 < C < \infty $. Recall that $ \mathcal H (\cdot) $ is a polynomial of its arguments with $ \mathcal H(0) = 0 $. 
\end{proposition}
\begin{proof}
	Taking the $ L^2 $-inner product of \subeqref{eq:dt-FBCPE}{1} with $ Z v_t $ yields, after applying integration by parts,
	\begin{equation}\label{TE-001}
		\begin{aligned}
			& \dfrac{d}{dt} \biggl\lbrace \dfrac{1}{2} \int z^\alpha Z^{\alpha+1} \abs{v_t}{2} \idx \biggr\rbrace + \int g (z^\alpha Z^\alpha \nablah Z)_t \cdot Zv_t \idx \\
			& ~~~~ + \int Z\bigl( \mu \abs{\nablah v_t}{2} + (\mu+\lambda) \abs{\dvh v_t}{2} + \mu \abs{\dz v_t}{2} \bigr) \idx \\
			& =  \int \mu Z \nablah(\log Z)_t \cdot \nablah v \cdot v_t + (\mu+\lambda) Z v_t \cdot \nablah(\log Z)_t \dvh v \idx\\
			& ~~~~ - \int z^\alpha Z^{\alpha+1} v_t \cdot \nablah v \cdot v_t + \alpha z^\alpha Z^\alpha Z_t ( v_t + v\cdot \nablah v) \cdot v_t \idx \\
			& ~~~~ - \int z^\alpha Z^{\alpha+1} W_t \dz v \cdot v_t \idx - \alpha \int z^\alpha Z^\alpha Z_t W \dz v \cdot v_t \idx\\
			& =: I_{4}+ I_5 + I_6 + I_7.
		\end{aligned}
	\end{equation}
	At the same time, after substituting \subeqref{eq:dt-FBCPE}{2} and noticing that $ Z $ is independent of the $ z $-variable, one has,
	\begin{align*}
		& \int g (z^\alpha Z^\alpha \nablah Z)_t \cdot Zv_t \idx = \dfrac{g}{\alpha+1} \int z^\alpha ( \nablah Z^{\alpha +1})_t \cdot Z v_t \idx \\
		&  = - g \int  Z^{\alpha+1} Z_t  \overline{z^\alpha\dvh v_t} +  Z^\alpha Z_t \overline{z^\alpha v_t} \cdot \nablah Z \idx \\
		&  = g \int Z^{\alpha} Z_t \bigl( \dt Z_t +(\alpha+1) \overline{z^\alpha v} \cdot \nablah Z_t + \overline{z^\alpha \dvh v} Z_t + \alpha \overline{z^\alpha v_t} \cdot \nablah Z  \bigr) \idx \\
		& = \dfrac{d}{dt} \biggl\lbrace \dfrac{g}{2} \int Z^\alpha \abs{Z_t}{2}\idx \biggr\rbrace - \dfrac{\alpha g}{2} \int Z^{\alpha-1} Z_t \abs{Z_t}{2} \idx \\
		& ~~~~ + \alpha g \int Z^\alpha Z_t \overline{z^\alpha v_t} \cdot \nablah Z \idx + \dfrac{ (1- \alpha) g}{2} \int Z^\alpha \overline{z^\alpha \dvh v} \abs{Z_t}{2} \idx \\
		& ~~~~  - \dfrac{(\alpha+1)\alpha g}{2} \int Z^{\alpha-1} \abs{Z_t}{2} \overline{z^\alpha v} \cdot\nablah Z \idx =: \dfrac{d}{dt} \biggl\lbrace \dfrac{g}{2} \int Z^\alpha \abs{Z_t}{2}\idx \biggr\rbrace \\
		& ~~~~ - I_{8}.
	\end{align*}
	Similarly as before, we estimate the terms involving the vertical velocity $ W $ first. In fact, we have from \subeqref{eq:dt-FBCPE}{3} that
	\begin{align*}
		& I_6 = - (\alpha+1) \int Z^{\alpha+1} (z^{\alpha+1} \overline{z^\alpha v} - \int_0^z\xi^\alpha v \,d\xi) \cdot \nablah (\log Z)_t   \dz v \cdot v_t \idx \\
		& ~~~~ - (\alpha + 1) \int Z^{\alpha+1} (z^{\alpha+1} \overline{z^\alpha v_t} - \int_0^z \xi^\alpha v_t \,d\xi) \cdot \nablah \log Z \dz v \cdot v_t \idx \\
		& ~~~~ - \int Z^{\alpha+1} (z^{\alpha+1} \overline{z^\alpha \dvh v_t} - \int_0^z \xi^\alpha \dvh v_t \,d\xi) \dz v \cdot v_t \idx\\
		&  =: I_6' + I_6'' + I_6'''.
	\end{align*}
	Applying integration by parts yields
	\begin{align*}
		& I_6' = (\alpha+1) \int Z^{\alpha+1} (z^{\alpha+1} \overline{z^\alpha \dvh v} - \int_0^z \xi^\alpha \dvh v \,d\xi) (\log Z)_t \dz v\cdot v_t \idx \\
		& ~~~~ + (\alpha+1)^2 \int Z^\alpha (\log Z)_t (z^{\alpha+1} \overline{z^\alpha v} - \int_0^z \xi^\alpha v\,d\xi) \cdot \nablah Z \dz v \cdot v_t \idx \\
		& ~~~~ + (\alpha+1) \int Z^{\alpha+1} (\log Z)_t ( z^{\alpha+1} \overline{z^\alpha v} - \int_0^z \xi^\alpha v \,d\xi ) \cdot \nablah (\dz v \cdot v_t) \idx\\
		& = : I_{6,1}' + I_{6,2}' + I_{6,3}'.
	\end{align*}
	Now we apply H\"older's, Minkowski's, and the Sobolev embedding inequalities,
	\begin{align*}
		& I_{6,1}' \lesssim \int_0^1 \hnorm{Z_t}{\Lnorm{2}} \hnorm{\dz v}{\Lnorm{8}} z^{\alpha/2}\hnorm{v_t}{\Lnorm{4}}\bigl(\hnorm{\overline{z^\alpha\nablah v}}{\Lnorm{8}}   + (\int_0^z \xi^{\alpha} \hnorm{\nablah v}{\Lnorm{8}}^2 \,d\xi)^{1/2} \bigr) \,dz \\
		& ~~~~ \lesssim   \int_0^1 \hnorm{\dz v}{\Hnorm{1}} \hnorm{z^{\alpha/2} v_t}{\Lnorm{2}}^{1/2} \hnorm{z^{\alpha/2} v_t}{\Hnorm{1}}^{1/2} \,dz \norm{Z_t}{\Lnorm{2}} \norm{{z^{\alpha/2}\nablah v}}{\hHnorm{1}}\\
		& ~~~~ \lesssim \norm{Z_t}{\Lnorm{2}} \norm{z^{\alpha/2} v}{\hHnorm{2}}   \norm{\dz v}{\hHnorm{1}}\bigl( \norm{z^{\alpha/2}v_t}{\Lnorm{2}} + \norm{\nablah v_t}{\Lnorm{2}} \bigr),\\
		& I_{6,2}' \lesssim \int_0^1 \hnorm{Z_t}{\Lnorm{2}} \hnorm{\nablah Z}{\Lnorm{16}} \hnorm{\dz v}{\Lnorm{8}} z^{\alpha/2}\hnorm{v_t}{\Lnorm{4}} \bigl( \hnorm{\overline{z^\alpha v}}{\Lnorm{16}} \\
		& ~~~~ + (\int_0^z \xi^\alpha \hnorm{v}{\Lnorm{16}}\,d\xi )^{1/2} \bigr) \,dz \lesssim \int_0^1 \hnorm{\dz v}{\Hnorm{1}} \hnorm{z^{\alpha/2} v_t}{\Hnorm{1}} \,dz \norm{Z_t}{\Lnorm{2}}\\
		& ~~~~ \times \norm{\nablah Z}{\Hnorm{1}} \norm{z^{\alpha/2} v}{\hHnorm{1}} \lesssim \norm{Z_t}{\Lnorm{2}} \norm{\nablah Z}{\Hnorm{1}} \norm{z^{\alpha/2}v}{\hHnorm{1}} \norm{\dz v}{\hHnorm{1}}\\
		& ~~~~ \times \bigl( \norm{z^{\alpha/2} v_t}{\Lnorm{2}} + \norm{\nablah v_t}{\Lnorm{2}} \bigr),\\
		& I_{6,3}' \lesssim \int_0^1 \hnorm{Z_t}{\Lnorm{2}} \bigl( \hnorm{\dz v}{\Lnorm{\infty}} z^{\alpha/2} \hnorm{\nablah v_t}{\Lnorm{2}} + \hnorm{\nablah \dz v}{\Lnorm{4}} z^{\alpha/2} \hnorm{v_t}{\Lnorm{4}} \bigr)\\
		& ~~~~ \times \bigl( \hnorm{\overline{z^\alpha v}}{\Lnorm{\infty}} + (\int_0^z \xi^\alpha \hnorm{v}{\Lnorm{\infty}}^2 \,d\xi)^{1/2} \bigr) \,dz \lesssim \int_0^1 \hnorm{\dz v}{\Hnorm{2}} \hnorm{z^{\alpha/2} v_t}{\Hnorm{1}} \,dz \\
		& ~~~~ \times \norm{Z_t}{\Lnorm{2}} \norm{z^{\alpha/2}v}{\hHnorm{2}}\lesssim \norm{Z_t}{\Lnorm{2}}\norm{z^{\alpha/2}v}{\hHnorm{2}} \norm{\dz v}{\hHnorm{2}} \bigl( \norm{z^{\alpha/2}v_t}{\Lnorm{2}} \\
		& ~~~~ + \norm{\nablah v_t}{\Lnorm{2}}\bigr).
	\end{align*}
	Similarly,
	\begin{align*}
		& I_6'' \lesssim \int_0^1 \hnorm{\nablah Z}{\Lnorm{4}} \hnorm{\dz v}{\Lnorm{4}} z^{\alpha/2} \hnorm{v_t}{\Lnorm{4}} \bigl( \hnorm{\overline{z^\alpha v_t}}{\Lnorm{4}} + (\int_0^z \xi^\alpha\hnorm{v_t}{\Lnorm{4}}^2\,d\xi)^{1/2} \bigr) \,dz\\
		& ~~~~ \lesssim \norm{\nablah Z}{\Hnorm{1}}\norm{z^{\alpha/2}v_t}{\Lnorm{2}} \norm{\dz v}{\hHnorm{1}} \bigl( \norm{z^{\alpha/2}v_t}{\Lnorm{2}} + \norm{\nablah v_t}{\Lnorm{2}} \bigr),\\
		& I_6''' \lesssim \int_0^1 \hnorm{\dz v}{\Lnorm{\infty}}z^{\alpha/2} \hnorm{v_t}{\Lnorm{2}} \bigl( \hnorm{\overline{z^\alpha\nablah v_t}}{\Lnorm{2}} + (\int_0^z \xi^\alpha \hnorm{\nablah v_t}{\Lnorm{2}}^2 \,d\xi)^{1/2}\,dz\\
		& ~~~~ \lesssim \norm{z^{\alpha/2}v_t}{\Lnorm{2}} \norm{\dz v}{\hHnorm{2}} \norm{\nablah v_t}{\Lnorm{2}}.
	\end{align*}
	At the same time, from \eqref{eq:verticalvelocity}, we have
	\begin{align*}
		& I_7 = - \alpha (\alpha+1) \int Z^{\alpha-1} Z_t \bigl( z^{\alpha+1} \overline{z^\alpha v} -  \int_0^z \xi^\alpha v\,d\xi  \bigr)\cdot\nablah Z \dz v \cdot v_t \idx \\
		& ~~~~ - \alpha \int Z^{\alpha} Z_t \bigl( z^{\alpha+1} \overline{z^\alpha\dvh v} - \int_0^z \xi^\alpha\dvh v\,d\xi\bigr)\dz v \cdot v_t \idx =: I_7' + I_7'', 
	\end{align*}
	with
	\begin{align*}
		& I_7' \lesssim \int_0^1 \hnorm{Z_t}{\Lnorm{2}} \hnorm{\nablah Z}{\Lnorm{8}}\hnorm{\dz v}{\Lnorm{8}} z^{\alpha/2} \hnorm{v_t}{\Lnorm{4}} \bigl( \hnorm{\overline{z^\alpha v}}{\Lnorm{\infty}} \\
		& ~~~~ + (\int_0^z \xi^\alpha \hnorm{v}{\Lnorm{\infty}}^2 \,d\xi)^{1/2} \bigr) \,dz \lesssim \norm{Z_t}{\Lnorm{2}}\norm{\nablah Z}{\Hnorm{1}} \norm{z^{\alpha/2} v}{\hHnorm{2}} \\
		& ~~~~ \times \norm{\dz v}{\hHnorm{1}} \bigl( \norm{z^{\alpha/2} v_t}{\Lnorm{2}} + \norm{\nablah v_t}{\Lnorm{2}} \bigr), \\
		& I_7'' \lesssim \int_0^1 \hnorm{Z_t}{\Lnorm{2}}\hnorm{\dz v}{\Lnorm{8}} z^{\alpha/2} \hnorm{v_t}{\Lnorm{4}} \bigl( \hnorm{\overline{z^\alpha \nablah v}}{\Lnorm{8}} \\
		& ~~~~ + (\int_0^z \xi^\alpha \hnorm{\nablah v}{\Lnorm{8}}^2 \,d\xi)^{1/2} \bigr) \,dz \lesssim \norm{Z_t}{\Lnorm{2}} \norm{z^{\alpha/2} v}{\hHnorm{2}} \norm{\dz v}{\hHnorm{1}} \\
		& ~~~~ \times \bigl( \norm{z^{\alpha/2} v_t}{\Lnorm{2}} + \norm{\nablah v_t}{\Lnorm{2}} \bigr).
	\end{align*}
	Therefore, we have
	\begin{align*}
		& I_6 + I_7 \lesssim \mathcal H(\norm{\nablah Z}{\Hnorm{1}},\norm{Z_t}{\Lnorm{2}}, \norm{z^{\alpha/2}v}{\hHnorm{2}}, \norm{z^{\alpha/2} v_t}{\Lnorm{2}} ) \norm{\dz v}{\hHnorm{2}} \\
		& ~~~~ \times \bigl( \norm{z^{\alpha/2}v_t}{\Lnorm{2}} + \norm{\nablah v_t}{\Lnorm{2}} \bigr). 
	\end{align*}
	We list the estimates for the rest in the following:
	\begin{align*}
		& I_4 = - \int \mu (\log Z)_t \nablah Z \cdot \nablah v \cdot v_t + (\mu+\lambda) (\log Z)_t v_t \cdot \nablah Z \dvh v \idx \\
		& ~~~~ - \int \mu Z (\log Z)_t \deltah v \cdot v_t + (\mu+\lambda) Z (\log Z)_t v_t \cdot \nablah \dvh v \idx \\
		& ~~~~ - \int \mu Z (\log Z)_t \nablah v : \nablah v_t + (\mu+\lambda) Z (\log Z)_t \dvh v_t \dvh v\idx \\
		& ~~~~ \lesssim \int_0^1 \hnorm{Z_t}{\Lnorm{2}} \bigl( \hnorm{\nablah Z}{\Lnorm{4}} \hnorm{\nablah v}{\Lnorm{8}} \hnorm{v_t}{\Lnorm{8}} + \hnorm{\nablah^2 v}{\Lnorm{4}} \hnorm{v_t}{\Lnorm{4}} \\
		& ~~~~+ \hnorm{\nablah v}{\Lnorm{\infty}} \hnorm{\nablah v_t}{\Lnorm{2}}  \bigr) \,dz 
		\lesssim \bigl( \norm{\nablah Z}{\Hnorm{1}} + 1\bigr) \norm{Z_t}{\Lnorm{2}} \norm{\nablah v}{\hHnorm{2}}\\
		& ~~~~ \times \bigl( \norm{z^{\alpha/2}v_t}{\Lnorm{2}} + \norm{\nablah v_t}{\Lnorm{2}} \bigr)
		,\\
		& I_5 \lesssim \norm{z^{\alpha/2} v_t}{\Lnorm{2}}\norm{v_t}{\Lnorm{3}} \norm{\nablah v}{\Lnorm{6}} + \norm{Z_t}{\Lnorm{2}} \norm{v_t}{\Lnorm{3}} \norm{v_t}{\Lnorm{6}}\\
		& ~~~~ + \norm{Z_t}{\Lnorm{2}} \norm{v}{\Lnorm{6}}\norm{\nablah v}{\Lnorm{6}} \norm{v_t}{\Lnorm{6}} \lesssim \norm{z^{\alpha/2} v_t}{\Lnorm{2}} \norm{\nablah v}{\Hnorm{1}} \norm{v_t}{\Lnorm{2}}^{1/2} \\
		& ~~~~ \times \norm{v_t}{\Hnorm{1}}^{1/2}+ \norm{Z_t}{\Lnorm{2}} \norm{v_t}{\Lnorm{2}}^{1/2} \norm{v_t}{\Hnorm{1}}^{3/2} + \norm{Z_t}{\Lnorm{2}}\norm{v}{\Hnorm{1}} \norm{\nablah v}{\Hnorm{1}} \\
		& ~~~~ \times  \norm{v_t}{\Hnorm{1}} \lesssim \mathcal H( \norm{Z_t}{\Lnorm{2}}, \norm{z^{\alpha/2} v_t}{\Lnorm{2}}, \norm{v}{\Hnorm{1}}) \bigl(\norm{\nablah v}{\Hnorm{1}}^2 \\
		& ~~~~ + \norm{\nabla v_t}{\Lnorm{2}}^2 + \norm{z^{\alpha/2} v_t}{\Lnorm{2}}^2 \bigr), \\
		& I_8 \lesssim \norm{Z_t}{\Lnorm{3}}^3 + \norm{\nablah Z}{\Hnorm{1}}^{3/2} \norm{z^{\alpha/2} v_t}{\Lnorm{2}}^{3/2} + \norm{z^{\alpha/2} \nablah v}{\hHnorm{1}}^3 \\
		& ~~~~ + \norm{\nablah Z}{\Hnorm{1}}^3 \norm{z^{\alpha/2} v}{\hHnorm{2}}^3,
	\end{align*}
	where we have applied the inequality \eqref{TE-002}. 
	After summing up the inequalities above from \eqref{TE-001}, one has the following inequality:
	\begin{equation}\label{apriori-ineq-003}
	\begin{aligned}
		& \dfrac{d}{dt} \biggl\lbrace \dfrac{1}{2} \int z^\alpha Z^{\alpha+1} \abs{v_t}{2} \idx + \dfrac{g}{2} \int Z^\alpha \abs{Z_t}{2}\idx \biggr\rbrace \\
		& ~~~~ + \int Z\bigl( \mu \abs{\nablah v_t}{2} + (\mu+\lambda) \abs{\dvh v_t}{2} + \mu \abs{\dz v_t}{2} \bigr) \idx\\
		&  \lesssim \mathcal H(\norm{\nablah Z}{\Hnorm{1}}, \norm{Z_t}{\Lnorm{2}}, \norm{z^{\alpha/2} v}{\hHnorm{2}}, \norm{v}{\Hnorm{1}}, \norm{z^{\alpha/2}v_t}{\Lnorm{2}} )\\
		& ~~~~ \times\bigl( \norm{z^{\alpha/2} v_t}{\Lnorm{2}}^2 + \norm{\nabla v_t}{\Lnorm{2}}^2 + \norm{\dz v}{\hHnorm{2}}^2 + \norm{\nablah v}{\Hnorm{2}}^2 \\
		& ~~~~ + \norm{\nablah Z}{\Hnorm{1}}^2 \bigr) 
		 + \norm{Z_t}{\Lnorm{3}}^3.
	\end{aligned}
\end{equation} 
	Integrating this inequality in the temporal variable will finish the proof.
\end{proof}

In the next proposition, we establish estimates for the horizontal derivatives. After applying $ \partial_h $ to \subeqref{FB-CPE}{2}, \eqref{eq:movingboundary}, \eqref{eq:verticalvelocity}, we record the resultant equations in the following:
\begin{equation}\label{eq:dh-FBCPE}
	\begin{cases}
		z^\alpha Z^\alpha (\dt v_h + v \cdot \nablah v_h + W\dz v_h ) + z^\alpha Z^\alpha (v_h \cdot \nablah v + W_h \dz v) \\
		~~~~ ~~~~ + \alpha z^\alpha Z^{\alpha-1} Z_h (v_t + v\cdot \nablah v + W\dz v) + g(z^\alpha Z^\alpha \nablah Z)_h \\
		~~~~ = \mu \deltah v_h +(\mu+\lambda) \nablah \dvh v_h + \mu \partial_{zz} v_h  + \mu \nablah \log Z \cdot \nablah v_h \\
		~~~~ ~~~~ + (\mu+\lambda) \dvh v_h \nablah \log Z + \mu \nablah (\log Z)_h \cdot \nablah v \\
		~~~~ ~~~~ + (\mu+\lambda) \dvh v \nablah (\log Z)_h  , \\
		\dt Z_h + (\alpha+1)\overline{z^\alpha v} \cdot \nablah Z_h + \overline{z^\alpha \dvh v} Z_h + (\alpha+1) \overline{z^\alpha v_h} \cdot \nablah Z\\
		~~~~ ~~~~ + \overline{z^\alpha\dvh v_h} Z = 0, \\
		z^\alpha W_h = z^{\alpha+1}\bigl( (\alpha+1) \overline{z^\alpha v} \cdot \nablah (\log Z)_h + (\alpha+1) \overline{z^\alpha v_h} \cdot \nablah \log Z \\
		~~~~ ~~~~ + \overline{z^\alpha \dvh v_h} \bigr) - (\alpha+1) \bigl( \int_0^z \xi^\alpha v\,d\xi \bigr) \cdot\nablah (\log Z)_h\\
		~~~~ ~~~~ - (\alpha+1) \bigl( \int_0^z \xi^\alpha v_h \,d\xi \bigr) \cdot\nablah \log Z - \int_0^z \xi^\alpha \dvh v_h \,d\xi.
	\end{cases}
\end{equation}

\begin{proposition}\label{lm:L^2-v-dhh-est}
Under the same assumptions as in Proposition \ref{lm:basic-energy}, one has
	\begin{align}
		&  \dfrac{1}{2} \int z^\alpha Z^{\alpha+1} \abs{v_h}{2} \idx + \dfrac{g}{2}\int Z^\alpha \abs{Z_h}{2}\idx \nonumber \\
		& ~~~~ + \int_0^T \int Z \bigl( \mu \abs{\nablah v_h}{2} + (\mu+\lambda)\abs{\dvh v_h}{2} + \mu \abs{\dz v_h}{2} \bigr) \idx \,dt \nonumber \\
		& \leq \dfrac{1}{2} \int z^\alpha Z_0^{\alpha+1} \abs{v_{0,h}}{2} \idx + \dfrac{g}{2}\int Z_0^\alpha \abs{Z_{0,h}}{2}\idx + \int_0^T \mathcal H( \norm{\nablah Z}{\Hnorm{1}}, \nonumber \\
		& ~~~~ ~~~~ \norm{Z_t}{\Lnorm{2}}, \norm{z^{\alpha/2} v}{\hHnorm{2}}, \norm{v}{\Hnorm{1}}, \norm{z^{\alpha/2}v_t}{\Lnorm{2}} ) \bigl( \norm{\dz v}{\hHnorm{2}}^2 \nonumber  \\
		& ~~~~ + \norm{\nablah v}{\Hnorm{1}}^2  + \norm{\nablah Z}{\Hnorm{1}}^2\bigr) \,dt, 
		\label{apriori:tangential-01} \\
		\intertext{and}
		& \dfrac{1}{2} \int z^\alpha Z^{\alpha+1} \abs{v_{hh}}{2} \idx + \dfrac{g}{2} \int \abs{Z_{hh}}{2} \idx   \nonumber \\
		& ~~~~ + \int_0^T \int Z \bigl( \mu \abs{\nablah v_{hh}}{2} + (\mu+\lambda) \abs{\dvh v_{hh}}{2} + \mu \abs{\dz v_{hh}}{2} \bigr)\idx\,dt   \nonumber \\
		& \leq  \dfrac{1}{2} \int z^\alpha Z_0^{\alpha+1} \abs{v_{0,hh}}{2} \idx + \dfrac{g}{2} \int \abs{Z_{0,hh}}{2} \idx + \int_0^T  \mathcal H (\norm{Z-1}{\Hnorm{2}}, \nonumber \\
		& ~~~~ ~~~~ \norm{z^{\alpha/2} v}{\hHnorm{2}}, \norm{v}{\Hnorm{1}} )  \bigl( \norm{\nablah Z}{\Hnorm{1}}^2 + \norm{\nabla v}{\Lnorm{2}}^2 
		 + \norm{\nablah v}{\Hnorm{2}}^2 \nonumber \\
		 & ~~~~ + \norm{z^{\alpha/2} v_t}{\Lnorm{2}}^2 
		  + \norm{\nabla v_t}{\Lnorm{2}}^2 \bigr) \,dt 
		\label{apriori:tangential-02} 
	\end{align}
	Recall that $ \mathcal H (\cdot) $ is a polynomial of its arguments with $ \mathcal H(0) = 0 $. 
\end{proposition}
\begin{proof}
	After taking inner product of \subeqref{eq:dh-FBCPE}{1} with $ Z v_{h} $, one has after applying integration by parts in the resultant, 
	\begin{equation}\label{TE-003}
		\begin{aligned}
			& \dfrac{d}{dt} \biggl\lbrace \dfrac{1}{2} \int z^\alpha Z^{\alpha+1} \abs{v_h}{2} \idx + \dfrac{g}{2}\int Z^\alpha \abs{Z_h}{2}\idx \biggr\rbrace\\
			& ~~~~ + \int \biggl( Z \bigl( \mu \abs{\nablah v_h}{2} + (\mu+\lambda)\abs{\dvh v_h}{2} + \mu \abs{\dz v_h}{2} \bigr) \biggr) \idx \\
			& = \int \biggl( \mu Z \nablah(\log Z)_h \cdot \nablah v \cdot v_h + (\mu+\lambda) Z v_h \cdot \nablah(\log Z)_h \dvh v \biggr) \idx\\
			& ~~~~ - \int \biggl( z^\alpha Z^{\alpha+1} v_h \cdot \nablah v \cdot v_h + \alpha z^\alpha Z^\alpha Z_h( v_t + v\cdot\nablah v) \cdot v_h \biggr) \idx \\
			& ~~~~ - \int z^\alpha Z^{\alpha+1} W_h \dz v \cdot v_h \idx - \alpha \int z^\alpha Z^\alpha Z_h W \dz v \cdot v_h \idx \\
			& ~~~~ + \dfrac{\alpha g}{2} \int Z^{\alpha-1} Z_t \abs{Z_h}{2} \idx  - \alpha g \int Z^\alpha Z_h \overline{z^\alpha v_h} \cdot \nablah Z \idx \\
			& ~~~~ - \dfrac{(1-\alpha)g}{2} \int Z^\alpha \abs{Z_h}{2} \overline{z^\alpha\dvh v} \idx \\
			& ~~~~ + \dfrac{(\alpha+1)\alpha g}{2} \int Z^{\alpha-1} \abs{Z_h}{2} \overline{z^\alpha v}\cdot \nablah Z \idx =: I_9 + I_{10} + I_{11}\\
			& ~~~~ ~~~~ + I_{12} + I_{13} + I_{14} + I_{15} + I_{16}.
		\end{aligned}
	\end{equation}
	As before, we shall estimate the integrals involving the vertical velocity first. 
	Substituting \eqref{eq:verticalvelocity} and \subeqref{eq:dh-FBCPE}{3}  yields
	\begin{align*}
		& I_{11} = - (\alpha+1) \int Z^{\alpha+1}\bigl(z^{\alpha+1} \overline{z^\alpha v} - \int_0^z \xi^\alpha v\,d\xi \bigr) \cdot \nablah (\log Z)_h \dz v \cdot v_h \idx \\
		& ~~~~ - (\alpha+1) \int Z^{\alpha+1} \bigl( z^{\alpha+1} \overline{z^\alpha v_h} - \int_0^z \xi^\alpha v_h\,d\xi\bigr) \cdot \nablah \log Z \dz v \cdot v_h \idx \\
		& ~~~~ - \int Z^{\alpha+1} \bigl( z^{\alpha+1} \overline{z^\alpha\dvh v_h} - \int_0^z \xi^\alpha \dvh v_h \,d\xi \bigr) \dz v \cdot v_h \idx,\\
		& I_{12} = - \alpha (\alpha+1) \int Z^{\alpha-1} Z_h \bigl( z^{\alpha+1} \overline{z^\alpha v} - \int_0^z \xi^\alpha v \,d\xi\bigr) \cdot \nablah Z \dz v\cdot v_h \idx\\
		& ~~~~ - \alpha \int Z^{\alpha} Z_h \bigl( z^{\alpha+1} \overline{z^\alpha \dvh v} - \int_0^z \xi^\alpha \dvh v\,d\xi \bigr) \dz v \cdot v_h \idx.
	\end{align*}
	Then after replacing the subscript $ t $ with $ h $, the estimates for $ I_{11}, I_{12} $ are the same as those for $ I_{6} $ and $ I_{7} $, above. Thus we have
	\begin{equation*}
		I_{11} + I_{12} \lesssim \mathcal H (\norm{\nablah Z}{\Hnorm{1}}, \norm{z^{\alpha/2} v}{\hHnorm{2}}) \norm{\dz v}{\hHnorm{2}} \norm{v_h}{\Hnorm{1}}.
	\end{equation*}
	The rest of the terms can be handled in the same way. That is, it holds that
	\begin{align*}
		& I_9 + I_{10}+ I_{13}+ I_{14}+ I_{15}+ I_{16} \\
		& ~~~~ \lesssim \mathcal H( \norm{\nablah Z}{\Hnorm{1}}, \norm{Z_t}{\Lnorm{2}}, \norm{z^{\alpha/2} v}{\hHnorm{2}}, \norm{v}{\Hnorm{1}} , \norm{z^{\alpha/2}v_t}{\Lnorm{2}} )\\
		& ~~~~ \times \bigl( \norm{\nablah v}{\Hnorm{1}}^2 + \norm{\nablah Z}{\Hnorm{1}}^2\bigr).
	\end{align*}
	Summing up the inequalities above yields 
	\begin{equation}\label{apriori-ineq-004}
	\begin{aligned}
		& \dfrac{d}{dt} \biggl\lbrace \dfrac{1}{2} \int z^\alpha Z^{\alpha+1} \abs{v_h}{2} \idx + \dfrac{g}{2}\int Z^\alpha \abs{Z_h}{2}\idx \biggr\rbrace  \\
		& ~~~~ + \int Z \bigl( \mu \abs{\nablah v_h}{2} + (\mu+\lambda)\abs{\dvh v_h}{2} + \mu \abs{\dz v_h}{2} \bigr) \idx  \\
		&   \lesssim \mathcal H( \norm{\nablah Z}{\Hnorm{1}}, \norm{Z_t}{\Lnorm{2}}, \norm{z^{\alpha/2} v}{\hHnorm{2}}, \norm{v}{\Hnorm{1}}  , \norm{z^{\alpha/2}v_t}{\Lnorm{2}}) \\
		& ~~~~ \times \bigl( \norm{\dz v}{\hHnorm{2}}^2 + \norm{\nablah v}{\Hnorm{1}}^2 + \norm{\nablah Z}{\Hnorm{1}}^2\bigr).
	\end{aligned}
	\end{equation}
	Integrating this inequality in the temporal variable completes the proof of \eqref{apriori:tangential-01}. 
	
	In order to show \eqref{apriori:tangential-02}, we take inner product of \subeqref{eq:dh-FBCPE}{1} with $ - (Z v_{hh})_h $. After applying integration by parts, it follows,
	\begin{equation}\label{TE-004}
		\begin{aligned}
			& \dfrac{d}{dt} \biggl\lbrace \dfrac{1}{2} \int z^\alpha Z^{\alpha+1} \abs{v_{hh}}{2} \idx \biggr\rbrace - g \int z^\alpha Z_{hh} \dvh (Z v_{hh}) \idx \\
			& + \int Z \bigl( \mu \abs{\nablah v_{hh}}{2} + (\mu+\lambda) \abs{\dvh v_{hh}}{2} + \mu \abs{\dz v_{hh}}{2} \bigr)\idx \\
			& =
			\int  \bigl( \mu \nablah (\log Z)_h\cdot\nablah v_h + (\mu+\lambda) \dvh v_h \nablah (\log Z)_h\bigr) \cdot Z v_{hh} \\
			& ~~~~  - \bigl( \mu \nablah (\log Z)_h\cdot\nablah v + (\mu+\lambda) \dvh v \nablah (\log Z)_h \bigr) \cdot \bigl( Z v_{hh} \bigr)_h \idx
			 \\
			& + \int z^\alpha Z^{\alpha-1} \bigl( Z v_h \cdot\nablah v + \alpha Z_h (v_t + v\cdot \nablah v) \bigr) (Z v_{hh})_h \idx \\
			& + \int z^\alpha Z^{\alpha} W_h \dz v \cdot (Z v_{hh})_h \idx + \alpha \int z^\alpha Z^{\alpha-1}Z_h W \dz v \cdot (Z v_{hh})_h \idx \\
			& - \int z^\alpha Z^{\alpha} \bigl( Z v_h \cdot \nablah v_h + \alpha Z_h ( \dt v_{h} + v\cdot \nablah v_h)  \bigr) \cdot v_{hh} \idx\\
			& - \int z^\alpha Z^{\alpha+1} W_h \dz v_h \cdot v_{hh} \idx - \alpha \int z^\alpha Z^\alpha Z_h W \dz v_h \cdot v_{hh} \idx\\
			& + g \int z^{\alpha} \bigl(  (Z^\alpha - 1) Z_{hh} + \alpha  Z^{\alpha-1} \abs{Z_h}{2}\bigr) \dvh ( Z v_{hh}) \idx \\
			& = : \sum_{k=17}^{24} I_{k}.
		\end{aligned}
	\end{equation}
	After substituting the following equation from applying horizontal derivative to \subeqref{eq:dh-FBCPE}{2},
	\begin{align*}
		& \dt Z_{hh} + (\alpha+1) \overline{z^\alpha v} \cdot \nablah Z_{hh} + \overline{z^\alpha \dvh v} Z_{hh}\\
		& ~~~~ +2 (\alpha+1) \overline{z^\alpha v_h} \cdot \nablah Z_h + 2 \overline{z^\alpha \dvh v_h} Z_h \\
		& ~~~~ + \alpha \overline{z^\alpha v_{hh}} \cdot \nablah Z + \overline{z^\alpha \dvh ( Z v_{hh})} = 0,
	\end{align*}
	one has that
	\begin{align*}
		& - g \int z^\alpha Z_{hh} \dvh (Z v_{hh}) \idx = g \int Z_{hh} \bigl( \dt Z_{hh} + (\alpha+1) \overline{z^\alpha v} \cdot \nablah Z_{hh} \\
		& ~~~~ + \overline{z^\alpha \dvh v} Z_{hh}
		 +2 (\alpha+1) \overline{z^\alpha v_h} \cdot \nablah Z_h + 2 \overline{z^\alpha \dvh v_h} Z_h \\
		& ~~~~ + \alpha \overline{z^\alpha v_{hh}} \cdot \nablah Z   \bigr) \idx = \dfrac{d}{dt} \biggl\lbrace \dfrac{g}{2} \int \abs{Z_{hh}}{2} \idx \biggr\rbrace\\
		& ~~~~ + g \int Z_{hh} \bigl( \dfrac{1-\alpha}{2} \overline{z^{\alpha}\dvh v} Z_{hh} + 2(\alpha+1) \overline{z^\alpha v_h} \cdot \nablah Z_h \\
		& ~~~~ + 2 \overline{z^\alpha \dvh v_h} Z_h + \alpha \overline{z^\alpha v_{hh}} \cdot \nablah Z \bigr) \idx =: \dfrac{d}{dt} \biggl\lbrace \dfrac{g}{2} \int \abs{Z_{hh}}{2} \idx \biggr\rbrace - I_{25}. 
	\end{align*}
	Now we estimate $ I_{17}, \cdots, I_{25} $. Substituting \subeqref{eq:dh-FBCPE}{3} in $ I_{19}, I_{22} $ yields
	\begin{align*}
		& I_{19} = (\alpha+1) \int \biggl\lbrack Z^\alpha \bigl( z^{\alpha+1} \overline{z^\alpha v} - \int_0^z \xi^\alpha v \,d\xi \bigr) \cdot \nablah (\log Z)_h \\
		& ~~~~ \times \dz v \cdot (Z v_{hhh} + Z_h v_{hh} ) \biggr\rbrack \idx 
		+ (\alpha+1) \int \biggl\lbrack Z^\alpha \bigl( z^{\alpha+1} \overline{z^\alpha v_h} \\
		& ~~~~ - \int_0^z \xi^\alpha v_h \,d\xi \bigr) \cdot \nablah \log Z 
		 \dz v \cdot (Z v_{hhh} + Z_h v_{hh} ) \biggr\rbrack \idx \\
		& ~~~~ + \int \biggl\lbrack Z^\alpha \bigl( z^{\alpha+1} \overline{z^\alpha\dvh v_h} - \int_0^z \xi^\alpha \dvh v_h \,d\xi \bigr) \\
		& ~~~~  \times \dz v \cdot (Z v_{hhh} + Z_h v_{hh} ) \biggr\rbrack \idx=: I_{19}' + I_{19}'' + I_{19}''', \\
		& I_{22} = - (\alpha+1) \int \biggl\lbrack Z^{\alpha+1} \bigl( z^{\alpha+1} \overline{z^\alpha v} - \int_0^z \xi^\alpha v\,d\xi \bigr) \cdot \nablah (\log Z)_h \\
		& ~~~~  \times \dz v_h \cdot v_{hh} \biggr\rbrack \idx
		 - (\alpha+1) \int \biggl\lbrack Z^{\alpha+1} \bigl( z^{\alpha+1} \overline{z^\alpha v_h} \\
		 & ~~~~ - \int_0^z \xi^\alpha v_h \,d\xi \bigr) \cdot \nablah \log Z 
		  \dz v_h \cdot v_{hh} \biggr\rbrack \idx \\
		& ~~~~ - \int Z^{\alpha+1} \bigl( z^{\alpha+1} \overline{z^\alpha \dvh v_h} - \int_0^z \xi^\alpha \dvh v_h \,d\xi \bigr)  \dz v_h \cdot v_{hh} \idx \\
		& ~~~~ =: I_{22}' + I_{22}'' + I_{22}'''.
	\end{align*}
Also, substituting \eqref{eq:verticalvelocity} in $ I_{20}, I_{23} $ yields,
\begin{align*}
	& I_{20} = \alpha(\alpha+1) \int \biggl\lbrack Z^{\alpha-2} Z_h \bigl(z^{\alpha+1} \overline{z^\alpha v} - \int_0^z \xi^\alpha v \,d\xi \bigr) \cdot \nablah Z \\
	& ~~~~ \times \dz v \cdot ( Z v_{hhh} + Z_h v_{hh}) \biggr\rbrack \idx 
	 + \alpha \int \biggl\lbrack Z^{\alpha-1} Z_h \bigl( z^{\alpha+1} \overline{z^\alpha\dvh v} \\
	& ~~~~ - \int_0^z \xi^\alpha \dvh v \,d\xi \bigr) \dz v \cdot ( Z v_{hhh} + Z_h v_{hh}) \biggr\rbrack \idx =: I_{20}' + I_{20}'', \\
	& I_{23} = - \alpha (\alpha+1) \int \biggl\lbrack Z^{\alpha-1} Z_h \bigl( z^{\alpha+1} \overline{z^\alpha v} - \int_0^z \xi^\alpha v\,d\xi \bigr) \cdot \nablah Z \\
	& ~~~~ \times \dz v_h \cdot v_{hh} \biggl\rbrack \idx - \alpha \int \biggl\lbrack Z^{\alpha} Z_h \bigl(z^{\alpha+1} \overline{z^\alpha \dvh v} \\
	& ~~~~ - \int_0^z \xi^\alpha \dvh v\,d\xi \bigr) \dz v_h \cdot v_{hh} \biggr\rbrack \idx =: I_{23}' + I_{23}''. 
\end{align*}

Then after applying H\"older's, Minkowski's, and the Sobolev embedding inequalities, we have
\begin{align*}
	& I_{19}' \lesssim \int_0^1 \biggl\lbrack z^{\alpha/2} \hnorm{\dz v}{\Lnorm{\infty}} \bigl( \hnorm{v_{hhh}}{\Lnorm{2}} + \hnorm{Z_h}{\Lnorm{4}} \hnorm{v_{hh}}{\Lnorm{4}} \bigr) \bigl( \hnorm{Z_{hh}}{\Lnorm{2}}+ \hnorm{Z_h}{\Lnorm{4}}^2 \bigr) \\
	& ~~~~ \times \bigl( \hnorm{\overline{z^\alpha v}}{\Lnorm{\infty}} + (\int_0^z \xi^\alpha \hnorm{v}{\Lnorm{\infty}}^2 \,d\xi)^{1/2}\bigr) \biggr\rbrack \,dz \lesssim \bigl(1+ \norm{\nablah Z}{\Hnorm{1}} \bigr)^2 \\
	& ~~~~ \times   \norm{z^{\alpha/2} v}{\hHnorm{2}} \norm{\nablah Z}{\Hnorm{1}}  \norm{\dz v}{\hHnorm{2}} \norm{\nablah^2 v}{\hHnorm{1}}, \\
	& I_{19}'' \lesssim \int_0^1\biggl\lbrack z^{\alpha/2} \hnorm{\dz v}{\Lnorm{8}} \bigl( \hnorm{v_{hhh}}{\Lnorm{2}} + \hnorm{Z_h}{\Lnorm{4}} \hnorm{v_{hh}}{\Lnorm{4}} \bigr) \hnorm{\nablah Z}{\Lnorm{4}} \\
	& ~~~~ \times \bigl( \hnorm{\overline{z^\alpha v_h}}{\Lnorm{8}} + (\int_0^z \xi^\alpha \hnorm{v_h}{\Lnorm{8}}^2 \,d\xi)^{1/2}\bigr) \biggr\rbrack \,dz \lesssim \bigl( 1 + \norm{\nablah Z}{\Hnorm{1}}\bigr) \\
	& ~~~~ \times \norm{z^{\alpha/2} v_h}{\hHnorm{1}} \norm{\nablah Z}{\Hnorm{1}} \norm{z^{\alpha/2} \dz v}{\hHnorm{1}} \norm{\nablah^2 v}{\hHnorm{1}},\\
	& I_{19}''' \lesssim \int_0^1 \biggl\lbrack z^{\alpha/2} \hnorm{\dz v}{\Lnorm{\infty}} \bigl( \hnorm{v_{hhh}}{\Lnorm{2}} + \hnorm{Z_h}{\Lnorm{4}} \hnorm{v_{hh}}{\Lnorm{4}} \bigr) \bigl\lbrack \hnorm{\overline{z^\alpha \nablah^2 v}}{\Lnorm{2}} \\
	& ~~~~ + (\int_0^z \xi^\alpha \hnorm{\nablah^2 v}{\Lnorm{2}}^2 \,d\xi)^{1/2} \bigr\rbrack \biggr\rbrack \,dz \lesssim \bigl(1+ \norm{\nablah Z}{\Hnorm{1}}\bigr) \norm{z^{\alpha/2} v}{\hHnorm{2}}\\
	& ~~~~ \times  \norm{\dz v}{\hHnorm{2}}  \norm{\nablah^2 v}{\hHnorm{1}}, \\
	& I_{22}' \lesssim \int_0^1 \biggl\lbrack z^{\alpha/2} \hnorm{\dz v_h}{\Lnorm{4}} \hnorm{v_{hh}}{\Lnorm{4}} \bigl( \hnorm{Z_{hh}}{\Lnorm{2}} + \hnorm{Z_h}{\Lnorm{4}}^2 \bigr)  \bigl( \hnorm{\overline{z^\alpha v}}{\Lnorm{\infty}} \\
	& ~~~~ + (\int_0^z \xi^\alpha \hnorm{v}{\Lnorm{\infty}}\,d\xi)^{1/2} \bigr) \biggr\rbrack \,dz \lesssim  \bigl(1+\norm{Z_h}{\Hnorm{1}}\bigr) \norm{Z_h}{\Hnorm{1}} \norm{z^{\alpha/2} v}{\hHnorm{2}}^{3/2}\\
	& ~~~~ \times \norm{z^{\alpha/2} \dz v_h}{\Lnorm{2}}^{1/2}  \norm{\nabla v}{\hHnorm{2}}, \\
	& I_{22}'' \lesssim \int_0^1 \biggl\lbrack z^{\alpha/2} \hnorm{\dz v_h}{\Lnorm{4}} \hnorm{v_{hh}}{\Lnorm{4}} \hnorm{\nablah Z}{\Lnorm{4}}\bigl( \hnorm{\overline{z^{\alpha} v_h}}{\Lnorm{4}}\\
	& ~~~~ + (\int_0^z \xi^\alpha \hnorm{v_h}{\Lnorm{4}}^2 \,d\xi)^{1/2} \bigr) \biggr\rbrack \,dz  \lesssim \norm{\nablah Z}{\Hnorm{1}}\norm{z^{\alpha/2} v}{\hHnorm{2}}^{3/2}\\
	& ~~~~ \times \norm{z^{\alpha/2} \dz v_h}{\Lnorm{2}}^{1/2}  \norm{\nabla v}{\hHnorm{2}},\\
	& I_{22}''' \lesssim \int_0^1 z^{\alpha/2} \hnorm{\dz v_h}{\Lnorm{4}} \hnorm{v_{hh}}{\Lnorm{4}} \bigl\lbrack \hnorm{\overline{z^\alpha \nablah^2 v}}{\Lnorm{2}} + (\int_0^z \xi^\alpha 
	\hnorm{\nablah^2 v}{\Lnorm{2}}^2 \,d\xi )^{1/2} \bigr\rbrack \,dz \\
	& ~~~~ \lesssim \norm{z^{\alpha/2} v}{\hHnorm{2}}^{3/2} \norm{z^{\alpha/2} \dz v_h}{\Lnorm{2}}^{1/2}  \norm{\nabla v}{\hHnorm{2}},\\
	& I_{20}' \lesssim \int_0^1 \biggl\lbrack z^{\alpha/2} \hnorm{\dz v}{\Lnorm{4}} \bigl( \hnorm{v_{hhh}}{\Lnorm{2}} + \hnorm{Z_h}{\Lnorm{4}} \hnorm{v_{hh}}{\Lnorm{4}} \bigr) \hnorm{\nablah Z}{\Lnorm{8}}^2 \bigl\lbrack \hnorm{\overline{z^\alpha v}}{\Lnorm{\infty}} \\
	& ~~~~ + (\int_0^z \xi^\alpha \hnorm{v}{\Lnorm{\infty}}^2 \,d\xi)^{1/2} \bigr\rbrack \biggr\rbrack \,dz \lesssim \bigl(1+ \norm{\nablah Z}{\Hnorm{1}}\bigr) \norm{\nablah Z}{\Hnorm{1}}^2\\
	& ~~~~ \times \norm{z^{\alpha/2} v}{\hHnorm{2}} \norm{\dz v}{\Lnorm{2}}^{1/2} \norm{z^{\alpha/2} \dz v}{\hHnorm{1}}^{1/2} \norm{\nablah^2 v}{\hHnorm{1}}, \\
	& I_{20}'' \lesssim \int_0^1 \biggl\lbrack z^{\alpha/2} \hnorm{\dz v}{\Lnorm{4}}\bigl( \hnorm{v_{hhh}}{\Lnorm{2}} + \hnorm{Z_h}{\Lnorm{4}} \hnorm{v_{hh}}{\Lnorm{4}}\bigr) \hnorm{Z_h}{\Lnorm{8}} \bigl\lbrack \hnorm{\overline{z^{\alpha} \nablah v}}{\Lnorm{8}} \\
	& ~~~~ + (\int_0^z \xi^\alpha \hnorm{\nablah v}{\Lnorm{8}}^2 \,d\xi)^{1/2} \bigr\rbrack \biggr\rbrack \,dz \lesssim \bigl( 1+ \norm{Z_h}{\Hnorm{1}}\bigr) \norm{Z_h}{\Hnorm{1}} \norm{z^{\alpha/2}v}{\hHnorm{2}}\\
	& ~~~~ \times \norm{z^{\alpha/2} \dz v}{\Lnorm{2}}^{1/2} \norm{z^{\alpha/2}\dz v}{\hHnorm{1}}^{1/2} \norm{\nablah^2 v}{\hHnorm{1}},\\
	& I_{23}' \lesssim \int_0^1 \biggl\lbrack z^{\alpha/2} \hnorm{\dz v_h}{\Lnorm{4}} \hnorm{v_{hh}}{\Lnorm{4}} \hnorm{\nablah Z}{\Lnorm{4}}^2 \bigl\lbrack \hnorm{\overline{z^\alpha v}}{\Lnorm{\infty}} \\
	& ~~~~ + (\int_0^z \xi^\alpha \hnorm{v}{\Lnorm{\infty}}^2\,d\xi)^{1/2} \bigr\rbrack \biggr\rbrack \,dz \lesssim \norm{\nablah Z}{\Hnorm{1}}^2 \norm{z^{\alpha/2} v}{\hHnorm{2}}^{3/2} \norm{z^{\alpha/2} \dz v}{\hHnorm{1}}^{1/2}\\
	& ~~~~ \times \norm{\nabla v}{\hHnorm{2}},\\
	& I_{23}'' \lesssim \int_0^1 \biggl\lbrack z^{\alpha/2} \hnorm{\dz v_h}{\Lnorm{4}} \hnorm{v_{hh}}{\Lnorm{4}} \hnorm{Z_h}{\Lnorm{4}} \bigl\lbrack \hnorm{\overline{z^\alpha \nablah v}}{\Lnorm{4}} \\
	& ~~~~ + (\int_0^z \xi^\alpha \hnorm{\nablah v}{\Lnorm{4}}^2 \,d\xi)^{1/2} \bigr\rbrack \biggr\rbrack \,dz \lesssim \norm{Z_h}{\Hnorm{1}} \norm{z^{\alpha/2} v}{\hHnorm{2}}^{3/2} \norm{z^{\alpha/2} \dz v}{\hHnorm{1}}^{1/2}\\
	& ~~~~ \times \norm{\nabla v}{\hHnorm{2}}.
\end{align*}
We list the estimates for the rest in the following:
\begin{align*}
	& I_{17} \lesssim \int_0^1 \biggl\lbrack \bigl( \hnorm{\nablah^2 Z}{\Lnorm{2}} + \hnorm{\nablah Z}{\Lnorm{4}}^2 \bigr) \bigl( \hnorm{\nablah^2  v}{\Lnorm{4}}^2 + \hnorm{\nablah v}{\Lnorm{\infty}} \hnorm{\nablah^3 v}{\Lnorm{2}} \\
	& ~~~~ + \hnorm{\nablah Z}{\Lnorm{4}} \hnorm{\nablah v}{\Lnorm{\infty}}  \hnorm{\nablah^2 v}{\Lnorm{4}} \bigr) \biggr\rbrack \,dz
	\lesssim \bigl( \norm{\nablah Z}{\Hnorm{1}}+1\bigl)^2 \norm{\nablah Z}{\Hnorm{1}} \\
	& ~~~~ \times \norm{\nablah v}{\hHnorm{2}}^2,
	\\
	& I_{18} \lesssim \bigl( \norm{z^{\alpha/2}v_h}{\Lnorm{2}}\norm{\nablah v}{\Lnorm{\infty}} + \norm{Z_h}{\Lnorm{6}}\norm{v_t}{\Lnorm{3}} \\
	& ~~~~ + \norm{Z_h}{\Lnorm{6}} \norm{v}{\Lnorm{6}} \norm{\nablah v}{\Lnorm{6}} \bigr) \bigl( \norm{v_{hhh}}{\Lnorm{2}} + \norm{Z_h}{\Lnorm{3}} \norm{v_{hh}}{\Lnorm{6}} \bigr) \\
	& ~~~~ \lesssim \bigl( 1+ \norm{Z_h}{\Hnorm{1}} \bigr) \bigl( \norm{z^{\alpha/2} v_h}{\Lnorm{2}} \norm{\nablah v}{\Hnorm{2}} + \norm{Z_h}{\Hnorm{1}} \norm{z^{\alpha/2} v_t}{\Lnorm{2}}\\
	& ~~~~ + \norm{Z_h}{\Hnorm{1}} \norm{\nabla v_t}{\Lnorm{2}} + \norm{Z_h}{\Hnorm{1}} \norm{v}{\Hnorm{1}} \norm{\nablah v}{\Hnorm{1}}   \bigr) \norm{\nablah^2 v}{\Hnorm{1}},\\
	& I_{21} \lesssim \bigl( \norm{z^{\alpha/2} v_h}{\Lnorm{2}} \norm{\nablah v_h}{\Lnorm{3}} + \norm{Z_h}{\Lnorm{3}} \norm{\dt v_h}{\Lnorm{2}} \\
	& ~~~~ + \norm{Z_h}{\Lnorm{3}} \norm{v}{\Lnorm{3}} \norm{\nablah v_h}{\Lnorm{6}}  \bigr)  \norm{v_{hh}}{\Lnorm{6}} \lesssim \bigl( \norm{z^{\alpha/2} v_h}{\Lnorm{2}} \norm{\nablah^2 v}{\Hnorm{1}} \\
	& ~~~~ + \norm{Z_h}{\Hnorm{1}} \norm{\nablah v_t}{\Lnorm{2}} + \norm{Z_h}{\Hnorm{1}} \norm{v}{\Hnorm{1}} \norm{\nablah^2 v}{\Hnorm{1}}\bigr) \norm{v_{hh}}{\Hnorm{1}}, \\
	& I_{24} \lesssim \bigl( \norm{Z-1}{\Lnorm{\infty}} \norm{Z_{hh}}{\Lnorm{2}} + \norm{Z_h}{\Lnorm{3}} \norm{Z_h}{\Lnorm{6}} \bigr) \bigl( \norm{\nablah v_{hh}}{\Lnorm{2}} \\
	& ~~~~ + \norm{\nablah Z}{\Lnorm{3}} \norm{v_{hh}}{\Lnorm{6}} \bigr) \lesssim \bigl( 1+ \norm{\nablah Z}{\Hnorm{1}}\bigr)  \norm{Z-1}{\Hnorm{2}} \norm{Z_h}{\Hnorm{1}} \\
	& ~~~~ \times  \norm{\nablah^2 v}{\Hnorm{1}}, \\
	& I_{25} \lesssim \bigl( \norm{Z_{hh}}{\Lnorm{2}} \norm{\overline{z^\alpha \nablah v}}{\Lnorm{\infty}} + \norm{Z_h}{\Lnorm{4}} \norm{\overline{z^\alpha \nablah^2 v}}{\Lnorm{4}} \bigr) \norm{Z_{hh}}{\Lnorm{2}}\\
	& ~~~~ \lesssim \norm{\nablah Z}{\Hnorm{1}}^2 \norm{\nablah v}{\hHnorm{2}},
\end{align*}
where we have applied \eqref{TE-002}. Then summing up the above inequalities from \eqref{TE-004} yields
\begin{equation}\label{apriori-ineq-005}
\begin{aligned}
	& \dfrac{d}{dt} \biggl\lbrace \dfrac{1}{2} \int z^\alpha Z^{\alpha+1} \abs{v_{hh}}{2} \idx + \dfrac{g}{2} \int \abs{Z_{hh}}{2} \idx  \biggr\rbrace \\
	& ~~~~ \int Z \bigl( \mu \abs{\nablah v_{hh}}{2} + (\mu+\lambda) \abs{\dvh v_{hh}}{2} + \mu \abs{\dz v_{hh}}{2} \bigr)\idx \\
	& \lesssim \mathcal H (\norm{Z-1}{\Hnorm{2}}, \norm{z^{\alpha/2} v}{\hHnorm{2}}, \norm{v}{\Hnorm{1}} )  \bigl( \norm{\nablah Z}{\Hnorm{1}}^2 \\
	& ~~~~ + \norm{\nabla v}{\Lnorm{2}}^2 + \norm{\nablah v}{\Hnorm{2}}^2 + \norm{z^{\alpha/2} v_t}{\Lnorm{2}}^2 + \norm{\nabla v_t}{\Lnorm{2}}^2 \bigr).
\end{aligned}\end{equation}
Integrating this inequality in the temporal variable implies \eqref{apriori:tangential-02}.
\end{proof}

\subsection{The spatial derivative estimates}
In this section, we will present some spatial derivative estimates for $ Z $ and $ v $. The main goal is to recover the dissipation factor of $ Z $ and the estimates for the vertical (normal) derivatives (i.e., $\dz$) of $ v $. 

We start by studying the dissipation of $ \nablah Z $. Indeed, since $ Z $ is independent of the $ z $-variable, by making use of the boundary conditions in  \eqref{FB-BC-CPE}, integrating \subeqref{FB-CPE}{2} in the vertical direction yields the identity,
\begin{equation}\label{eq:dh-Z}
	\begin{aligned}
		& \dfrac{g}{\alpha+1} Z^\alpha \nablah Z = \mu \overline{\deltah v} + (\mu+\lambda) \overline{\dvh \nablah v} + \mu \nablah \log Z \cdot  \overline{\nablah v}\\
		& ~~~~ + (\mu+\lambda) \overline{\dvh v} \nablah \log Z
		- Z^\alpha \overline{z^\alpha v_t} - Z^\alpha \overline{ z^\alpha v\cdot \nablah v}\\
		& ~~~~ + Z^\alpha \overline{\dz (z^\alpha W) v}.
	\end{aligned}
\end{equation}
Meanwhile, from \eqref{eq:verticalvelocity}, one has
\begin{equation}\label{eq:dz-W}
\begin{aligned}
	& \dz(z^\alpha W) = (\alpha+1) z^\alpha \bigl( (\alpha+1) \overline{z^\alpha v} \cdot \nablah \log Z + \overline{z^\alpha\dvh v} \bigr) \\
	& ~~~~ - (\alpha+1) z^\alpha v \cdot \nablah \log Z - z^\alpha\dvh v.
\end{aligned}
\end{equation}
Thus, after plugging this identity into the terms in \eqref{eq:dh-Z}, applying H\"older's inequality and the Sobolev embedding inequalities leads to the following estimates:
\begin{align*}
	& \hnorm{Z^\alpha \overline{\dz(z^\alpha W) v}}{\Lnorm{2}} \lesssim \hnorm{\overline{z^\alpha v}}{\Lnorm{\infty}} \bigl( \hnorm{\overline{z^\alpha v}}{\Lnorm{\infty}} \hnorm{\nablah Z}{\Lnorm{2}}
	   + \hnorm{\overline{z^\alpha \nablah v}}{\Lnorm{2}} \bigr) \\
	& ~~~~ + \hnorm{\nablah Z}{2} \int_0^1 \hnorm{z^{\alpha/2} v}{\infty}^2 \,dz + \int_0^1 \hnorm{z^{\alpha/2} v}{\Lnorm{\infty}} \hnorm{z^{\alpha/2} \nablah v}{\Lnorm{2}} \,dz \\
	& ~~~~ \lesssim \norm{z^{\alpha/2} v}{\hHnorm{2}}^2 \norm{\nablah Z}{\Lnorm{2}} + \norm{z^{\alpha/2} v}{\hHnorm{2}} \norm{z^{\alpha/2}\nablah v}{\Lnorm{2}}, \\
	& \hnorm{Z^\alpha\overline{z^\alpha v\cdot \nablah v}}{\Lnorm{2}} \lesssim \int_0^1 \hnorm{z^{\alpha/2} v}{\Lnorm{\infty}} \hnorm{z^{\alpha/2} \nablah v}{\Lnorm{2}} \,dz \\
	& ~~~~ \lesssim \norm{z^{\alpha/2} v}{\hHnorm{2}} \norm{z^{\alpha/2}\nablah v}{\Lnorm{2}},\\
	& \hnorm{Z^\alpha\overline{z^\alpha v_t}}{\Lnorm{2}} \lesssim \norm{z^{\alpha/2} v_t}{\Lnorm{2}}, \\
	& \hnorm{\mu \overline{\deltah v} + (\mu+\lambda) \overline{\dvh \nablah v}}{\Lnorm{2}} \lesssim \norm{\nablah^2 v}{\Lnorm{2}},\\
	& \hnorm{\mu \nablah \log Z \cdot \overline{\nablah v} + (\mu+\lambda)\overline{\dvh v}\nablah \log Z}{\Lnorm{2}} \lesssim \norm{\nablah Z}{\Hnorm{1}} \norm{\nablah v}{\hHnorm{1}}.
\end{align*}
Therefore we have from \eqref{eq:dh-Z} that 
\begin{equation}\label{SPE-001}
	\begin{aligned}
	& \norm{\nablah Z}{\Lnorm{2}} = \hnorm{\nablah Z}{\Lnorm{2}} \lesssim \norm{z^{\alpha/2} v_t}{\Lnorm{2}} + \norm{\nablah^2 v}{\Lnorm{2}} \\
	& ~~~~ + \norm{z^{\alpha/2} v}{\hHnorm{2}}^2 \norm{\nablah Z}{\Lnorm{2}} + \norm{z^{\alpha/2} v}{\hHnorm{2}} \norm{z^{\alpha/2}\nablah v}{\Lnorm{2}}\\
	& ~~~~ + \norm{\nablah Z}{\Hnorm{1}} \norm{\nablah v}{\hHnorm{1}}.
	\end{aligned}
\end{equation}
With \eqref{SPE-001} in hand, we will be able to derive the dissipation of $ \partial_{zz} v $. In fact, after taking inner product of \subeqref{FB-CPE}{2} with $ - Z \partial_{zz} v $, applying integration by parts in the resultant equation yields,
\begin{equation}\label{SPE-002}
	\begin{aligned}
		& \dfrac{d}{dt} \biggl\lbrace \dfrac{1}{2} \int z^\alpha Z^{\alpha+1} \abs{v_z}{2} \idx  \biggr\rbrace  + \int Z\bigl( \mu \abs{\nablah \dz v}{2} + (\mu+\lambda) \abs{\dvh \dz v}{2}\\
		& ~~~~ + \mu \abs{\partial_{zz} v}{2} \bigr)\idx
		= 
		g\int z^{\alpha} Z^{\alpha+1} \nablah Z \cdot \partial_{zz} v \idx \\
		& ~~~~ - \int z^\alpha Z^{\alpha+1}  v_z \cdot \nablah v \cdot \dz v \idx - \int Z^{\alpha+1} (z^\alpha  W)_z \dz v \cdot \dz v \idx \\
		& ~~~~ - \alpha \int z^{\alpha-1} Z^{\alpha+1} \bigl( \dt v + v \cdot \nablah v ) \cdot \dz v\idx
		 =: \sum_{k=26}^{29} I_k.
	\end{aligned}
\end{equation}
Applying Young's inequality implies that, for any constant $ \delta > 0 $, 
\begin{equation*}
	I_{26} \lesssim \delta \norm{\partial_{zz} v}{\Lnorm{2}}^2 + \delta^{-1}  \norm{\nablah Z}{\Lnorm{2}}^2.
\end{equation*}
Also, after substituting \eqref{eq:dz-W}, we have
\begin{align*}
	& I_{28} = - (\alpha+1)  \int  z^{\alpha}  Z^{\alpha+1} \bigl((\alpha+1) \overline{z^\alpha v} \cdot \nablah \log Z + \overline{z^\alpha \dvh v} \bigr) \dz v \cdot \dz v \idx\\
	& ~~~~ - \int Z^{\alpha+1} \bigl( (\alpha+1) z^\alpha v \cdot \nablah \log Z + z^\alpha \dvh v \bigr)  \dz v \cdot \dz v\idx \\
	& ~~~~ \lesssim \int_0^1 \bigl( \hnorm{\overline{z^\alpha v}}{\Lnorm{\infty}} \hnorm{\nablah Z}{\Lnorm{2}} + \hnorm{\overline{z^\alpha \nablah v}}{\Lnorm{2}}\bigr) \hnorm{\dz v}{\Lnorm{4}}^2 \,dz +  \bigl( \norm{\nablah Z}{\Lnorm{3}}\norm{v}{\Lnorm{6}} \\
	& ~~~~ + \norm{z^{\alpha/2} \nablah v}{\Lnorm{2}} \bigr) \norm{\dz v}{\Lnorm{3}} \norm{\dz v}{\Lnorm{6}} \lesssim \bigl( \norm{z^{\alpha/2} v}{\hHnorm{2}} \norm{\nablah Z}{\Lnorm{2}} \\
	& ~~~~ + \norm{v}{\Hnorm{1}} \norm{\nablah Z}{\Hnorm{1}} + \norm{z^{\alpha/2} \nablah v}{\Lnorm{2}}  \bigr) \norm{\dz v}{\Hnorm{1}}^2.
\end{align*}
Applying H\"older's, Hardy's (Lemma \ref{lm:hardy}), and the Sobolev embedding inequalities, yields,
\begin{align*}
	& I_{29} \lesssim \norm{v_t}{\Lnorm{2}} \norm{z^{-1}\dz v}{\Lnorm{2}} + \norm{v}{\Lnorm{6}} \norm{\nablah v}{\Lnorm{3}} \norm{z^{-1} \dz v}{\Lnorm{2}} \\
	& ~~~~ \lesssim \delta \norm{\partial_{zz} v}{\Lnorm{2}}^2 + \delta^{-1} \bigl( \norm{z^{\alpha/2} v_t}{\Lnorm{2}}^2 + \norm{\nabla v_t }{\Lnorm{2}}^2  \bigr) \\
	& ~~~~ + \norm{v}{\Hnorm{1}} \norm{\nablah v}{\Hnorm{1}} \norm{\partial_{zz}v}{\Lnorm{2}}, 
\end{align*}
where we have used \eqref{TE-002} and the following inequalities from Lemma \ref{lm:hardy}, 
\begin{equation}\label{SPE-007}
\begin{aligned}
	& \norm{z^{-1} \dz v}{\Lnorm{2}}^2 = \int_0^1 z^{-2} \hnorm{\dz v-\dz v|_{z=0}}{\Lnorm{2}}^2 \,dz  \\
	& ~~~~ \lesssim \int_0^1 
	\hnorm{\partial_{zz} v}{\Lnorm{2}}^2 \,dz = 
	\norm{\partial_{zz} v}{\Lnorm{2}}^2,
	\end{aligned}
	\end{equation}
since $ \dz v\big|_{z=0,1} = 0 $ from \eqref{FB-BC-CPE}. 
Similarly, we list the rest of the estimates in the following:
\begin{align*}
	& I_{27} \lesssim \norm{v_z}{\Lnorm{3}} \norm{z^{\alpha/2} \nablah v}{\Lnorm{2}} \norm{v_z}{\Lnorm{6}} \lesssim \norm{z^{\alpha/2} \nablah v}{\Lnorm{2}} \norm{\dz v}{\Hnorm{1}}^2.
\end{align*}
Thus, from \eqref{SPE-001}, \eqref{SPE-002} and the above inequalities with suitably small $ \delta $, we have the following inequality:
\begin{equation}\label{apriori-ineq-006}
	\begin{aligned}
		& \int \abs{\nablah Z}{2} \idx \lesssim  \hnorm{\nablah Z}{\Lnorm{2}}^2 \lesssim \norm{z^{\alpha/2} v_t}{\Lnorm{2}}^2 + \norm{\nablah^2 v}{\Lnorm{2}}^2 \\
		& ~~~~ + \norm{z^{\alpha/2} v}{\hHnorm{2}}^4 \norm{\nablah Z}{\Lnorm{2}}^2 + \norm{z^{\alpha/2} v}{\hHnorm{2}}^2 \norm{z^{\alpha/2}\nablah v}{\Lnorm{2}}^2\\
		& ~~~~+ \norm{\nablah Z}{\Hnorm{1}}^2 \norm{\nablah v}{\hHnorm{1}}^2,\\
		& \dfrac{d}{dt} \biggl\lbrace \dfrac{1}{2} \int z^\alpha Z^{\alpha+1} \abs{v_z}{2} \idx  \biggr\rbrace  + \int Z\bigl( \mu \abs{\nablah \dz v}{2} + (\mu+\lambda) \abs{\dvh \dz v}{2}\\
		& ~~~~ + \mu/2 \abs{\partial_{zz} v}{2} \bigr)\idx \lesssim \int \abs{\nablah Z}{2} \idx + \norm{z^{\alpha/2} v_t}{\Lnorm{2}}^2 + \norm{\nabla v_t }{\Lnorm{2}}^2 \\
		& ~~~~ + \mathcal H( \norm{z^{\alpha/2} v}{\hHnorm{2}}, \norm{\nablah Z}{\Hnorm{1}}, \norm{v}{\Hnorm{1}}) \bigl( \norm{\nablah Z}{\Lnorm{2}}^2\\
		& ~~~~ + \norm{\nabla v}{\Hnorm{1}}^2 \bigr).
	\end{aligned}
	\end{equation}
Then integrating the above inequalities in the temporal variable concludes the following proposition:
\begin{proposition}\label{lm:L^2-v-dz-est}Under the same assumptions as in Proposition \ref{lm:basic-energy}, one has that
	\begin{equation}\label{apriori:dsptn-dh-Z}
		\begin{aligned}
			& \dfrac{1}{2} \int z^\alpha Z^{\alpha+1} \abs{v_z}{2} \idx   + \int_0^T \int \biggl( Z\bigl( \mu \abs{\nablah \dz v}{2} + (\mu+\lambda) \abs{\dvh \dz v}{2}\\
		& ~~~~ + \mu/2 \abs{\partial_{zz} v}{2} \bigr) + \abs{\nablah Z}{2} \biggr) \idx \,dt \leq \dfrac{1}{2} \int z^\alpha Z_0^{\alpha+1} \abs{v_{0,z}}{2} \idx \\
		& ~~~~ + C \int_0^T \biggl( \norm{z^{\alpha/2} v_t}{\Lnorm{2}}^2 + \norm{\nabla v_t }{\Lnorm{2}}^2 \biggr) \,dt + C \int_0^T \biggl( \mathcal H( \norm{z^{\alpha/2} v}{\hHnorm{2}}, \\
		& ~~~~ ~~~~
		 \norm{\nablah Z}{\Hnorm{1}}, \norm{v}{\Hnorm{1}})  \bigl( \norm{\nablah Z}{\Lnorm{2}}^2
		 + \norm{\nabla v}{\Hnorm{1}}^2 \bigr) \biggr) \,dt,
		\end{aligned}
	\end{equation}
	for some constant $ 0 < C < \infty $. 
\end{proposition}
We will perform similar analysis to \subeqref{eq:dh-FBCPE}{1} in the following. The result leads to the following proposition:
\begin{proposition}\label{lm:L^2-v-dhz-est}
Under the same assumptions as in Proposition \ref{lm:basic-energy}, one has that 
	\begin{equation}\label{apriori:dsptn-dhh-z}
		\begin{aligned}
			& \dfrac{1}{2} \int z^\alpha Z^{\alpha+1} \abs{v_{hz}}{2} \idx  + \int_0^T \int Z \bigl( \mu \abs{\nablah \dz v_h}{2} \\ 
			& ~~~~ + (\mu+\lambda) \abs{\dvh \dz v_h}{2} 
			+ \mu/2 \abs{\partial_{zz} v_h}{2} \bigr) + \abs{\nablah^2 Z}{2} \idx\,dt \\
			&  \lesssim  \dfrac{1}{2} \int z^\alpha Z_0^{\alpha+1} \abs{v_{0,hz}}{2} \idx  + \int_0^T \norm{\nablah v_t}{\Lnorm{2}}^2 + \norm{\nablah^3 v}{\Lnorm{2}}^2 \,dt \\
			& ~~~~ 
			 + C \int_0^T \mathcal H( \norm{\nablah Z}{\Hnorm{1}}, \norm{z^{\alpha/2} v}{\hHnorm{2}}, \norm{v}{\Hnorm{1}} ) \\
			 & ~~~~ \times \bigl( \norm{\nablah Z}{\Hnorm{1}}^2 + \norm{\nablah v}{\Hnorm{2}}^2 
			 + \norm{\dz v}{\Hnorm{1}}^2 
			 + \norm{z^{\alpha/2} v_t}{\Lnorm{2}}^2 \\
			 & ~~~~ + \norm{\nabla v_t}{\Lnorm{2}}^2 \bigr) \,dt .
		\end{aligned}
	\end{equation}
\end{proposition} 
\begin{proof}
	Applying $ \partial_h $ to \eqref{eq:dh-Z} and \eqref{eq:dz-W} gives us the following identities:
	\begin{align}
		&\dfrac{g}{\alpha+1} Z^\alpha \nablah Z_h = - \dfrac{\alpha g}{\alpha+1} Z^{\alpha-1} Z_h \nablah Z + \mu  \overline{\deltah v_h} + (\mu+\lambda) \overline{\dvh \nablah v_h} \nonumber \\
		& ~~~~ + \mu \nablah (\log Z)_h \cdot \overline{\nablah v} + (\mu+\lambda) \overline{\dvh v} \nablah (\log Z)_h + \mu \nablah \log Z \cdot \overline{\nablah v_h} \nonumber \\
		& ~~~~ + (\mu+\lambda) \overline{\dvh v_h} \nablah \log Z  
		 - Z^\alpha \overline{z^\alpha \dt v_h} - Z^\alpha \overline{z^\alpha v_h \cdot \nablah v} \nonumber \\
		& ~~~~ - Z^\alpha \overline{z^\alpha v \cdot \nablah v_h} + Z^\alpha \overline{\dz(z^\alpha W) v_h}
		+ Z^\alpha \overline{\dz( z^\alpha W_h) v} - \alpha Z^{\alpha-1}Z_h \overline{z^\alpha v_t} \nonumber\\
		& ~~~~ -  \alpha Z^{\alpha-1} Z_h \overline{z^\alpha v\cdot \nablah v} + \alpha Z^{\alpha-1} Z_h \overline{\dz (z^\alpha W)v}, \label{eq:dhh-Z}\\
		\intertext{and}
		& \dz (z^\alpha W_h) = (\alpha+1) z^\alpha \bigl( (\alpha+1) \overline{z^\alpha v_h} \cdot\nablah \log Z + (\alpha+1) \overline{z^\alpha v} \cdot \nablah (\log Z)_h \nonumber\\
		& ~~~~ + \overline{z^\alpha \dvh v_h} \bigr) - (\alpha+1) z^\alpha v_h \cdot \nablah \log Z - (\alpha+1) z^\alpha v \cdot \nablah (\log Z)_h \nonumber\\
		& ~~~~ - z^\alpha \dvh v_h. \label{eq:dzh-W}
	\end{align}
	Now we estimate the $ L^2 $-norm of $ \nablah^2 Z $. It suffices to calculate the $ L^2 $ norms of the right-hand side of \eqref{eq:dhh-Z}. Inserting \eqref{eq:dzh-W} in the following and applying Minkowski's inequality yield,
	\begin{align*}
		& \hnorm{Z^\alpha \overline{\dz(z^\alpha W_h) v}}{\Lnorm{2}} \lesssim \hnorm{\overline{z^{\alpha} v}}{\Lnorm{\infty}} \bigl\lbrack \hnorm{\overline{z^\alpha v_h}}{\Lnorm{4}}\hnorm{\nablah Z}{\Lnorm{4}} + \hnorm{\overline{z^\alpha v}}{\Lnorm{\infty}} \hnorm{\nablah^2 Z}{\Lnorm{2}} \\
		& ~~~~ + \hnorm{\overline{z^\alpha v}}{\Lnorm{\infty}} \hnorm{\nablah Z}{\Lnorm{4}}^2 + \hnorm{\overline{z^\alpha\nablah^2 v}}{\Lnorm{2}}  \bigr\rbrack + \hnorm{\nablah Z}{\Lnorm{4}} \int_0^1 \hnorm{z^{\alpha/2}v}{\Lnorm{\infty}} \hnorm{z^{\alpha/2} v_h}{\Lnorm{4}} \,dz\\
		& ~~~~ + \bigl( \hnorm{\nablah^2 Z }{\Lnorm{2}} + \hnorm{\nablah Z}{\Lnorm{4}}^2\bigr) \int_0^1 \hnorm{z^{\alpha/2} v}{\Lnorm{\infty}}^2 \,dz + \int_0^1 \hnorm{z^{\alpha/2} \nablah^2 v}{\Lnorm{2}} \hnorm{z^{\alpha/2} v}{\Lnorm{\infty}}\,dz\\
		& ~~~~ \lesssim \bigl(1 + \norm{\nablah Z}{\Hnorm{1}} \bigr) \norm{z^{\alpha/2} v}{\hHnorm{2}}^2  \norm{\nablah Z}{\Hnorm{1}} + \norm{z^{\alpha/2} v}{\hHnorm{2}} \norm{z^{\alpha/2} \nablah^2 v}{\Lnorm{2}}.
	\end{align*}
	Similarly, substituting \eqref{eq:dz-W} into the terms involving $ \dz(z^\alpha W) $ implies, 
	\begin{align*}
		& \hnorm{Z^\alpha \overline{\dz(z^\alpha W)v_h}}{\Lnorm{2}} \lesssim \hnorm{\overline{z^\alpha v_h}}{\Lnorm{4}} \bigl( \hnorm{\overline{z^\alpha v}}{\Lnorm{\infty}} \hnorm{\nablah Z}{\Lnorm{4}} + \hnorm{\overline{z^\alpha \nablah v}}{\Lnorm{4}} \bigr)  \\
		& ~~~~ + \hnorm{\nablah Z}{\Lnorm{4}} \int_0^1 \hnorm{z^{\alpha/2} v}{\Lnorm{\infty}} \hnorm{z^{\alpha/2} v_h}{\Lnorm{4}} \,dz + \int_0^1 \hnorm{z^{\alpha/2} \nablah v}{\Lnorm{4}}^2 \,dz \\
		& ~~~~ \lesssim \bigl( 1 + \norm{\nablah Z}{\Hnorm{1}}\bigr) \norm{z^{\alpha/2} v}{\hHnorm{2}} \norm{z^{\alpha/2} \nablah v}{\hHnorm{1}}  ,  \\
		& \hnorm{Z^{\alpha-1} Z_h \overline{\dz(z^\alpha W) v} }{\Lnorm{2}} \lesssim \hnorm{\nablah Z}{\Lnorm{4}} \hnorm{\overline{z^\alpha v}}{\Lnorm{\infty}} \bigl( \hnorm{\overline{z^\alpha v}}{\Lnorm{\infty}} \hnorm{\nablah Z}{\Lnorm{4}} \\
		& ~~~~ + \hnorm{\overline{z^\alpha \nablah v}}{\Lnorm{4}} \bigr) + \hnorm{\nablah Z}{\Lnorm{4}}^2 \int_0^1 \hnorm{z^{\alpha/2} v}{\Lnorm{\infty}}^2 \,dz + \hnorm{\nablah Z}{\Lnorm{4}} \\
		& ~~~~ \times \int_0^1 \hnorm{z^{\alpha/2} \nablah v}{\Lnorm{4}} \hnorm{z^{\alpha/2} v}{\Lnorm{\infty}} \,dz \lesssim \bigl( 1 + \norm{\nablah Z}{\Hnorm{1}}\bigr) \\
		& ~~~~ \times \norm{z^{\alpha/2} v}{\hHnorm{2}}^2 \norm{\nablah Z}{\Hnorm{1}}. 
	\end{align*}
	We list the rest of estimates in the following:
	\begin{align*}
		& \hnorm{Z^{\alpha-1} Z_h \nablah Z}{\Lnorm{2}} \lesssim \norm{\nablah Z}{\Hnorm{1}}^2, \\
		& \hnorm{Z^{\alpha} \overline{z^\alpha v_h \cdot \nablah v}}{\Lnorm{2}} \lesssim \int_0^1 \hnorm{z^{\alpha/2} \nablah v}{\Lnorm{4}}^2 \,dz \lesssim \norm{z^{\alpha/2} \nablah v}{\hHnorm{1}}^2, \\
		& \hnorm{Z^\alpha\overline{z^\alpha v\cdot \nablah v_h}}{\Lnorm{2}} \lesssim \int_0^1 \hnorm{z^{\alpha/2} v}{\Lnorm{\infty}} \hnorm{z^{\alpha/2} \nablah^2 v}{\Lnorm{2}}\,dz \\
		& ~~~~ \lesssim \norm{z^{\alpha/2} v}{\hHnorm{2}} \norm{z^{\alpha/2}\nablah v}{\hHnorm{1}}, \\
		& \hnorm{Z^{\alpha-1} Z_h \overline{z^\alpha v_t}}{\Lnorm{2}} \lesssim \hnorm{\nablah Z}{\Lnorm{4}} \hnorm{\overline{v_t}}{\Lnorm{4}} \lesssim \hnorm{\nablah Z}{\Hnorm{1}} \hnorm{\overline{v_t}}{\Lnorm{2}}^{1/2} \hnorm{\overline{v_t}}{\Hnorm{1}}^{1/2} \\
		& ~~~~ \lesssim \norm{\nablah Z}{\Hnorm{1}} \bigl( \norm{z^{\alpha/2}v_t}{\Lnorm{2}} + \norm{\nabla v_t}{\Lnorm{2}} \bigr), \\
		& \hnorm{Z^{\alpha-1}Z_h \overline{z^\alpha v \cdot \nablah v}}{\Lnorm{2}} \lesssim \hnorm{\nablah Z}{\Lnorm{4}} \int_0^1 \hnorm{z^{\alpha/2} v}{\Lnorm{\infty}} \hnorm{z^{\alpha/2} \nablah v}{\Lnorm{4}} \,dz \\
		& ~~~~ \lesssim \norm{\nablah Z}{\Hnorm{1}} \norm{z^{\alpha/2} v}{\hHnorm{2}} \norm{z^{\alpha/2} \nablah v}{\hHnorm{1}}, \\
		& \hnorm{Z^\alpha\overline{z^\alpha\dt v_h}}{\Lnorm{2}} \lesssim \norm{\nabla v_t}{\Lnorm{2}}, \\
		& \hnorm{\mu \overline{\deltah v_h} + (\mu+\lambda) \overline{\dvh \nablah v_h}}{\Lnorm{2}} \lesssim \norm{\nablah^3 v}{\Lnorm{2}},\\
		& \hnorm{\mu \nablah (\log Z)_h\cdot\overline{\nablah v} + (\mu + \lambda) \overline{\dvh v}\nablah (\log Z)_h + \mu \nablah \log Z \cdot \overline{\nablah v_h} \\
		& ~~~~ + (\mu+\lambda) \overline{\dvh v_h}\nablah \log Z}{\Lnorm{2}} \lesssim \bigl(\norm{\nablah Z}{\Hnorm{1}} + \norm{\nablah Z}{\Hnorm{1}}^2\bigr)\norm{\nablah v}{\hHnorm{2}},
	\end{align*}
	where we have applied \eqref{TE-002}. Then from \eqref{eq:dhh-Z}, we have
	\begin{equation}\label{SPE-003}
		\begin{aligned}
			& \norm{\nablah^2 Z}{\Lnorm{2}} = \hnorm{\nablah^2 Z}{\Lnorm{2}} \lesssim \norm{\nablah v_t}{\Lnorm{2}} + \norm{\nablah^3 v}{\Lnorm{2}} \\
			& ~~~~ + \mathcal H( \norm{\nablah Z}{\Hnorm{1}}, \norm{z^{\alpha/2} v}{\hHnorm{2}}) \bigl( \norm{\nablah Z}{\Hnorm{1}} + \norm{z^{\alpha/2} \nablah v}{\hHnorm{1}} \\
			& ~~~~ + \norm{z^{\alpha/2} v_t}{\Lnorm{2}} + \norm{\nabla v_t}{\Lnorm{2}} + \norm{\nablah v}{\hHnorm{2}} \bigr).
		\end{aligned}
	\end{equation}
	Integrating the square of \eqref{SPE-003} in the temporal variable yields,
	\begin{equation}\label{SPE-005}
		\begin{aligned}
			& \int_0^T \norm{\nablah^2 Z}{\Lnorm{2}}^2 \,dt \lesssim \int_0^T \norm{\nablah v_t}{\Lnorm{2}}^2 + \norm{\nablah^3 v}{\Lnorm{2}}^2 \\
			& ~~~~ + \mathcal H( \norm{\nablah Z}{\Hnorm{1}}, \norm{z^{\alpha/2} v}{\hHnorm{2}}) \bigl( \norm{\nablah Z}{\Hnorm{1}}^2 + \norm{z^{\alpha/2} \nablah v}{\hHnorm{1}}^2 \\
			& ~~~~ + \norm{z^{\alpha/2} v_t}{\Lnorm{2}}^2 + \norm{\nabla v_t}{\Lnorm{2}}^2 + \norm{\nablah v}{\hHnorm{2}}^2 \bigr)\,dt.
		\end{aligned}
	\end{equation}
	
	Next, we will derive the dissipation of $ \nablah \partial_{zz} v $. After taking inner product of \subeqref{eq:dh-FBCPE}{1} with $ - Z \partial_{zz} v_h $ and applying integration by parts, we have the following equation: 
	\begin{align}\label{SPE-004}
		& \dfrac{d}{dt}\biggl\lbrace \dfrac{1}{2} \int z^\alpha Z^{\alpha+1} \abs{v_{hz}}{2} \idx \biggr\rbrace + \int Z \bigl( \mu \abs{\nablah \dz v_h}{2} + (\mu+\lambda) \abs{\dvh \dz v_h}{2} \nonumber \\
		& ~~~~ + \mu \abs{\partial_{zz} v_h}{2} \bigr)\idx =
		 \int \mu \nablah (\log Z)_h \cdot \nablah v_z \cdot Z \dz v_h \nonumber \\
		& ~~~~ + (\mu+\lambda) \dvh v_z \nablah(\log Z)_h \cdot Z \dz v_h \idx
		+ g \int ( z^\alpha Z^\alpha \nablah Z)_h \cdot Z \partial_{zz} v_h \idx \nonumber \\
		& ~~~~ - \int z^\alpha Z^{\alpha+1} v_z \cdot \nablah v_h \cdot \dz v_h \idx - \int Z^{\alpha+1} ( z^\alpha W)_z \dz v_h \cdot \dz v_h \idx \nonumber\\
		& ~~~~ - \alpha \int z^{\alpha-1} Z^{\alpha+1} (\dt v_h + v\cdot\nablah v_h ) \cdot \dz v_h \idx \nonumber \\
		& ~~~~ + \int z^\alpha Z^{\alpha+1} v_h \cdot \nablah v \cdot \partial_{zz} v_h \idx  + \int z^\alpha Z^{\alpha+1} W_h \dz v \cdot \partial_{zz} v_h\idx \nonumber\\
		& ~~~~ + \alpha \int z^\alpha Z^\alpha Z_h (v_t + v \cdot \nablah v) \cdot \partial_{zz} v_h \idx \nonumber \\
		& ~~~~ + \alpha \int z^\alpha Z^\alpha Z_h W \dz v \cdot \partial_{zz} v_h \idx =: \sum_{k=30}^{38} I_{k}. 
	\end{align}	
	Applying Young's inequality yields that, for any constant $ \delta > 0 $, 
	\begin{align*}
		& I_{31} \lesssim \norm{\nablah^2 Z}{\Lnorm{2}} \norm{\partial_{zz}v_h}{\Lnorm{2}} + \norm{\nablah Z}{\Hnorm{1}}^2 \norm{\partial_{zz}v_h}{\Lnorm{2}} \\
		& ~~~~ \lesssim \delta \norm{\partial_{zz}v_h }{\Lnorm{2}}^2 + \delta^{-1} \norm{\nablah^2 Z}{\Lnorm{2}}^2 +  \norm{\nablah Z}{\Hnorm{1}}^2 \norm{\partial_{zz}v_h}{\Lnorm{2}}.
	\end{align*}
	Now we calculate the terms involving the vertical velocity. After substituting \eqref{eq:dz-W} into $I_{33}$, one has that
	\begin{align*}
		& I_{33} = - (\alpha+1) \int z^\alpha \dz v_h \cdot \dz v_h \bigl( \overline{z^\alpha v} \cdot \nablah Z^{\alpha+1} + \overline{z^\alpha \dvh v} Z^{\alpha+1} \bigr) \idx \\
		& ~~~~ + \int z^\alpha \dz v_h \cdot \dz v_h \bigl( v \cdot \nablah Z^{\alpha+1} + \dvh v Z^{\alpha+1} \bigr) \idx \lesssim \bigl( \hnorm{\nablah Z}{\Lnorm{2}} \hnorm{\overline{z^{\alpha} v}}{\Lnorm{\infty}} \\
		& ~~~~ + \hnorm{\overline{z^\alpha\nablah v}}{\Lnorm{2}}\bigr) \int_0^1 \hnorm{z^{\alpha/2} \dz v_h}{\Lnorm{4}}^2 \,dz + \bigl( \hnorm{\nablah Z}{\Lnorm{6}} \norm{v}{\Lnorm{3}} + \norm{\nablah v}{\Lnorm{2}} \bigr) \\
		& ~~~~ \times \norm{\dz v_h}{\Lnorm{6}}\norm{\dz v_h}{\Lnorm{3}} \lesssim \bigl( \norm{\nablah Z}{\Hnorm{1}} + 1\bigr) \bigl( \norm{z^{\alpha/2} v}{\hHnorm{2}} + \norm{v}{\Hnorm{1}} \bigr) \\
		& ~~~~ \times \norm{\dz v_h}{\Hnorm{1}}^2.
	\end{align*}
	Now we substitute  \subeqref{eq:dh-FBCPE}{3} into $ I_{36} $, and it follows that,
	\begin{align*}
		& I_{36} = \int z^{\alpha+1} \dz v \cdot \partial_{zz} v_h \bigl((\alpha+1) \overline{z^\alpha v} \cdot \nablah (\log Z)_h Z^{\alpha+1} + \overline{z^\alpha v_h} \cdot \nablah Z^{\alpha+1} \\
		& ~~~~ + \overline{z^\alpha \dvh v_h} Z^{\alpha+1}  \bigr) \idx - \int \dz v \cdot \partial_{zz} v_h \bigl( (\alpha+1) \int_0^z \xi^\alpha v \,d\xi \cdot \nablah (\log Z)_h \\
		& ~~~~ \times Z^{\alpha+1}
		 + \int_0^z \xi^\alpha v_h \,d\xi \cdot \nablah Z^{\alpha+1} + \int_0^z \xi^\alpha\dvh v_h\,d\xi Z^{\alpha+1}  \bigr) \idx\\
		& ~~~~ \lesssim \bigl((\hnorm{\nablah^2Z}{\Lnorm{2}} + \hnorm{\nablah Z}{\Lnorm{4}}^2 )\hnorm{\overline{z^\alpha v}}{\Lnorm{\infty}} + \hnorm{\overline{z^\alpha v_h}}{\Lnorm{4}} \hnorm{\nablah Z}{\Lnorm{4}} + \hnorm{\overline{z^\alpha\nablah^2 v}}{\Lnorm{2}} \bigr)\\
		& ~~~~ \times \int_0^1 \hnorm{\dz v}{\Lnorm{\infty}} \hnorm{\partial_{zz}v_h}{\Lnorm{2}} \,dz + \int_0^1 \bigl( (\hnorm{\nablah^2Z}{\Lnorm{2}} + \hnorm{\nablah Z}{\Lnorm{4}}^2 ) \\
		& ~~~~ \times (\int_0^z \xi^\alpha \hnorm{v}{\Lnorm{\infty}}^2 \,d\xi)^{1/2} + \hnorm{\nablah Z}{\Lnorm{4}} (\int_0^z \xi^\alpha \hnorm{v_h}{\Lnorm{4}}^2\,d\xi)^{1/2}\\
		& ~~~~ + (\int_0^z \xi^\alpha \hnorm{\nablah^2 v}{\Lnorm{2}}^2\,d\xi)^{1/2}  \bigr)  \hnorm{\dz v}{\Lnorm{\infty}} \hnorm{\partial_{zz} v_h}{\Lnorm{2}} \,dz \lesssim \bigl( 1  + \norm{\nabla Z}{\Hnorm{1}}^2  \bigr) \\
		& ~~~~ \times \norm{z^{\alpha/2} v}{\hHnorm{2}} \norm{\dz v}{\hHnorm{2}} \norm{\partial_{zz}v_h}{\Lnorm{2}}.
	\end{align*}
	Next, after substituting \eqref{eq:verticalvelocity} into $ I_{38} $, we have that,
	\begin{align*}
		& I_{38} = \alpha \int z^{\alpha+1} \dz v \cdot \partial_{zz}v_h \bigl( (\alpha+1) \overline{z^\alpha v} \cdot \nablah Z Z^{\alpha-1} Z_h \\
		& ~~~~ + \overline{z^\alpha \dvh v} Z^\alpha Z_h \bigr) \idx - \alpha \int \dz v\cdot \partial_{zz}v_h \bigl( (\alpha+1) \int_0^z \xi^\alpha v\,d\xi \cdot \nablah Z \\
		& ~~~~ \times Z^{\alpha-1} Z_h + \int_0^z \xi^\alpha \dvh v \,d\xi Z^\alpha Z_h \bigr) \idx \lesssim \bigl( \hnorm{\overline{z^\alpha v}}{\Lnorm{\infty}} \hnorm{\nablah Z}{\Lnorm{8}}^2 \\
		& ~~~~ + \hnorm{\overline{z^\alpha \dvh v}}{\Lnorm{8}} \hnorm{\nablah Z}{\Lnorm{8}} \bigr)  \int_0^1 \hnorm{\dz v}{\Lnorm{4}} \hnorm{\partial_{zz} v_h}{\Lnorm{2}} \,dz \\
		& ~~~~ + \int_0^1 \bigl( (\int_0^z \xi^\alpha \hnorm{v}{\Lnorm{\infty}}^2\,d\xi)^{1/2} \hnorm{\nablah Z}{\Lnorm{8}}^2 + (\int_0^z \xi^\alpha \hnorm{\nablah v}{\Lnorm{8}}^2 \,d\xi)^{1/2} \\
		& ~~~~ \times \hnorm{\nablah Z}{\Lnorm{8}} \bigr) \hnorm{\dz v}{\Lnorm{4}} \hnorm{\partial_{zz} v_h}{\Lnorm{2}} \,dz \lesssim \bigl( 1 + \norm{\nablah Z}{\Hnorm{1}}  \bigr)  \norm{\nablah Z}{\Hnorm{1}} \\
		& ~~~~ \times \norm{z^{\alpha/2} v}{\hHnorm{2}} \norm{\dz v}{\hHnorm{1}}\norm{\partial_{zz}v_h}{\Lnorm{2}}.
	\end{align*}
	On the other hand, applying H\"older's, Hardy's (Lemma \ref{lm:hardy}), and the Sobolev embedding inequalities, yields that for some $\delta \in (0,1)$,
	\begin{align*}
		& I_{34} \lesssim \norm{\dt v_h}{\Lnorm{2}} \norm{z^{-1} \dz v_h}{\Lnorm{2}} + \norm{v}{\Lnorm{6}} \norm{\nablah v_h}{\Lnorm{3}} \norm{z^{-1} \dz v_h}{\Lnorm{2}} \\
		& ~~~~ \lesssim \delta \norm{\partial_{zz}v_h}{\Lnorm{2}}^2 + \delta^{-1} \norm{\nablah v_t}{\Lnorm{2}}^2 + \norm{v}{\Hnorm{1}} \norm{\nablah^2 v}{\Hnorm{1}}\norm{\partial_{zz}v_h}{\Lnorm{2}},
	\end{align*}
	where we have used the inequality,
	\begin{equation}\label{SPE-008}
	\begin{aligned}
		& \norm{z^{-1} \dz v_h}{\Lnorm{2}}^2 = \int_0^1 z^{-2} \hnorm{\dz v_h - \dz v_h|_{z=0}}{\Lnorm{2}}^2 \,dz \\
		& ~~~~ \lesssim \int_0^1 \hnorm{\partial_{zz} v_h}{\Lnorm{2}}^2\,dz = \norm{\partial_{zz}v_h}{\Lnorm{2}}^2.
	\end{aligned}\end{equation}
	To finish the proof of this proposition, we list the estimates for the rest in \eqref{SPE-004}: 
	\begin{align*}
		& I_{30} \lesssim \int \bigl( \hnorm{\nablah^2 Z}{\Lnorm{2}} + \hnorm{\nablah Z}{\Lnorm{4}}^2 \bigr) \hnorm{\nablah v_z}{\Lnorm{4}}^2 \,dz \lesssim \bigl( \norm{\nablah Z}{\Hnorm{1}} + 1) \norm{\nablah Z}{\Hnorm{1}} \\
		& ~~~~ \times \norm{\dz v}{\hHnorm{2}}^2 ,\\
		& I_{32} \lesssim \norm{z^{\alpha/2} \nablah v_h}{\Lnorm{2}} \norm{v_z}{\Lnorm{3}} \norm{\dz v_h}{\Lnorm{6}} \lesssim \norm{z^{\alpha/2} v}{\hHnorm{2}} \norm{\dz v}{\Hnorm{1}} \\
		& ~~~~ \times \norm{\nablah v}{\Hnorm{2}},\\
		& I_{35} \lesssim \norm{z^{\alpha/2}v_h}{\Lnorm{2}} \norm{\nablah v}{\Lnorm{\infty}} \norm{\partial_{zz}v_h}{\Lnorm{2}} \lesssim \norm{z^{\alpha/2} \nablah v}{\Lnorm{2}} \norm{\nablah v}{\Hnorm{2}}^2, \\
		& I_{37} \lesssim \norm{Z_h}{\Lnorm{6}} \bigl( \norm{v_t}{\Lnorm{3}} + \norm{v}{\Lnorm{6}} \norm{\nablah v}{\Lnorm{6}}  \bigr) \norm{\partial_{zz}v_h}{\Lnorm{2}} \lesssim \norm{Z_h}{\Hnorm{1}} \\
		& ~~~~ \times \bigl( \norm{z^{\alpha/2}v_t}{\Lnorm{2}} + \norm{\nabla v_t}{\Lnorm{2}} + \norm{v}{\Hnorm{1}} \norm{\nablah v}{\Hnorm{1}}\bigr) \norm{\partial_{zz}v_h}{\Lnorm{2}}.
	\end{align*}
	Summing up the inequalities above with suitably small $ \delta $ will yield the following inequality,
	\begin{equation}\label{apriori-ineq-007}
	\begin{aligned}
			& \dfrac{d}{dt}\biggl\lbrace \dfrac{1}{2} \int  z^\alpha Z^{\alpha+1} \abs{v_{hz}}{2} \idx \biggr\rbrace + \int \biggl( Z \bigl( \mu \abs{\nablah \dz v_h}{2} \\ 
			& ~~~~ + (\mu+\lambda) \abs{\dvh \dz v_h}{2} 
			+ \mu/2 \abs{\partial_{zz} v_h}{2} \bigr)\biggr) \idx \lesssim \norm{\nablah^2 Z}{\Lnorm{2}}^2 \\
			& ~~~~  + \norm{\nablah v_t}{\Lnorm{2}}^2 + \norm{\nablah^3 v}{\Lnorm{2}}^2 
			 + \mathcal H( \norm{\nablah Z}{\Hnorm{1}}, \norm{z^{\alpha/2} v}{\hHnorm{2}}, \norm{v}{\Hnorm{1}} ) \\
			 & ~~~~ \times \bigl( \norm{\nablah Z}{\Hnorm{1}}^2 + \norm{\nablah v}{\Hnorm{2}}^2 
			 + \norm{\dz v}{\Hnorm{1}}^2 
			 + \norm{z^{\alpha/2} v_t}{\Lnorm{2}}^2 \\
			 & ~~~~ + \norm{\nabla v_t}{\Lnorm{2}}^2 \bigr).
	\end{aligned}\end{equation}
	Then integrating this inequality with respect to the temporal variable, together with \eqref{SPE-005} completes the proof of \eqref{apriori:dsptn-dhh-z}.
\end{proof}

\subsection{Stability}

This section is devoted to show Proposition \ref{prop:stability-theory}. Throughout the following arguments, we use $ 0< C <\infty $ to denote a positive constant independent of the solution $ (Z,v) $, which depends on the initial data and might be different from line to line. Summing up \eqref{apriori:conserv-kinetic-energy}, \eqref{apriori:dissipation-Z-t}, \eqref{apriori:dissipation-Z-t-L3}, \eqref{apriori:H^1-v-est}, \eqref{apriori:temporal-derivative}, \eqref{apriori:tangential-01}, \eqref{apriori:tangential-02}, \eqref{apriori:dsptn-dh-Z}, and \eqref{apriori:dsptn-dhh-z}, yields the inequality
\begin{equation}\label{SPE-006}
	\begin{aligned}
		& \sup_{0\leq t\leq T} \mathfrak E_c(t) + \int_0^T \mathfrak D(t) \,dt \leq C({\Arrowvert Z \Arrowvert_{\Lnorm{\infty}},\Arrowvert Z^{-1} \Arrowvert_{\Lnorm{\infty}}}) \mathcal E_0 + \int_0^T \mathcal H (\mathcal E(t)) \\
		& ~~~~ \times \bigl( \norm{Z-1}{\Hnorm{2}}^2 + \norm{Z_t}{\Lnorm{2}}^2 + \norm{\nabla v}{\Hnorm{1}}^2 + \norm{\nablah v}{\Hnorm{2}}^2 \\
		& ~~~~ + \norm{z^{\alpha/2}v_t}{\Lnorm{2}}^2 + \norm{\nabla v_t}{\Lnorm{2}}^2 \bigr),\dt,
	\end{aligned}
\end{equation}
where $ 0 < C({\Arrowvert Z \Arrowvert_{\Lnorm{\infty}},\Arrowvert Z^{-1} \Arrowvert_{\Lnorm{\infty}}}) < \infty $ is a constant depending on $ \Arrowvert Z \Arrowvert_{\Lnorm{\infty}},\Arrowvert Z^{-1} \Arrowvert_{\Lnorm{\infty}} < \infty $. Then, if $ \sup_{0\leq t\leq T} \mathcal E(t) < \varepsilon_2 < \varepsilon_1 $ for some sufficiently small $ 0 < \varepsilon_2 < \varepsilon_1 $ with $ \varepsilon_1 $ given in Proposition \ref{lm:equal-of-energies}, one has that
\begin{align*}
	&  \int_0^T \mathcal H (\mathcal E(t))\bigl( \norm{Z-1}{\Hnorm{2}}^2 + \norm{Z_t}{\Lnorm{2}}^2 + \norm{\nabla v}{\Hnorm{1}}^2 + \norm{\nablah v}{\Hnorm{2}}^2 \\
	& ~~~~ + \norm{z^{\alpha/2}v_t}{\Lnorm{2}}^2 + \norm{\nabla v_t}{\Lnorm{2}}^2 \bigr)\,dt \leq \dfrac{1}{2} \int_0^T \mathfrak D(t) \,dt,\\
	\text{and} \quad & | Z - 1 | \leq C \varepsilon_1^{1/2} < 1/2,
\end{align*}
where we have applied \eqref{aprasm:Z} and  \eqref{aprasm:dissipation}. Therefore \eqref{SPE-006} yields,
\begin{equation*}
	\sup_{0\leq t\leq T} \mathfrak E_c(t) + \dfrac{1}{2} \int_0^T \mathfrak D(t) \,dt \leq C \mathcal E_0. 
\end{equation*}
Furthermore, taking $ c = \frac{\mu}{8g} $ as in Proposition \ref{lm:equal-of-energies}, together with \eqref{aprasm:energy} and \eqref{aprasm:dissipation}, will finish the proof of Proposition \ref{prop:stability-theory}.

\section{Regularity up to the vacuum boundary}\label{sec:regularity-1}

In this section, we study the regularity of $ v $ near the vacuum boundary $ z= 0 $. We will make use of the viscosity to obtain the $ L^\infty(0,T; H^2(\Omega)) $ estimates for $ v $. In fact, we will show the following proposition:
\begin{proposition}\label{prop:H^2-of-v}
Under the same assumptions as in Proposition \ref{prop:stability-theory}, 
there is a constant $  \varepsilon_3 > 0 $ 
such that if $ \sup_{0\leq t \leq T} \mathcal E(t) \leq \varepsilon_3 $, then the following inequality holds:
	\begin{equation}\label{apriori:H^2-v}
		\norm{v(t)}{\Hnorm{2}} \leq C \bigl( \mathcal E(t) + 1 \bigr) \bigl( \mathcal E(t) \bigr)^{1/2},
	\end{equation}
	for some constant $ 0< C <\infty $ and any $ 0\leq t\leq T $. 
\end{proposition}
\begin{proof}
	Directly from \subeqref{FB-CPE}{2}, one has
\begin{equation}\label{RG-VacuumBD}
	\begin{aligned}
		& \norm{\mu \deltah v + (\mu+\lambda) \nablah \dvh v + \mu \partial_{zz} v}{\Lnorm{2}} \\
		& ~~~~ \lesssim \norm{\mu (\nablah \log Z)\cdot \nablah v + (\mu+\lambda)(\dvh v) \nablah \log Z }{\Lnorm{2}} \\
		& ~~~~ + \norm{z^\alpha Z^\alpha \nablah Z}{\Lnorm{2}} + \norm{z^\alpha Z^\alpha \dt v}{\Lnorm{2}} + \norm{z^\alpha Z^\alpha v\cdot\nablah v}{\Lnorm{2}} \\
		& ~~~~ + \norm{z^{\alpha} Z^\alpha W \dz v}{\Lnorm{2}}. 
	\end{aligned}
\end{equation}
On the one hand, applying integration by parts in \eqref{RG-VacuumBD} yields 
\begin{align*}
	& \norm{\mu \deltah v + (\mu+\lambda) \nablah \dvh v + \mu \partial_{zz} v}{\Lnorm{2}}^2 = \mu^2 \norm{\nablah^2 v}{\Lnorm{2}}^2 + \mu^2 \norm{\partial_{zz}v}{\Lnorm{2}}^2\\
	& ~~~~ + (\mu+\lambda)^2 \norm{\nablah \dvh v}{\Lnorm{2}}^2 + 2\mu (\mu+\lambda) \norm{\nabla \dvh v}{\Lnorm{2}}^2\\
	& ~~~~ + 2\mu^2 \norm{\nablah \dz v}{\Lnorm{2}}^2 \gtrsim \norm{\nabla^2 v}{\Lnorm{2}}^2. 
\end{align*}
On the other hand, we will show that the right-hand side of \eqref{RG-VacuumBD} can be hounded by $ \mathcal E(t) $. In fact, after substituting \eqref{eq:verticalvelocity}, one can check that
\begin{align*}
	& \norm{z^\alpha Z^\alpha W\dz v}{\Lnorm{2}}^2 \lesssim \int \biggl\lbrack z^{2\alpha+2} \bigl( \abs{\overline{z^\alpha v}}{2} \abs{\nablah Z^\alpha}{2} + \abs{\overline{z^\alpha\dvh v}}{2} \abs{Z}{2\alpha}  \bigr)\abs{\dz v}{2} \\
	& ~~~~ + \bigl(\int_0^z \xi^\alpha v \,d\xi\bigr)^2 \abs{\nablah Z^\alpha}{2} \abs{\dz v}{2} + \bigl( \int_0^z \xi^\alpha \dvh v \,d\xi \bigr)^2 \abs{Z}{2\alpha} \abs{\dz v}{2} \biggr\rbrack \idx\\
	& ~~~~ \lesssim \int_0^1 \biggl\lbrack \bigl( \hnorm{\overline{z^\alpha v}}{\Lnorm{\infty}}^2 \hnorm{Z^{2\alpha-2}}{\Lnorm{\infty}} \hnorm{\nablah Z}{\Lnorm{4}}^2 + \hnorm{Z}{\Lnorm{\infty}}^{2\alpha} \hnorm{\overline{z^\alpha\dvh v}}{\Lnorm{4}}^2\bigr)  \hnorm{z^{\alpha/2}\dz v}{\Lnorm{4}}^2 \\
	& ~~~~ + \bigl( \int_0^z \xi^\alpha \hnorm{v}{\Lnorm{\infty}}^2 \,d\xi \hnorm{Z^{2\alpha-2}}{\Lnorm{\infty}} \hnorm{\nablah Z}{\Lnorm{4}}^2 +  \int_0^z \xi^\alpha \hnorm{\nablah v}{\Lnorm{4}}^2 \,d\xi \hnorm{Z}{\Lnorm{\infty}}^{2\alpha}  \bigr) \\
	& ~~~~ \times \hnorm{z^{\alpha/2}\dz v}{\Lnorm{4}}^2 \biggr\rbrack \,dz \lesssim \bigl( \norm{Z^{2\alpha-2}}{\Lnorm{\infty}} + \norm{Z}{\Lnorm{\infty}}^{2\alpha} \bigr)\norm{z^{\alpha/2} v}{\hHnorm{2}}^2 \bigl( \norm{\nablah Z}{\Hnorm{1}}^2 \\
	& ~~~~ + 1 \bigr) \norm{z^{\alpha/2}v_z }{\hHnorm{1}}^2 \leq C_{\Arrowvert Z \Arrowvert_{\Lnorm{\infty}}, \Arrowvert Z^{-1} \Arrowvert_{\Lnorm{\infty}} } \bigl( \bigl(\mathcal E(t) \bigr)^3 +\bigl( \mathcal E(t) \bigr)^2 \bigr) .
\end{align*}
Similarly, the estimates for the rest terms in \eqref{RG-VacuumBD} are listed in the following:
\begin{align*}
	& \norm{\mu \nablah \log Z\cdot \nablah v + (\mu+\lambda)\dvh v \nablah \log Z }{\Lnorm{2}} \lesssim \norm{Z^{-1}}{\Lnorm{\infty}} \\
	& ~~~~ \times \norm{\nablah Z}{\Hnorm{1}} \norm{\nablah v}{\hHnorm{1}} \leq C({\Arrowvert Z^{-1} \Arrowvert_{\Lnorm{\infty}}}) \bigl( \mathcal E(t) \bigr)^{1/2} \norm{\nablah v}{\hHnorm{1}} , \\
	& \norm{z^\alpha Z^\alpha \nablah Z}{\Lnorm{2}} + \norm{z^\alpha Z^\alpha \dt v}{\Lnorm{2}} \leq C({\Arrowvert Z \Arrowvert_{\Lnorm{\infty}}}) \bigl(\mathcal E(t) \bigr)^{1/2} ,\\
	& \norm{z^\alpha Z^\alpha v\cdot\nablah v}{\Lnorm{2}} \leq C({\Arrowvert Z \Arrowvert_{\Lnorm{\infty}}}) \norm{v}{\Lnorm{3}} \norm{\nablah v}{\Lnorm{6}} \leq  C({\Arrowvert Z \Arrowvert_{\Lnorm{\infty}}})\norm{v}{\Hnorm{1}} \\
	& ~~~~ \times \norm{\nablah v}{\Hnorm{1}} \leq C({\Arrowvert Z \Arrowvert_{\Lnorm{\infty}}})  \bigl(\mathcal E(t) \bigr)^{1/2} \norm{\nablah v}{\Hnorm{1}}.
\end{align*}
Therefore, \eqref{RG-VacuumBD} yields 
\begin{equation*}
	\norm{\nabla^2 v}{\Lnorm{2}} \leq C({\Arrowvert Z \Arrowvert_{\Lnorm{\infty}}, \Arrowvert Z^{-1} \Arrowvert_{\Lnorm{\infty}}}) \bigl(\mathcal E(t) \bigr)^{1/2} \bigl( \mathcal E(t) + \norm{\nablah v}{\Hnorm{1}} + 1 \bigr).
\end{equation*}
This concludes the proof of \eqref{apriori:H^2-v} after choosing $ \sup_{0\leq t\leq T} \mathcal E(t) < \varepsilon_3 $ sufficiently small. 
\end{proof}

Next, we will derive estimate for $ \partial_{zzz} v $. 
After applying $ \dz $ to \subeqref{FB-CPE}{2}, one has the following equation:
\begin{equation}\label{eq:dzzz-v}
	\begin{aligned}
		& \mu \partial_{zzz} v = - \mu \deltah v_z - (\mu + \lambda) \nablah \dvh v_z - \mu \nablah \log Z \cdot \nablah v_z \\
		& ~~~~ - (\mu+\lambda) \dvh v_z \nablah \log Z + \alpha z^{\alpha-1} Z^\alpha \bigl(  \dt v + v\cdot \nablah v \bigr)\\ 
		& ~~~~ + z^\alpha Z^\alpha \bigl( \dt v_z + v_z \cdot \nablah v + v \cdot \nablah v_z + W \dz v_z \bigr)  + \dz (z^\alpha W) Z^\alpha \dz v \\
		& ~~~~ + \alpha g z^{\alpha-1} Z^{\alpha} \nablah Z.
	\end{aligned}
\end{equation}
For the sake of convenience, in the following arguments, we recall again assumption \eqref{apriori-boundness-of-Z}, i.e., that $ 1/2 \leq \abs{Z-1}{} \leq 2 $. 
Hence we have
\begin{align*}
	& \norm{\mu \deltah v_z + (\mu+\lambda) \nablah \dvh v_z}{\Lnorm{2}} \lesssim \norm{\nablah v}{\Hnorm{2}},\\
	& \norm{\mu \nablah \log Z \cdot \nablah v_z + (\mu+\lambda) \dvh v_z \nablah \log Z}{\Lnorm{2}} \lesssim \norm{\nablah Z}{\Hnorm{1}} \norm{\nablah v}{\Hnorm{2}} \\
	& ~~~~ \lesssim \bigl(\mathcal E(t)\bigr)^{1/2} \norm{\nablah v}{\Hnorm{2}},\\
	& \norm{z^\alpha Z^\alpha \dt v_z}{\Lnorm{2}} \lesssim \norm{\nabla v_t}{\Lnorm{2}}, \\
	& \norm{z^\alpha Z^\alpha v_z \cdot \nablah v}{\Lnorm{2}} \lesssim \norm{z^{\alpha/2}v_z}{\Lnorm{2}} \norm{\nablah v}{\Lnorm{\infty}} \lesssim \bigl(\mathcal E(t)\bigr)^{1/2} \norm{\nablah v}{\Hnorm{2}} ,\\
	& \norm{z^\alpha Z^\alpha v \cdot \nablah v_z}{\Lnorm{2}} \lesssim \norm{v}{\Lnorm{3}} \norm{\nablah v_z}{\Lnorm{6}} \lesssim \norm{v}{\Hnorm{1}} \norm{\nablah v}{\Hnorm{2}} \\
	& ~~~~ \lesssim \bigl(\mathcal E(t)\bigr)^{1/2} \norm{\nablah v}{\Hnorm{2}}.
\end{align*}
On the other hand, after substituting \eqref{eq:dz-W} into the following, one can easily check that 
\begin{align*}
	& \norm{\dz (z^\alpha W) Z^\alpha \dz v}{\Lnorm{2}} \lesssim \bigl( \norm{\overline{z^\alpha v}}{\Lnorm{\infty}}\norm{\nablah Z }{\Lnorm{3}} + \norm{\overline{z^\alpha \dvh v}}{\Lnorm{3}} \bigr) \norm{\dz v}{\Lnorm{6}}\\
	& ~~~~ +  \norm{v}{\Lnorm{6}} \norm{\nablah Z}{\Lnorm{6}} \norm{\dz v}{\Lnorm{6}} + \norm{\nablah v}{\Lnorm{\infty}}\norm{z^{\alpha/2}\dz v}{\Lnorm{2}} \\
	& \lesssim \bigl( \norm{\nablah Z}{\Hnorm{1}} + 1\bigr) \norm{z^{\alpha/2} v}{\hHnorm{2}} \norm{\dz v}{\Hnorm{1}} + \norm{v}{\Hnorm{1}} \norm{\nablah Z}{\Hnorm{1}} \norm{\dz v}{\Hnorm{1}} \\
	& ~~~~ + \norm{z^{\alpha/2} v_z}{\Lnorm{2}} \norm{\nablah v}{\Hnorm{2}} \lesssim \bigl( \mathcal E(t) + \bigl( \mathcal E(t) \bigr)^{1/2} \bigr)\bigl( \norm{\nabla v}{\Hnorm{1}} + \norm{\nablah v}{\Hnorm{2}} \bigr) .
\end{align*}
Moreover, after using \eqref{eq:verticalvelocity}, as before, one concludes that
\begin{align*}
	& \norm{z^\alpha Z^\alpha W\dz v_z}{\Lnorm{2}}^2 \lesssim \int \biggl(  z^{2\alpha+2} \bigl( \abs{\overline{z^\alpha v}}{2} \abs{\nablah Z^\alpha}{2} + \abs{\overline{z^\alpha\dvh v}}{2} \abs{Z}{2\alpha}  \bigr)\abs{\dz v_z}{2} \\
	& ~~~~ + \bigl(\int_0^z \xi^\alpha v \,d\xi\bigr)^2 \abs{\nablah Z^\alpha}{2} \abs{\dz v_z}{2} + \bigl( \int_0^z \xi^\alpha \dvh v \,d\xi \bigr)^2 \abs{Z}{2\alpha} \abs{\dz v_z}{2} \biggr) \idx\\
	& ~~~~ \lesssim \int_0^1 \biggl( \bigl( \hnorm{\overline{z^\alpha v}}{\Lnorm{\infty}}^2 \hnorm{Z^{2\alpha-2}}{\Lnorm{\infty}} \hnorm{\nablah Z}{\Lnorm{4}}^2 + \hnorm{Z}{\Lnorm{\infty}}^{2\alpha} \hnorm{\overline{z^\alpha\dvh v}}{\Lnorm{4}}^2\bigr)  \hnorm{\dz v_z}{\Lnorm{4}}^2 \\
	& ~~~~ + \bigl\lbrack \bigl( \int_0^z \xi^\alpha \hnorm{v}{\Lnorm{\infty}}^2 \,d\xi \bigr) \hnorm{Z^{2\alpha-2}}{\Lnorm{\infty}} \hnorm{\nablah Z}{\Lnorm{4}}^2 + \bigl( \int_0^z \xi^\alpha \hnorm{\nablah v}{\Lnorm{4}}^2 \,d\xi \bigr) \hnorm{Z}{\Lnorm{\infty}}^{2\alpha}  \bigr\rbrack \\
	& ~~~~ \times \hnorm{\dz v_z}{\Lnorm{4}}^2 \biggr) \,dz \lesssim \norm{z^{\alpha/2} v}{\hHnorm{2}}^2 \bigl( \norm{\nablah Z}{\Hnorm{1}}^2 + 1 \bigr) \bigl( \norm{\partial_{zz}v}{\Lnorm{2}}^2 + \norm{\nablah v}{\Hnorm{2}}^2 \bigr) \\
	& ~~~~ \lesssim \bigl( \bigl(\mathcal E(t) \bigr)^2 + \mathcal E(t) \bigr) \bigl( \norm{\nabla v}{\Hnorm{1}}^2 + \norm{\nablah v}{\Hnorm{2}}^2 \bigr).
\end{align*}
Next, if $ \alpha > 1/2 $, (or equivalently $ 2 \alpha - 2 > -1 $), i.e., $ 1< \gamma < 3 $, we have,
\begin{align*}
	& \norm{z^{\alpha-1} Z^\alpha \nablah Z}{\Lnorm{2}}^2 = \int_0^1 z^{2\alpha-2} \,dz \times \hnorm{Z^\alpha \nablah Z }{\Lnorm{2}}^2 \lesssim \norm{\nablah Z}{\Lnorm{2}}^2, \\
	& \norm{z^{\alpha-1} Z^\alpha \dt v}{\Lnorm{2}}^2 \lesssim \int_0^1 z^{2\alpha - 2} \hnorm{\dt v}{\Lnorm{2}}^2\,dz \lesssim \norm{z^{\alpha/2} \dt v}{\Lnorm{2}}^2 + \norm{\nabla \dt v}{\Lnorm{2}}^2,\\
	& \norm{z^{\alpha-1} Z^\alpha v\cdot \nablah v}{\Lnorm{2}}^2 \lesssim \int_0^1 z^{2\alpha-2} \hnorm{v\cdot \nablah v}{\Lnorm{2}}^2\,dz \lesssim \int_0^1 z^{2\alpha} \bigl(\hnorm{v\cdot \nablah v}{\Lnorm{2}}^2 \\
	& ~~~~ + \hnorm{\dz (v\cdot \nablah v)}{\Lnorm{2}}^2 \bigr)\,dz \lesssim \norm{v \cdot \nablah v}{\Lnorm{2}}^2 + \norm{\dz v \cdot \nablah v}{\Lnorm{2}}^2 \\
	& ~~~~ + \norm{v \cdot \nablah v_z}{\Lnorm{2}}^2 \lesssim \norm{v}{\Lnorm{2}}^2 \norm{\nablah v}{\Lnorm{\infty}}^2 + \norm{\dz v}{\Lnorm{2}}^2 \norm{\nablah v}{\Lnorm{\infty}}^2 \\
	& ~~~~ + \norm{v}{\Lnorm{3}}^2 \norm{\nablah v_z}{\Lnorm{6}}^2 \lesssim \norm{v}{\Hnorm{1}}^2  \norm{\nablah v}{\Hnorm{2}}^2 \lesssim \mathcal E(t) \norm{\nablah v}{\Hnorm{2}}^2,
\end{align*}
where we have applied H\"older's, Hardy's (Lemma \ref{lm:hardy}), and the Sobolev embedding inequalities. 
Therefore, after summing up the above inequalities, one has, for $ \alpha > 1/2 $, 
\begin{align*}
	& \norm{\partial_{zzz} v}{\Lnorm{2}} \lesssim
	\norm{\nablah v}{\Hnorm{2}} + \norm{z^{\alpha/2} \dt v}{\Lnorm{2}} + \norm{\nabla v_t}{\Lnorm{2}} + \norm{\nablah Z}{\Lnorm{2}} \\
	& ~~~~ + \bigl( \mathcal E(t) + \bigl(\mathcal E(t)\bigr)^{1/2} \bigr) \bigl( \norm{\nabla v}{\Hnorm{1}} + \norm{\nablah v}{\Hnorm{2}} \bigr).
\end{align*}
On the other hand, if $ 0 < \alpha \leq 1/2 $, i.e., $ \gamma \geq 3 $, let $ \beta $ be any number such that $ 2\alpha + 2\beta - 2 > -1 $. Then the same arguments as above show that 
\begin{align*}
	& \norm{z^{\beta}\partial_{zzz} v}{\Lnorm{2}} \lesssim
	\norm{\nablah v}{\Hnorm{2}} + \norm{z^{\alpha/2} \dt v}{\Lnorm{2}} + \norm{\nabla v_t}{\Lnorm{2}} + \norm{\nablah Z}{\Lnorm{2}} \\
	& ~~~~ + \bigl( \mathcal E(t) + \bigl(\mathcal E(t)\bigr)^{1/2} \bigr) \bigl( \norm{\nabla v}{\Hnorm{1}} + \norm{\nablah v}{\Hnorm{2}} \bigr).
\end{align*}
Therefore, we have demonstrated the following proposition: 
\begin{proposition}\label{lm:diss-dzzz-v}
Under the same assumptions as in Proposition \ref{prop:stability-theory}, one has
	\begin{equation}\label{apriori:diss-dzzz-v-est}
		\begin{aligned}
			& \bullet \text{if $ \alpha > 1/2 $, i.e., $ 1 < \gamma < 3 $}, ~ \int_0^T \norm{\partial_{zzz} v}{\Lnorm{2}}^2 \,dt \leq \int_0^T \biggl\lbrack \norm{\nablah v}{\Hnorm{2}}^2 \\
			& ~~~~ + \norm{z^{\alpha/2} \dt v}{\Lnorm{2}}^2 + \norm{\nabla v_t}{\Lnorm{2}}^2 + \norm{\nablah Z}{\Lnorm{2}}^2 \\
			& ~~~~ + \bigl( \bigl(\mathcal E(t)\bigr)^2 + \mathcal E(t) \bigr) \bigl( \norm{\nabla v}{\Hnorm{1}}^2
			 + \norm{\nablah v}{\Hnorm{2}}^2 \bigr) \biggr\rbrack \,dt;\\
			& \bullet \text{if $ 0 < \alpha\leq 1/2 $, i.e., $ \gamma \geq 3 $}, ~ \int_0^T \norm{z^{\beta}\partial_{zzz} v}{\Lnorm{2}}^2  \,dt \leq \int_0^T \norm{\nablah v}{\Hnorm{2}}^2 \\
			& ~~~~ + \norm{z^{\alpha/2} \dt v}{\Lnorm{2}}^2
			 + \norm{\nabla v_t}{\Lnorm{2}}^2 + \norm{\nablah Z}{\Lnorm{2}}^2 + \bigl( \bigl(\mathcal E(t)\bigr)^2 + \mathcal E(t) \bigr) \\
			 & ~~~~ \times \bigl( \norm{\nabla v}{\Hnorm{1}}^2
			+ \norm{\nablah v}{\Hnorm{2}}^2 \bigr) \,dt,
		\end{aligned}
	\end{equation}
	for any $ \beta > \frac{1}{2} - \alpha $. In particular,
	\begin{equation}\label{apriori:diss-dzzz-v-est-02}
		\begin{aligned}
			&  \int_0^T \norm{z^{1/2}\partial_{zzz} v}{\Lnorm{2}}^2 \,dt \leq \int_0^T \biggl\lbrack \norm{\nablah v}{\Hnorm{2}}^2 + \norm{z^{\alpha/2} \dt v}{\Lnorm{2}}^2\\
			& ~~~~ + \norm{\nabla v_t}{\Lnorm{2}}^2 + \norm{\nablah Z}{\Lnorm{2}}^2 + \bigl( \bigl(\mathcal E(t)\bigr)^2 + \mathcal E(t) \bigr) \bigl( \norm{\nabla v}{\Hnorm{1}}^2\\
			& ~~~~ + \norm{\nablah v}{\Hnorm{2}}^2 \bigr) \biggr\rbrack \,dt. 
		\end{aligned}
	\end{equation}
\end{proposition}

\section{Global well-posedness and stability of the stationary equilibrium state}\label{sec:global-stability}
In this section, we will establish the global well-posedness and the stability of the stationary equilibrium state of the compressible primitive equation \eqref{FB-CPE}. 

\subsection{Local well-posedness theory}\label{sec:local}
We first show the local-in-time well-posedness theory of \eqref{FB-CPE} with the given initial data $ (Z_0, v_0) \in H^2(\Omega)\times H^2(\Omega) $. For this purpose, we introduce the following ``linearized'' system associated with \eqref{FB-CPE}, given $ (\hat Z, \hat v) $ smooth enough, to be specified later:
\begin{equation}\label{eq:linearized-FBCPE}
	\begin{cases}
		\dt Z + (\alpha+1) \overline{z^\alpha \hat v} \cdot \nablah Z + \overline{z^\alpha\dvh \hat v} \hat Z = 0, \\
		(z^\alpha {\hat Z}^{\alpha} + \eta )\dt v + z^\alpha {\hat Z}^\alpha \bigl( \hat v \cdot \nablah \hat v + \hat W \dz \hat v + g \nablah \hat Z \bigr) = \mu \deltah v \\
		~~~~ + (\mu+\lambda) \nablah \dvh v + \mu \partial_{zz} v + \mu \nablah \log \hat Z \cdot \nablah \hat v \\
		~~~~ + (\mu+\lambda) \dvh \hat v \nablah \log \hat Z ,\\
		\dz Z =0,
	\end{cases}
\end{equation}
for any fixed positive constant $ \eta \in (0,1) $, where $ \hat W $ is given by
\begin{equation}\label{eq:linearized-vertical-velocity}
	\begin{aligned}
		& z^\alpha {\hat Z} \hat W = z^{\alpha+1}\bigl( (\alpha+1) \overline{z^\alpha \hat v} \cdot \nablah \hat Z + \overline{z^\alpha \dvh \hat v} \hat Z\bigr)\\
		& ~~~~ - (\alpha+1) \bigl(\int_0^z \xi^\alpha \hat v \,d\xi\bigr) \cdot \nablah \hat Z - \bigl(\int_0^z \xi^\alpha \dvh \hat v \,d\xi\bigr) \hat Z.
	\end{aligned}
\end{equation}
Also, the following boundary condition is imposed
$$ \dz v|_{z=0,1} = 0. $$
In the following, we will study \eqref{eq:linearized-FBCPE} with the given approximating initial data $ (Z,v)\big|_{t=0} = (Z_{0,\eta}, v_{0,\eta}) $. The approximating initial data satisfies
\begin{gather*}
	\norm{Z_{0,\eta}-1}{\Hnorm{2}}, \norm{v_{0,\eta}}{\Hnorm{2}} < \infty,\\
	(Z_{0,\eta}, v_{0,\eta}) \rightarrow (Z_0,v_0), ~~ \text{as $ \eta \rightarrow 0^+ $ in $ H^2(\Omega) \times H^2(\Omega) $}.
\end{gather*}
Also we denote the approximating total initial energy as
\begin{equation}\label{linearized:ini-energy}
	\begin{aligned}
		& \mathcal E_{0,\eta} : = \norm{(z^{\alpha}+ \eta)^{1/2} v_{0,\eta}}{\hHnorm{2}}^2 + \norm{(z^\alpha +\eta)^{1/2} v_{1,\eta}}{\Lnorm{2}}^2\\
		& ~~~~ + \norm{(z^\alpha+\eta)^{1/2} v_{0,\eta,z}}{\hHnorm{1}}^2 + \norm{v_{0,\eta}}{\Hnorm{1}}^2  + \norm{Z_{0,\eta}-1}{\Hnorm{2}}^2 \\
		& ~~~~ + \norm{Z_{1,\eta}}{\Lnorm{2}}^2 \lesssim \mathcal E_0 < \infty,
	\end{aligned}
\end{equation}
where $ v_{1,\eta} = v_{t}|_{t=0}, Z_{1,\eta} = Z_{t}|_{t=0} $ are given by
\begin{align*}
	&  Z_{1,\eta} + (\alpha+1) \overline{z^\alpha v_{0,\eta}} \cdot \nablah Z_{0,\eta} + \overline{z^\alpha\dvh v_{0,\eta}} Z_{0,\eta} = 0, \\
	&	(z^\alpha Z_{0,\eta}^{\alpha} + \eta )v_{1,\eta} + z^\alpha Z_{0,\eta}^\alpha \bigl( v_{0,\eta} \cdot \nablah v_{0,\eta} +  W_{0,\eta} \dz  v_{0,\eta} + g \nablah  Z_{0,\eta} \bigr) \\
	& ~~~~= \mu \deltah v_{0,\eta} 
	+ (\mu+\lambda) \nablah \dvh v_{0,\eta} + \mu \partial_{zz} v_{0,\eta} + \mu \nablah \log Z_{0,\eta} \cdot \nablah v_{0,\eta}\\
	& ~~~~ + (\mu+\lambda) \dvh v_{0,\eta} \nablah \log Z_{0,\eta}, \\
	& z^\alpha Z_{0,\eta} W_{0,\eta} = z^{\alpha+1} \bigl( (\alpha+1) \overline{z^\alpha v_{0,\eta}} \cdot \nablah Z_{0,\eta} + \overline{z^\alpha \dvh v_{0,\eta}} Z_{0,\eta}\bigr)  \nonumber \\
	& ~~~~ - (\alpha+ 1) \int_0^z \xi^\alpha v_{0,\eta}\,d\xi \cdot \nablah Z_{0,\eta} - \int_0^z \xi^\alpha \dvh v_{0,\eta} \,d\xi Z_{0,\eta}.
\end{align*}
Recall that $ \mathcal E_0 $ is given in \eqref{STB-ttl-ini-energy}.
Then as $ \eta \rightarrow 0^+ $, it is easy to verify $ \mathcal E_{0,\eta} \rightarrow \mathcal E_0 $. Accordingly, the approximating total energy functional is defined as,
\begin{equation}\label{linearized:tot-energy}
	\begin{aligned}
		& \mathcal E_{\eta}(t): = \norm{(z^{\alpha}+ \eta)^{1/2} v}{\hHnorm{2}}^2 + \norm{(z^\alpha +\eta)^{1/2} v_{t}}{\Lnorm{2}}^2 + \norm{(z^\alpha+\eta)^{1/2} v_{z}}{\hHnorm{1}}^2 \\
		& ~~~~ + \norm{v}{\Hnorm{1}}^2  + \norm{Z-1}{\Hnorm{2}}^2 
		 + \norm{Z_{t}}{\Lnorm{2}}^2. 
	\end{aligned}
\end{equation}

\bigskip

We will establish the local well-posedness of \eqref{FB-CPE} by a fixed point argument. In fact, the well-posedness of the ``linearized'' system \eqref{eq:linearized-FBCPE} follows from the standard existence theory of hyperbolic-parabolic systems with $ (\hat Z, \hat v) $ given and smooth enough (cf. \cite{EvansPDE}). In the following, we show some estimates for the solution $ (Z,v) $ to \eqref{eq:linearized-FBCPE} in a small time interval which are independent of $ \eta $; and show that the map $ \Phi : (\hat Z, \hat v) \mapsto (Z, v) $ has a fixed point $ (Z_\eta, v_\eta) $ in some functional space via the Tychonoff fixed point theorem. Then taking the limit $ \eta\rightarrow 0^+ $ yields the existence of local solution to \eqref{FB-CPE}, since our estimates are independent of $ \eta $. Such a procedure is similar to that presented in our previous work \cite{LT2018a}. In the following, we will only sketch the estimates and show that they are independent of $ \eta $. In fact, we will work in the following functional space:
\begin{equation}\label{ln:functional-space}
	\begin{aligned}
		& X_T : = \bigl\lbrace (Z,v)\in H^2(\Omega) \times H^2(\Omega) | (Z,v)|_{t=0} = (Z_{0,\eta},v_{0,\eta}) , ~  \abs{Z-1}{} \leq \frac{1}{2}, \\
		& ~~~~ \sup_{0\leq t \leq T} \mathcal E_\eta(t) + \int_0^T \norm{\nabla v}{\Hnorm{1}}^2 + \norm{\nablah v}{\Hnorm{2}}^2 + \norm{\nabla v_t}{\Lnorm{2}}^2 \,dt \\
		& ~~~~  \leq C_0 \mathcal E_{0,\eta} 
		\bigr\rbrace,
	\end{aligned}
\end{equation}
for some $ T, C_0 \in (0,\infty) $ to be specified later. Note that, $ X_T $ is convex and compact in the $ L^2(0,T;L^2(\Omega)) \times L^2(0,T;L^2(\Omega)) $ topology. We will also require $ \mathcal E_{0,\eta}$ to be small. Let $ (\hat Z, \hat v) \in X_T $ and $ (Z,v) $ be the corresponding solution to \eqref{eq:linearized-FBCPE} which is guaranteed to exist by \cite{EvansPDE}. For the sake of convenience, we denote the total energy functional for $ (\hat Z, \hat v) $ as
\begin{equation*}
	\begin{aligned}
		& \mathcal{\hat E}_{\eta}:= \norm{(z^{\alpha}+ \eta)^{1/2} \hat v}{\hHnorm{2}}^2 + \norm{(z^\alpha +\eta)^{1/2} \hat v_{t}}{\Lnorm{2}}^2 + \norm{(z^\alpha+\eta)^{1/2} \hat v_{z}}{\hHnorm{1}}^2 \\
		& ~~~~ + \norm{\hat v}{\Hnorm{1}}^2  + \norm{\hat Z-1}{\Hnorm{2}}^2 
		 + \norm{\hat Z_{t}}{\Lnorm{2}}^2. 
	\end{aligned}
\end{equation*}
Also, we will frequently use \eqref{SPE-007}, \eqref{SPE-008}, and the following inequality, which resembles \eqref{STB-001} and \eqref{TE-002},
\begin{gather}\label{ineq:poincare-type-001}
	\norm{f}{\Lnorm{2}}^2 \lesssim \norm{z^{\alpha/2}f}{\Lnorm{2}}^2 + \norm{\nabla f}{\Lnorm{2}}^2. 
\end{gather}
\subsubsection*{Estimates for $ Z $}
We start with the estimates for $ Z $ by presenting the $ H^2 $ estimates for $ Z-1 $. First, multiplying \subeqref{eq:linearized-FBCPE}{1} with $ 2(Z-1) $  and integrating the resultant yield
\begin{align*}
	& \dfrac{d}{dt} \norm{Z-1}{\Lnorm{2}}^2 = (\alpha+1) \int \overline{z^\alpha \dvh \hat v} \abs{Z-1}{2} \idx - 2\int \overline{z^\alpha \dvh \hat v} \hat Z (Z-1)\idx \\
	& ~~~~ \lesssim \norm{\overline{z^{\alpha} \dvh \hat v}}{\Lnorm{\infty}} \bigl( \norm{Z-1}{\Lnorm{2}}^2 + 1 \bigr) \lesssim \norm{\nablah \hat v}{\Hnorm{2}}\norm{Z-1}{\Lnorm{2}}^2 \\
	& ~~~~ + \norm{\nablah \hat v}{\Hnorm{2}}.
\end{align*}
Then applying Gr\"onwall's inequality implies
\begin{equation}\label{linear:001}
	\begin{aligned}
	& \sup_{0<t<T}\norm{Z-1}{\Lnorm{2}}^2 \leq e^{C \int_0^T \norm{\nablah \hat v}{\Hnorm{2}}\,dt}\bigl( \norm{Z_{0,\eta}-1}{\Lnorm{2}}^2 \\
	& ~~~~ + \int_0^T \norm{\nablah \hat v}{\Hnorm{2}}\,dt \bigr) \leq e^{CC_0^{1/2}\mathcal E_{0,\eta}^{1/2} T^{1/2}} \bigl( \mathcal E_{0,\eta} + C_0^{1/2}\mathcal E_{0,\eta}^{1/2} T^{1/2} \bigr)\\
	& ~~~~ \leq 2 \mathcal E_{0,\eta},
	\end{aligned}
\end{equation}
provided $ T $ small enough. On the other hand, after applying $ \partial_{hh} $  to \subeqref{eq:linearized-FBCPE}{1}, it holds
\begin{align*}
	& \dt Z_{hh} + (\alpha+1) \overline{z^\alpha \hat v} \cdot \nablah Z_{hh} + 2 (\alpha+1) \overline{z^\alpha \hat v_h} \cdot \nablah Z_h\\
	& ~~~~ + (\alpha+1) \overline{z^\alpha \hat v_{hh}} \cdot \nablah Z
	 + \overline{z^\alpha \dvh \hat v} \hat Z_{hh} + 2 \overline{z^\alpha \dvh \hat v_h} \hat Z_{h}\\
	 & ~~~~ + \overline{z^\alpha \dvh \hat v_{hh}} \hat Z=0.
\end{align*}
Taking the $ L^2 $-inner product of the above equation with $ 2Z_{hh} $ yields
\begin{equation*}
	\dfrac{d}{dt} \norm{Z_{hh}}{\Lnorm{2}}^2 \lesssim \norm{\nablah \hat v}{\Hnorm{2}} \norm{\nablah Z}{\Hnorm{1}}^2 + \norm{\hat Z}{\Hnorm{2}}\norm{\nablah \hat v}{\Hnorm{2}} \norm{\nablah Z}{\Hnorm{1}}.
\end{equation*}
To close the above estimate, we observe that similar arguments also hold for $ Z_h $. That is, we have
\begin{equation*}
	\dfrac{d}{dt} \norm{\nablah Z}{\Hnorm{1}}^2 \lesssim \norm{\nablah \hat v}{\Hnorm{2}} \norm{\nablah Z}{\Hnorm{1}}^2 + \norm{\hat Z}{\Hnorm{2}}\norm{\nablah \hat v}{\Hnorm{2}} \norm{\nablah Z}{\Hnorm{1}}.
\end{equation*}
Therefore, G\"onwall's inequality implies that
\begin{equation}\label{linear:002}
	\begin{aligned}
		& \sup_{0\leq t\leq T} \norm{\nablah Z}{\Hnorm{1}}^2 \leq e^{C \int_0^T (1+ \norm{\hat Z}{\Hnorm{2}}) \norm{\nablah \hat v}{\Hnorm{2}}\,dt }\bigl( \norm{\nablah Z_{0,\eta}}{\Hnorm{1}}^2 \\
		& ~~~~ + C \int_0^T (1+ \norm{\hat Z}{\Hnorm{2}}) \norm{\nablah \hat v}{\Hnorm{2}}\,dt\bigr)\leq 2 \mathcal E_{0,\eta},
	\end{aligned}
\end{equation}
for $ T $ small enough. Furthermore, from \subeqref{eq:linearized-FBCPE}{1}, similar arguments as those established in \eqref{BE-004} yield that
\begin{equation}\label{linear:003}
	\begin{aligned}
	& \norm{Z_t}{\Lnorm{2}}\lesssim \norm{z^{\alpha/2}\hat v}{\hHnorm{1}} \norm{\nablah Z}{\Hnorm{1}} + \norm{z^{\alpha/2} \nablah \hat v}{\Lnorm{2}}\\
	& ~~~~ \lesssim \norm{(z^{\alpha}+\eta)^{1/2}\hat v}{\hHnorm{1}} \norm{\nablah Z}{\Hnorm{1}} + \norm{(z^{\alpha}+\eta)^{1/2} \nablah \hat v}{\Lnorm{2}}\\
	& ~~~~ \lesssim \mathcal E_{0,\eta} + \mathcal E_{0,\eta}^{1/2}. 
	\end{aligned}
\end{equation}

\subsubsection*{Estimates for $ v $}

{\par\noindent\bf The $ L^2 $ estimate for $ v $ ~~ }
We start with the $ L^2 $ estimate for $ v $. After taking inner product of $ 2 v $ with \subeqref{eq:linearized-FBCPE}{2}, we have
\begin{equation}\label{ln-est:901}
\begin{aligned}
	& \dfrac{d}{dt} \norm{(z^\alpha\hat Z^\alpha+\eta)^{1/2} v}{\Lnorm{2}}^2 + 2\mu \norm{\nablah v}{\Lnorm{2}}^2 + 2(\mu+\lambda) \norm{\dvh v}{\Lnorm{2}}^2 \\
	& ~~~~ + 2 \mu \norm{\dz v}{\Lnorm{2}}^2 = \int \alpha z^\alpha \hat Z^{\alpha-1} \hat Z_t \abs{v}{2} \idx - \int 2 z^\alpha\hat Z^\alpha \hat v \cdot \nablah \hat v \cdot v \\
	& ~~~~ + 2 g z^\alpha \hat Z^\alpha \nablah \hat Z \cdot v \idx - 2 \int z^\alpha \hat Z^\alpha \hat W  \dz \hat v \cdot v \idx + 2 \int \bigl( \mu \nablah \log \hat Z \cdot \nablah \hat v\\
	& ~~~~ + (\mu+\lambda) \dvh \hat v \nablah \log \hat Z\bigr) \cdot v\idx \lesssim \bigl( 1 + \norm{\hat Z_t}{\Lnorm{2}} + \norm{\hat v}{\Hnorm{1}}^2 \bigr) \norm{v}{\Hnorm{1}}^2 \\
	& ~~~~ + \norm{\hat v}{\Hnorm{1}}^2 + \norm{\nablah \hat Z}{\Hnorm{1}}^2 - 2 \int z^\alpha \hat Z^\alpha \hat W \dz \hat v \cdot v\idx.
\end{aligned}
\end{equation}
Meanwhile, after plugging \eqref{eq:linearized-vertical-velocity} into the following term and applying H\"older's, Minkowski's, and the Sobolev embedding inequalities, one has that
\begin{equation}\label{ln-est:902}
\begin{aligned}
	& - 2 \int z^\alpha \hat Z^\alpha \hat W \dz \hat v \cdot v\idx \lesssim \int_0^1 \biggl\lbrack z^{\alpha/2} \hnorm{\dz \hat v}{\Lnorm{2}} \hnorm{v}{\Lnorm{4}} \bigl( \hnorm{\overline{z^\alpha \hat v}}{\Lnorm{8}}\hnorm{\nablah \hat Z}{\Lnorm{8}} + \hnorm{\overline{z^\alpha \nablah \hat v}}{\Lnorm{4}} \\
	& ~~~~ + (\int_0^z \xi^{\alpha} \hnorm{\hat v}{\Lnorm{8}}^2 \,d\xi)^{1/2} \hnorm{\nablah \hat Z}{\Lnorm{8}} + (\int_0^z \xi^\alpha\hnorm{\nablah \hat v}{\Lnorm{4}}^2\,d\xi)^{1/2} \bigr) \biggr\rbrack \,dz \\
	& ~~~~ \lesssim  \bigl( \norm{\nablah \hat Z}{\Hnorm{1}} + 1 \bigr)\norm{\dz \hat v}{\Lnorm{2}} \norm{v}{\Hnorm{1}} \norm{z^{\alpha/2}\hat v}{\hHnorm{2}}.
	\end{aligned}
\end{equation}
	Therefore, after substituting \eqref{ln-est:902} into \eqref{ln-est:901}, one integrates the resulting equation of \eqref{ln-est:901} with respect to the temporal variable to obtain:
	\begin{equation}\label{est-linearized-eq-001}
		\begin{aligned}
			&  \sup_{0\leq t\leq T} \norm{(z^\alpha\hat Z^\alpha+\eta)^{1/2} v}{\Lnorm{2}}^2 + \int_0^T \biggl\lbrack 2\mu \norm{\nablah v}{\Lnorm{2}}^2 + 2(\mu+\lambda) \norm{\dvh v}{\Lnorm{2}}^2 \\
	& ~~~~ + 2 \mu \norm{\dz v}{\Lnorm{2}}^2 \biggr\rbrack \,dt \leq \norm{(z^\alpha Z_{0,\eta}^\alpha+\eta)^{1/2} v_{0,\eta}}{\Lnorm{2}}^2 \\
	& ~~~~ + \int_0^T \mathcal H(\mathcal{\hat E}_{\eta}) \bigl( \norm{v}{\Hnorm{1}}^2 + 1 \bigr)\,dt + C T.
		\end{aligned}
	\end{equation}
	On the other hand, taking the $ L^2 $-inner product of \subeqref{eq:linearized-FBCPE}{2} with $ 2 v_t $ yields,
	\begin{equation}\label{ln-est:903}
	\begin{aligned}
		& \dfrac{d}{dt} \bigl( \mu \norm{ \nablah v}{\Lnorm{2}}^2 +(\mu+\lambda)\norm{\dvh v}{\Lnorm{2}}^2 + \lambda \norm{\dz v}{\Lnorm{2}}^2 \bigr) \\
		& ~~~~ + 2 \norm{(z^\alpha \hat Z^\alpha +\eta)^{1/2} v_t}{\Lnorm{2}}^2 = - 2\int z^\alpha \hat Z^\alpha \hat v \cdot \nablah \hat v \cdot v_t + g z^\alpha \hat Z^\alpha \nablah \hat Z \cdot v_t \idx \\
		& ~~~~ - 2\int z^\alpha \hat Z^\alpha \hat W \dz \hat v \cdot v_t \idx + 2 \int \bigl( \mu \nablah\log \hat Z \cdot \nablah \hat v \\
		& ~~~~ + (\mu+\lambda)\dvh \hat v \nablah \log \hat Z \bigr) \cdot v_t \idx  
		\lesssim \norm{\hat v}{\Hnorm{1}} \norm{\nablah \hat v}{\Hnorm{1}} \norm{z^{\alpha/2} v_t}{\Lnorm{2}}\\
		& ~~~~ + \norm{\nablah \hat Z}{\Lnorm{2}}  \norm{z^{\alpha/2}v_t}{\Lnorm{2}} + \norm{\nablah \hat Z}{\Hnorm{1}} \norm{\hat v}{\Hnorm{1}} \bigl( \norm{z^{\alpha/2} v_t}{\Lnorm{2}} + \norm{\nabla v_t}{\Lnorm{2}} \bigr) \\
		& - 2 \int z^\alpha \hat Z^\alpha \hat W \dz \hat v \cdot v_t \idx.
	\end{aligned}
\end{equation}
	Similarly as before, applying H\"older's and Minkowski's inequalities implies that
	\begin{equation}\label{ln-est:904}
	\begin{aligned}
		& - 2 \int z^\alpha \hat Z^\alpha \hat W \dz \hat v \cdot v_t \idx \lesssim \bigl( \norm{\nablah \hat Z}{\Hnorm{1}} + 1\bigr) \norm{z^{\alpha/2} v_t}{\Lnorm{2}} \norm{z^{\alpha/2}\hat v}{\hHnorm{2}} \\
		& ~~~~ \times \norm{\dz \hat v}{\hHnorm{1}}.
	\end{aligned}
\end{equation}
	Then substituting \eqref{ln-est:904} into \eqref{ln-est:903} and integrating the resulting inequality yields that
	\begin{equation}\label{est-linearized-eq-002}
		\begin{aligned}
			& \sup_{0\leq t\leq T} \bigl\lbrace \mu \norm{\nablah v(t)}{\Lnorm{2}}^2 + (\mu+\lambda) \norm{\dvh v(t)}{\Lnorm{2}}^2 + \mu\norm{\dz v(t)}{\Lnorm{2}}^2 \bigr\rbrace \\
			& ~~~~ + 2 \int_0^T \norm{(z^\alpha\hat Z^\alpha + \eta )^{1/2}v_t}{\Lnorm{2}}^2\,dt \leq  \mu \norm{\nablah v_{0,\eta}}{\Lnorm{2}}^2 \\
			& ~~~~ + (\mu+\lambda) \norm{\dvh v_{0,\eta}}{\Lnorm{2}}^2 + \lambda\norm{\dz v_{0,\eta}}{\Lnorm{2}}^2 +  \omega \int_0^T\norm{\dz \hat v}{\hHnorm{1}}^2 \\
			& ~~~~ + \norm{\nablah \hat v}{\Hnorm{1}}^2 + \norm{\nabla v_t}{\Lnorm{2}}^2 \,dt  +\int_0^T \mathcal H( \mathcal{\hat E}_{\eta}, 1/\omega) \bigl(\norm{z^{\alpha/2}v_t}{\Lnorm{2}}^2 + 1\bigr) \,dt,
		\end{aligned}
	\end{equation}
	for any $ \omega \in (0,1) $. 
{\par\noindent\bf The temporal derivative of $ v $ ~~ }
After applying $ \dt $ to \subeqref{eq:linearized-FBCPE}{2}, we have
\begin{align*}
	& (z^\alpha {\hat Z}^\alpha + \eta) \dt v_t= \mu \deltah v_t + (\mu+\lambda) \nablah \dvh v_t + \mu \partial_{zz} v_t \\
	& ~~~~ - \alpha z^\alpha {\hat Z}^{\alpha-1} \hat Z_t \dt v - z^\alpha {\hat Z}^\alpha \bigl( \hat v \cdot \nablah \hat v_t + \hat v_t \cdot \nablah \hat v + g \nablah \hat Z_t \bigr)  \\
	& ~~~~ - z^\alpha {\hat Z}^\alpha \bigl( \hat W \dz \hat v_t + \hat W_t \dz \hat v \bigr) - \alpha z^\alpha {\hat Z}^{\alpha-1}\hat Z_t \bigl( \hat v \cdot \nablah \hat v + g \nablah \hat Z \bigr) \\
	& ~~~~ - \alpha z^\alpha {\hat Z}^{\alpha-1} \hat Z_t \hat W \dz \hat v + \mu \nablah (\log \hat Z)_t \cdot \nablah \hat v + (\mu+\lambda) \dvh \hat v \nablah (\log \hat Z)_t\\
	& ~~~~ + \mu \nablah \log \hat Z \cdot \nablah \hat v_t + (\mu+\lambda)\dvh \hat v_t \nablah \log \hat Z.
\end{align*}
Taking the $ L^2 $-inner product of the above equation with $ 2v_t $ and applying integration by parts in the resultant yield that
\begin{align*}
	& \dfrac{d}{dt} \norm{(z^\alpha {\hat Z}^\alpha + \eta)^{1/2} v_t}{\Lnorm{2}}^2 + 2\mu\norm{\nablah v_t}{\Lnorm{2}}^2 + 2(\mu+\lambda) \norm{\dvh v_t}{\Lnorm{2}}^2 \\
	& ~~~~ + 2\mu \norm{\dz v_t}{\Lnorm{2}}^2 = - \alpha \int z^\alpha {\hat Z}^{\alpha-1} \hat Z_t \abs{\dt v}{2} \idx - 2 \int \bigl( z^\alpha {\hat Z}^\alpha \hat v \cdot \nablah \hat v_t \cdot v_t \bigr) \idx\\
	& ~~~~ - 2 \int \bigl( z^\alpha {\hat Z}^\alpha \hat v_t \cdot \nablah \hat v \cdot v_t \bigr) \idx  +  2 g \int z^\alpha {\hat Z}^{\alpha-1} \hat Z_t ( \alpha \nablah \hat Z \cdot v_t + {\hat Z} \dvh v_t ) \idx\\
	& ~~~~ - 2 \int z^\alpha  \hat Z^\alpha \hat W \dz \hat v_t\cdot v_t \idx  - 2 \int z^\alpha {\hat Z}^\alpha \hat W_t \dz \hat v \cdot v_t \idx \\
	& ~~~~ - 2\alpha \int \bigl( z^\alpha \hat Z^{\alpha-1} \hat Z_t \hat W \dz \hat v \cdot v_t \bigr) \idx  - 2 \alpha \int \bigl\lbrack z^\alpha {\hat Z}^{\alpha-1} \hat Z_t (\hat v\cdot\nablah \hat v + g \nablah \hat Z) \cdot v_t \bigr\rbrack \idx\\
	& ~~~~ -2 \int \bigl\lbrack\mu (\log \hat Z)_t \deltah \hat v  +(\mu+\lambda) (\log\hat Z)_t \nablah \dvh \hat v \bigr\rbrack \cdot v_t \idx \\
	& ~~~~ - 2 \int \bigl\lbrack \mu (\log \hat Z)_t \nablah \hat v : \nablah v_t + (\mu+\lambda) (\log \hat Z)_t \dvh \hat v \dvh v_t \bigr\rbrack \idx \\
	& ~~~~ + 2 \int \bigl\lbrack \bigl( \mu \nablah \log \hat Z \cdot \nablah \hat v_t + (\mu+\lambda) \dvh \hat v_t \nablah \log \hat Z\bigr) \cdot v_t \bigr\rbrack \idx \\
	& ~~~~ =: \sum_{k=1}^{11} L_k.
\end{align*}
Similarly as before, the estimates for $ L_k$'s are as follows: for every $ \delta, \omega \in (0,1) $,
\begin{align*}
	& L_1 \lesssim \delta \norm{\nabla v_t}{\Lnorm{2}}^2 + (C_\delta \norm{\hat Z_t}{\Lnorm{2}}^4 + \delta) \norm{z^{\alpha/2} v_t}{\Lnorm{2}}^{2},\\
	& L_2 \lesssim \delta \norm{\nabla v_t}{\Lnorm{2}}^2 + \omega \norm{\nablah \hat v_t}{\Lnorm{2}}^2 + (C_{\delta,\omega} \norm{\hat v}{\Hnorm{1}}^4 + \delta) \norm{z^{\alpha/2}v_t}{\Lnorm{2}}^2, \\
	& L_3 \lesssim \delta \norm{\nabla v_t}{\Lnorm{2}}^2 + \omega \norm{\nablah \hat v_t}{\Lnorm{2}}^2 + (C_{\delta,\omega} \norm{\hat v}{\Hnorm{1}}^4 + \omega) \norm{z^{\alpha/2}\hat v_t}{\Lnorm{2}}^2\\
	& ~~~~ + \delta \norm{z^{\alpha/2}v_t}{\Lnorm{2}}^2,\\
	& L_4 \lesssim \delta \norm{\nabla v_t}{\Lnorm{2}}^2 + \delta \norm{z^{\alpha/2}v_t}{\Lnorm{2}}^2 + C_\delta \bigl(\norm{\hat Z}{\Hnorm{2}}^2 + 1 \bigr) \norm{\hat Z_t}{\Lnorm{2}}^2, \\
	& L_5 = -2 \int \bigl( \hat Z^{\alpha-1} \dz \hat v_t \cdot v_t \bigr) \bigl\lbrack z^{\alpha+1} \bigl( (\alpha+1) \overline{z^\alpha \hat v} \cdot \nablah \hat Z + \overline{z^\alpha \dvh \hat v} \hat Z \bigr)\\
	& ~~~~ - (\alpha+1) (\int_0^z \xi^\alpha \hat v \,d\xi) \cdot \nablah \hat Z - (\int_0^z \xi^\alpha \dvh \hat v\,d\xi) \hat Z \bigr\rbrack \idx\\
	& ~~~~\lesssim \delta \norm{\nabla v_t}{\Lnorm{2}}^2 +\bigl( \delta  + C_\delta \bigr) \norm{z^{\alpha/2}v_t}{\Lnorm{2}}^2  + \bigl( \norm{\nablah \hat Z}{\Hnorm{1}}^2 + 1 \bigr) \\
	& ~~~~ \times \norm{z^{\alpha/2} \hat v}{\hHnorm{2}}^2 \norm{\nabla\hat v_t}{\Lnorm{2}}^2, \\
	& L_6 = 2 (\alpha+1) \int \biggl\lbrack (\dz \hat v \cdot v_t)  \hat Z^\alpha (z^{\alpha+1} \overline{z^\alpha \dvh \hat v} - \int_0^z \xi^\alpha \dvh \hat v \, d\xi ) (\log\hat Z)_t \\
	& ~~~~ ~~~~ + \alpha (\dz \hat v \cdot v_t) \hat Z^{\alpha-1}( z^{\alpha+1} \overline{z^\alpha \hat v} - \int_0^z \xi^\alpha \hat v\,d\xi)\cdot \nablah \hat Z (\log \hat Z)_t \\
	& ~~~~ ~~~~ + \hat Z^\alpha (z^{\alpha+1} \overline{z^\alpha\hat v} - \int_0^z \xi^\alpha \hat v\,d\xi) \cdot \nablah (\dz \hat v \cdot v_t) ( \log \hat Z)_t \biggr\rbrack \idx \\
	& ~~~~ - 2 (\alpha+1) \int \biggl\lbrack (\hat Z^\alpha \dz \hat v \cdot v_t) ( z^{\alpha+1} \overline{z^\alpha\hat v_t} - \int_0^z \xi^\alpha\hat v_t \,d\xi)\cdot\nablah \log \hat Z \biggr\rbrack \idx \\
	& ~~~~ - 2 \int \biggl\lbrack (\hat Z^\alpha \dz \hat v \cdot v_t) ( z^{\alpha+1}\overline{z^\alpha \dvh \hat v_t} - \int_0^z \xi^\alpha \dvh \hat v_t \,d\xi) \biggr\rbrack \idx \\
	& ~~~~ \lesssim \delta \norm{\nabla v_t}{\Lnorm{2}}^2+ \delta \norm{z^{\alpha/2}v_t}{\Lnorm{2}}^2 + C_\delta \bigl(\norm{\hat Z}{\Hnorm{2}}^2 + 1\bigr) \norm{\hat Z_t}{\Lnorm{2}}^2\norm{z^{\alpha/2}\hat v}{\hHnorm{2}}^2 \\
	& ~~~~ \times \norm{\dz \hat v}{\hHnorm{2}}^2 + C_\delta \bigl( \norm{\nablah \hat Z}{\Hnorm{1}}^2 + 1 \bigr) \norm{z^{\alpha/2}\dz \hat v}{\hHnorm{1}}^2 \bigl( \norm{z^{\alpha/2}\hat v_t}{\Lnorm{2}}^2 \\
	& ~~~~ + \norm{\nabla \hat v_t}{\Lnorm{2}}^2 \bigr) ,\\
	& L_7 = - 2\alpha \int \biggl\lbrack \hat Z^{\alpha-2} \hat Z_t \dz \hat v \cdot v_t \bigl( z^{\alpha+1}\bigl( (\alpha+1) \overline{z^\alpha \hat v} \cdot \nablah \hat Z + \overline{z^\alpha \dvh \hat v} \hat Z\bigr) \\
	& ~~~~ - (\alpha+1) (\int_0^z \xi^\alpha \hat v \,d\xi) \cdot \nablah \hat Z - (\int_0^z \xi^\alpha \dvh \hat v \,d\xi) \hat Z\bigr) \biggr\rbrack \idx \\
	& ~~~~  \lesssim  \delta \norm{\nabla v_t}{\Lnorm{2}}^2+ \delta \norm{z^{\alpha/2}v_t}{\Lnorm{2}}^2 \\
	& ~~~~ + C_\delta \bigl(\norm{\hat Z}{\Hnorm{2}}^2 + 1\bigr)\norm{\hat Z_t}{\Lnorm{2}}^2 \norm{z^{\alpha/2}\hat v}{\hHnorm{2}}^2 \norm{\hat v_z}{\hHnorm{1}}^2 ,\\
	& L_8 \lesssim \delta \norm{\nabla v_t}{\Lnorm{2}}^2 + \delta \norm{z^{\alpha/2}v_t}{\Lnorm{2}}^2 + C_\delta \norm{\hat Z}{\Hnorm{2}}^2 \norm{\hat Z_t}{\Lnorm{2}}^2 \\
	& ~~~~ + C_\delta \norm{\hat Z_t}{\Lnorm{2}}^2 \norm{\hat v}{\Hnorm{1}}^2 \norm{\nablah \hat v}{\Hnorm{1}}^2, \\
	& L_9 + L_{10} \lesssim \delta \norm{\nabla v_t}{\Lnorm{2}}^2 + \delta\norm{z^{\alpha/2}v_t}{\Lnorm{2}}^2 + C_\delta \norm{\hat Z_t}{\Lnorm{2}}^2 \norm{\nablah \hat v}{\Hnorm{2}}^2, \\
	& L_{11} \lesssim \delta \norm{\nabla v_t}{\Lnorm{2}}^2 + \delta\norm{z^{\alpha/2}v_t}{\Lnorm{2}}^2 + C_\delta \norm{\nablah \hat Z}{\Hnorm{1}}^2 \norm{\nabla\hat v_t}{\Lnorm{2}}^2,
\end{align*}
where we have applied integration by parts, H\"older's, Minkowski's, the Sobolev embedding inequalities, and
the fact that from \eqref{eq:linearized-vertical-velocity}, one has
\begin{align}
	& z^\alpha \hat W_t = z^{\alpha+1} \bigl( (\alpha+1) \overline{z^\alpha \hat v} \cdot \nablah (\log \hat Z)_t + ( \alpha+1) \overline{z^\alpha \hat v_t} \cdot \nablah \log \hat Z \nonumber\\
	& ~~~~ + \overline{z^\alpha \dvh \hat v_t} \bigr) - (\alpha+1) (\int_0^z \xi^\alpha \hat v \,d\xi) \cdot \nablah (\log \hat Z)_t \nonumber \\
	& ~~~~ - (\alpha + 1) (\int_0^z \xi^\alpha \hat v_t \,d\xi) \cdot \nablah \log \hat Z - (\int_0^z \xi^\alpha \dvh \hat v_t\,d\xi). \nonumber
\end{align}
Then after summing up the above inequalities with small enough $ \delta > 0 $ and every $ \omega \in (0,1) $ and integrating the resultant in the temporal variable, it follows that,
\begin{equation}\label{est-linearized-eq-003}
\begin{aligned}
	& \sup_{0\leq t\leq T} \norm{(z^\alpha\hat Z^\alpha + \eta)^{1/2} v_t(t)}{\Lnorm{2}}^2 + \int_0^T \mu \norm{\nablah v_t}{\Lnorm{2}}^2 + \mu \norm{\dz v_t}{\Lnorm{2}}^2 \,dt \\
	& ~~~~ \leq \norm{(z^\alpha Z_0^\alpha + \eta)^{1/2} v_{1,\eta}}{\Lnorm{2}}^2
	 + \int_0^T \mathcal H(\mathcal{\hat E}_{\eta}, \omega) \bigl(\norm{\nabla \hat v_t}{\Lnorm{2}}^2 \\
	 & ~~~~ ~~~~ + \norm{\nablah \hat v}{\Hnorm{2}}^2 \bigr) \,dt 
	 +\int_0^T \biggl( \mathcal H( \mathcal{\hat E}_{\eta}, 1/\omega ,1) \norm{z^{\alpha/2}v_t}{\Lnorm{2}}^2 \\
	 & ~~~~ ~~~~ + \mathcal H(\mathcal{\hat E}_{\eta}, 1/\omega , \omega) \biggr) \,dt.
\end{aligned}
\end{equation}

{\par\noindent\bf The horizontal spatial derivatives of $ v $ ~~ }
After applying $ \partial_{hh} $ to \subeqref{eq:linearized-FBCPE}{2}, it holds that,
\begin{equation}\label{eq:dhh-momentumeq-linearized}
\begin{aligned}
	& (z^\alpha \hat Z^\alpha +\eta ) \dt v_{hh} - \mu \deltah v_{hh} - (\mu+\lambda) \nablah \dvh v_{hh} - \mu \partial_{zz}v_{hh} \\
	& ~~~~ = - 2 \alpha z^\alpha\hat Z^{\alpha-1} \hat Z_h \dt v_h - \alpha z^\alpha \hat Z^{\alpha-1} \hat Z_{hh} \dt v - \alpha(\alpha-1) z^\alpha \hat Z^{\alpha-2} (\hat Z_h)^2 \dt v\\
	& ~~~~ - \partial_{hh}\bigl( z^\alpha \hat Z^\alpha (\hat v\cdot \nablah \hat v + g\nablah \hat Z  )\bigr) - \partial_{hh} ( z^\alpha\hat Z^\alpha \hat W \dz \hat v)\\
	& ~~~~ + \partial_{hh}\bigl( \mu \nablah \log \hat Z \cdot \nablah \hat v + 
	( \mu+\lambda) \dvh\hat v \nablah \log \hat Z\bigr) .
\end{aligned}
\end{equation}
Then taking the $ L^2 $-inner product of the above equation with $ 2 v_{hh} $ yields that
\begin{align*}
	& \dfrac{d}{dt} \norm{(z^\alpha\hat Z^\alpha+\eta)^{1/2} v_{hh}}{\Lnorm{2}}^2 + 2\mu \norm{\nablah v_{hh}}{\Lnorm{2}}^2 + 2(\mu+\lambda) \norm{\dvh v_{hh}}{\Lnorm{2}}^2 \\
	& ~~~~ + 2\mu\norm{\dz v_{hh}}{\Lnorm{2}}^2 = \alpha\int z^\alpha \hat Z^{\alpha-1} \hat Z_t \abs{v_{hh}}{2}\idx \\
	& ~~~~ - 4 \alpha \int z^\alpha \hat Z^{\alpha-1} \hat Z_h \dt v_h \cdot v_{hh} \idx - 2 \alpha \int \biggl\lbrack z^\alpha ( \hat Z^{\alpha-1} \hat Z_{hh} \\
	& ~~~~ ~~~~ + ( \alpha-1) \hat Z^{\alpha-2}(\hat Z_h)^2 ) \dt v \cdot v_{hh} \biggr\rbrack \idx + 2\int \biggl\lbrack  \partial_h \bigl( z^\alpha \hat Z^\alpha (\hat v\cdot \nablah \hat v \\
	& ~~~~ ~~~~ + g\nablah \hat Z)\bigr) \cdot v_{hhh} \biggr\rbrack \idx + 2\int \biggl\lbrack \partial_h(z^\alpha \hat Z^\alpha \hat W \dz \hat v) \cdot v_{hhh} \biggr\rbrack \idx \\
	& ~~~~ - 2 \int \biggl\lbrack  \partial_h\bigl( \mu \nablah \log\hat Z \cdot \nablah \hat v + (\mu + \lambda) \dvh \hat v\nablah \log \hat Z \bigr) \cdot v_{hhh} \biggr\rbrack \idx  \\
	& ~~~~ =: \sum_{k=12}^{17} L_k.
\end{align*}
Then applying H\"older's, Minkowski's, and the Sobolev embedding inequalities yields that, for every $ \delta,\omega\in (0,1) $, 
\begin{align*}
	& L_{12} \lesssim \delta \norm{\nabla v_{hh}}{\Lnorm{2}}^2 + \bigl(C_\delta \norm{\nablah \hat Z}{\Hnorm{1}}^4 + \delta\bigr) \norm{z^{\alpha/2} v_{hh}}{\Lnorm{2}}^2, \\
	& L_{13} \lesssim \delta \norm{\nabla v_{hh}}{\Lnorm{2}}^2 + \omega \norm{\nabla v_t}{\Lnorm{2}}^2 + \bigl( C_{\delta,\omega}\norm{\nablah \hat Z}{\Hnorm{1}}^4 + \delta \bigr) \norm{z^{\alpha/2}v_{hh}}{\Lnorm{2}}^2, \\
	& L_{14} \lesssim \delta \norm{\nabla v_{hh}}{\Lnorm{2}}^2 + \omega \norm{\nabla v_t}{\Lnorm{2}}^2 + \bigl( C_{\delta,\omega} (\norm{\nablah \hat Z}{\Hnorm{1}}^8 + \norm{\nablah \hat Z}{\Hnorm{1}}^4) + \omega \bigr) \\
	& ~~~~ \times \norm{z^{\alpha/2}v_t}{\Lnorm{2}}^2 + \delta \norm{z^{\alpha/2}v_{hh}}{\Lnorm{2}}^2, \\
	& L_{15} \lesssim \delta \norm{\nabla v_{hh}}{\Lnorm{2}}^2 + C_\delta \bigl( \norm{\nablah \hat Z}{\Hnorm{1}}^2 + 1 \bigr) \bigl( \norm{\nablah \hat Z}{\Hnorm{1}}^2 + \norm{\hat v}{\Hnorm{1}}^2 \norm{\nablah \hat v}{\Hnorm{2}}^2 \bigr), \\
	& L_{16} = \int \biggl\lbrack \hat Z^\alpha \partial_{hz} \hat v \cdot v_{hhh} \bigl( z^{\alpha+1} ((\alpha+1)\overline{z^\alpha\hat v} \cdot \nablah \log \hat Z + \overline{z^\alpha \dvh \hat v} ) \\
	& ~~~~ - (\alpha+1) (\int_0^z \xi^\alpha\hat v\,d\xi) \cdot \nablah \log \hat Z - \int_0^z \xi^\alpha \dvh \hat v\,d\xi\bigr) \biggr\rbrack \idx \\
	& ~~~~ + \int \biggl\lbrack \dz \hat v \cdot v_{hhh} \bigl( z^{\alpha+1} (\frac{{\alpha+1}}{\alpha }\overline{z^\alpha\hat v} \cdot \nablah (\hat Z^{\alpha})_h + \overline{z^\alpha \dvh \hat v}(\hat Z^{\alpha})_h)\\
	& ~~~~ - \frac{\alpha+1}{\alpha} (\int_0^z \xi^\alpha\hat v \,d\xi) \cdot \nablah (\hat Z^{\alpha})_h - \int_0^z \xi^\alpha\dvh \hat v \,d\xi ( \hat Z^\alpha )_h \bigr) \biggr\rbrack \idx \\
	& ~~~~ + \int \biggl\lbrack \dz \hat v \cdot v_{hhh} \bigl( z^{\alpha+1} (\frac{{\alpha+1}}{\alpha }\overline{z^\alpha\hat v_h} \cdot \nablah \hat Z^{\alpha} + \overline{z^\alpha \dvh \hat v_h}\hat Z^{\alpha})\\
	& ~~~~ - \frac{\alpha+1}{\alpha} (\int_0^z \xi^\alpha\hat v_h \,d\xi) \cdot \nablah \hat Z^{\alpha} - \int_0^z \xi^\alpha\dvh \hat v_h \,d\xi \hat Z^\alpha  \bigr) \biggr\rbrack \idx \\
	& ~~~~ \lesssim \delta \norm{v_{hhh}}{\Lnorm{2}}^2 + C_\delta \bigl(\norm{\nablah \hat Z}{\Hnorm{1}}^4 + 1\bigr) \norm{z^{\alpha/2}\hat v}{\hHnorm{2}}^2\norm{\nablah \hat v}{\Hnorm{2}}^2, \\
	& L_{17} \lesssim \delta \norm{v_{hhh}}{\Lnorm{2}}^2 + C_\delta \bigl( \norm{\nablah \hat Z}{\Hnorm{1}}^4 + \norm{\nablah \hat Z}{\Hnorm{1}}^2\bigr) \norm{\nablah \hat v}{\Hnorm{2}}^2.
\end{align*}
Therefore, integrating the sum of the above inequalities and the equation with $ \delta > 0 $ small enough with respect to the temporal variable yields that
\begin{equation}\label{est-linearized-eq-004}
	\begin{aligned}
		& \sup_{0\leq t\leq T} \norm{(z^\alpha\hat Z^\alpha+\eta)^{1/2} v_{hh}(t)}{\Lnorm{2}}^2 + \int_0^T \biggl( \mu \norm{\nablah v_{hh}}{\Lnorm{2}}^2 + \mu\norm{\dz v_{hh}}{\Lnorm{2}}^2\biggr) \,dt\\
		& ~~~~ \leq \norm{(z^\alpha Z_{0,\eta}^\alpha+\eta)^{1/2} v_{0,\eta,hh}}{\Lnorm{2}}^2 + \omega \int_0^T \norm{\nabla v_t}{\Lnorm{2}}^2 \,dt  + \int_0^T \biggl( \mathcal H(\mathcal{\hat E}_\eta) \\
		& ~~~~ \times \norm{\nablah \hat v}{\Hnorm{2}}^2 \biggr) \,dt 
		 + \int_0^T \biggl( \mathcal H(\mathcal{\hat E}_\eta, \omega, 1/\omega ,1) 
		\bigl(  \norm{z^{\alpha/2}v}{\hHnorm{2}}^2 \\
		& ~~~~ + \norm{z^{\alpha/2}v_t}{\Lnorm{2}}^2 
		 + 1 \bigr) \biggr) \,dt .
	\end{aligned}
\end{equation}
Similar arguments yield that
\begin{equation}\label{est-linearized-eq-005}
	\begin{aligned}
		& \sup_{0\leq t\leq T} \norm{(z^\alpha\hat Z^\alpha+\eta)^{1/2} v_{h}(t)}{\Lnorm{2}}^2 + \int_0^T \biggl( \mu \norm{\nablah v_{h}}{\Lnorm{2}}^2 + \mu\norm{\dz v_{h}}{\Lnorm{2}}^2 \biggr) \,dt\\
		& ~~~~ \leq \norm{(z^\alpha Z_{0,\eta}^\alpha+\eta)^{1/2} v_{0,\eta,h}}{\Lnorm{2}}^2 + \omega \int_0^T \norm{\nabla v_t}{\Lnorm{2}}^2 \,dt  + \int_0^T \biggl(  \mathcal H(\mathcal{\hat E}_\eta) \\
		& ~~~~ \times \norm{\nablah \hat v}{\Hnorm{2}}^2 \biggr) \,dt 
		 + \int_0^T \biggl(  \mathcal H(\mathcal{\hat E}_\eta, \omega, 1/\omega ,1) 
		\bigl(  \norm{z^{\alpha/2}v}{\hHnorm{2}}^2 \\
		& ~~~~ + \norm{z^{\alpha/2}v_t}{\Lnorm{2}}^2 
		 + 1 \bigr) \biggr) \,dt .
	\end{aligned}
\end{equation}
Therefore, from \eqref{est-linearized-eq-001}, \eqref{est-linearized-eq-002}, \eqref{est-linearized-eq-003}, \eqref{est-linearized-eq-004}, and \eqref{est-linearized-eq-005}, we have that
\begin{equation}\label{est-linearized-eq-total-v-1}
	\begin{aligned}
		&\sup_{0\leq t\leq T} \bigl\lbrace \norm{(z^\alpha+\eta)^{1/2} v(t)}{\hHnorm{2}}^2 + \mu \norm{\nabla v(t)}{\Lnorm{2}}^2 + \norm{(z^\alpha +\eta)^{1/2} v_t(t)}{\Lnorm{2}}^2  \bigr\rbrace  \\
		& ~~~~ + \int_0^T \biggl( \norm{(z^{\alpha} + \eta)^{1/2} v_t}{\Lnorm{2}}^2 + \mu \norm{\nablah v}{\hHnorm{2}}^2 + \mu \norm{\dz v}{\hHnorm{2}}^2 \\
		& ~~~~ + \mu \norm{\nablah v_t}{\Lnorm{2}}^2 + \mu \norm{\dz v_t}{\Lnorm{2}}^2 \biggr)  \,dt
		 \leq C \bigl( \norm{(z^\alpha+\eta)^{1/2} v_{0,\eta}}{\hHnorm{2}}^2 \\
		 & ~~~~ +  \norm{v_{0,\eta}}{\Hnorm{1}}^2 + \norm{(z^\alpha+\eta)^{1/2} v_{1,\eta}}{\Lnorm{2}}^2  \bigr) + \omega \int_0^T \norm{\nabla v_t}{\Lnorm{2}}^2 \,dt \\
		& ~~~~ + \int_0^T \biggl( \mathcal H(\mathcal{\hat E}_\eta ,\omega) \bigl( \norm{\nabla \hat v_t}{\Lnorm{2}}^2 + \norm{\nabla \hat v}{\Hnorm{1}}^2 + \norm{\nablah \hat v}{\Hnorm{2}}^2 \bigr) \biggr) \,dt \\
		& ~~~~ +  \int_0^T \biggl( \mathcal H(\mathcal{\hat E}_\eta,\omega, 1/\omega, 1) \bigl( \norm{v}{\Hnorm{1}}^2 + \norm{z^{\alpha/2}v_t}{\Lnorm{2}}^2 + \norm{z^{\alpha/2} v}{\hHnorm{2}}^2 + 1\bigr) \biggr) \,dt \\
		& ~~~~ + CT.
	\end{aligned}
\end{equation}

{\par\noindent \bf The vertical spatial derivative of $ v $ ~~ }
Taking the $ L^2 $-inner product of \eqref{eq:dhh-momentumeq-linearized} with $ 2 \partial_{zz} v $, after applying integration by parts in the resultant, yields that
\begin{align*}
	& \dfrac{d}{dt} \norm{(z^\alpha\hat Z^\alpha+\eta)^{1/2} v_{hz}}{\Lnorm{2}}^2 + 2 \mu \norm{\nablah v_{hz}}{\Lnorm{2}}^2 + 2(\mu+\lambda) \norm{\dvh v_{hz}}{\Lnorm{2}}^2 \\
	& ~~~~ + 2 \mu \norm{\dz v_{hz}}{\Lnorm{2}}^2 = \alpha \int z^\alpha\hat Z^{\alpha-1} \hat Z_t \abs{v_{hz}}{2} \idx \\
	& ~~~~ - 2 \alpha \int (z^\alpha \hat Z^{\alpha-1} \hat Z_h \dt v_z \cdot v_{hz}) \idx
	- 2 \alpha\int (z^{\alpha-1} \hat Z^\alpha \dt v_h \cdot v_{hz}) \idx \\\
	& ~~~~ - 2 \alpha^2 \int (z^{\alpha-1} \hat Z^{\alpha-1} \hat Z_h \dt v \cdot v_{hz}) \idx  + 2 \int (\partial_{h} (z^\alpha\hat Z^\alpha \hat W \dz \hat v) \cdot v_{hzz}) \idx \\
	& ~~~~ + 2 \int \bigl\lbrack \partial_{h} \bigl( z^\alpha \hat Z^\alpha (\hat v \cdot \nablah \hat v + g \nablah \hat Z) \bigr) \cdot v_{hzz} \bigr\rbrack \idx \\
	& ~~~~ - 2 \int \bigl\lbrack\partial_h \bigl( \mu \nablah \log \hat Z \cdot \nablah \hat v + (\mu+\lambda) \dvh \hat v\nablah \log \hat Z \bigr) \cdot v_{hzz}\bigr\rbrack \idx \\
	& ~~~~ =: \sum_{k=18}^{24} L_k.
\end{align*}
We list the estimates for the right-hand side in the following: for every $ \delta \in (0,1) $, 
\begin{align*}
	& L_{18} \lesssim \delta \norm{\nabla v_{hz}}{\Lnorm{2}}^2 + \bigl( \delta + C_\delta \norm{\hat Z_t}{\Lnorm{2}}^4 \bigr)\norm{z^{\alpha/2} v_{hz}}{\Lnorm{2}}^2, \\
	& L_{19} \lesssim \delta \norm{\nabla v_{hz}}{\Lnorm{2}}^2 + \delta \norm{z^{\alpha/2} v_{hz}}{\Lnorm{2}}^2 + C_\delta \norm{\nablah \hat Z}{\Hnorm{1}}^2 \norm{\nabla \dt v}{\Lnorm{2}}^2, \\
	& L_{20} \lesssim \delta \norm{\nabla v_{hz}}{\Lnorm{2}}^2 + C_\delta \norm{\nabla v_t}{\Lnorm{2}}^2, \\
	& L_{21} \lesssim \delta \norm{\nabla v_{hz}}{\Lnorm{2}}^2 + C_\delta \norm{\nablah \hat Z}{\Hnorm{1}}^2 \bigl( \norm{\nabla \dt v}{\Lnorm{2}}^2 + \norm{z^{\alpha/2} v_t}{\Lnorm{2}}^2 \bigr), \\
	& L_{22} = 2 \alpha \int (z^\alpha \hat Z^{\alpha-1} \hat Z_h \hat W \dz \hat v \cdot v_{hzz}) \idx + 2 \int (z^\alpha \hat Z^\alpha \hat W_h \dz \hat v \cdot v_{hzz}) \idx \\
	& ~~~~ + 2 \int (z^\alpha \hat Z^\alpha \hat W \dz \hat v_h \cdot v_{hzz})\idx \\
	& ~~~~ = 2\alpha \int \biggl( \hat Z^{\alpha-1} \dz \hat v \cdot v_{hzz} \hat Z_h \bigl\lbrack z^{\alpha+1}((\alpha+1) \overline{z^\alpha \hat v} \cdot \nablah \log \hat Z + \overline{z^\alpha \dvh \hat v})\\
	& ~~~~ - (\alpha+1) \int_0^z \xi^\alpha \hat v \,d\xi \cdot \nablah \log \hat Z - \int_0^z \xi^\alpha \dvh \hat v\,d\xi  \bigr\rbrack \biggr) \idx \\
	& ~~~~ + 2 \int \biggl( \hat Z^{\alpha} \dz \hat v \cdot v_{hzz} \bigl\lbrack (\alpha+1) z^{\alpha+1} ( \overline{z^\alpha\hat v} \cdot \nablah ( \log \hat Z)_h + \overline{z^\alpha\hat v_h}\cdot \nablah \log \hat Z) \\
	& ~~~~ + z^{\alpha+1} \overline{z^\alpha\dvh \hat v_h} - (\alpha+1) (\int_0^z \xi^\alpha\hat v \,d\xi \cdot \nablah (\log \hat Z)_h \\
	& ~~~~ + (\int_0^z \xi^\alpha \hat v_h \,d\xi) \cdot \nablah \log \hat Z)     - \int_0^z \xi^\alpha \dvh \hat v_h \,d\xi   \bigr\rbrack \biggr) \idx \\
	& ~~~~ + 2 \int \biggl( \hat Z^\alpha \dz \hat v_h \cdot v_{hzz} \bigl\lbrack z^{\alpha+1} ((\alpha+1) \overline{z^\alpha\hat v} \cdot \nablah \log\hat Z + \overline{z^\alpha\dvh \hat v}) \\
	& ~~~~ - (\alpha+1) \int_0^z \xi^\alpha \hat v \,d\xi \cdot \nablah \log \hat Z - \int_0^z \xi^\alpha\dvh \hat v \,d\xi  \bigr\rbrack \biggr) \idx \\
	& ~~~~ \lesssim \delta \norm{v_{hzz}}{\Lnorm{2}}^2 + C_\delta \bigl( \norm{\nablah \hat Z}{\Hnorm{1}}^4 + 1\bigr) \norm{z^{\alpha/2} \hat v}{\hHnorm{2}}^2 \norm{\dz \hat v}{\hHnorm{2}}^2  ,\\
	& L_{23} \lesssim \delta \norm{v_{hzz}}{\Lnorm{2}}^2 + C_\delta \bigl( \norm{\nablah \hat Z}{\Hnorm{1}}^2 + 1 \bigr) \bigl( \norm{\nablah \hat Z}{\Hnorm{1}}^2 + \norm{\hat v}{\Hnorm{1}}^2 \norm{\nablah \hat v}{\Hnorm{2}}^2 \bigr),\\
	& L_{24} \lesssim \delta \norm{v_{hzz}}{\Lnorm{2}}^2 + C_\delta \bigl( \norm{\nablah \hat Z}{\Hnorm{1}}^4 + \norm{\nablah \hat Z}{\Hnorm{1}}^2 \bigr) \norm{\nablah \hat v}{\Hnorm{2}}^2.
\end{align*}
Therefore, choosing $ \delta > 0 $ to be small enough and integrating the above equation with respect to the temporal variable yields that 
\begin{equation}\label{est-linearized-eq-006}
	\begin{aligned}
		& \sup_{0\leq t\leq T} \norm{(z^\alpha\hat Z^\alpha+\eta)^{1/2} v_{hz}(t)}{\Lnorm{2}}^2 + \int_0^T \biggl( \mu \norm{\nablah v_{hz}}{\Lnorm{2}}^2 + \mu \norm{\dz v_{hz}}{\Lnorm{2}}^2 \biggr) \,dt\\
		& ~~~~ \leq \norm{(z^\alpha Z_{0,\eta}^\alpha+\eta)^{1/2} v_{0,\eta,hz}}{\Lnorm{2}}^2 + C \int_0^T \norm{\nabla \dt v}{\Lnorm{2}}^2\,dt \\
		& ~~~~ + \int_0^T \mathcal H(\mathcal{\hat E}_\eta) \bigl( \norm{\nabla \dt v}{\Lnorm{2}}^2 + \norm{\nablah \hat v}{\Hnorm{2}}^2 \bigr)\,dt \\
		& ~~~~ + \int_0^T \mathcal H(\mathcal{\hat E}_\eta, 1)  \bigl(\norm{z^{\alpha/2} v_{z}}{\hHnorm{1}}^2 + \norm{z^{\alpha/2}v_t}{\Lnorm{2}}^2 + 1 \bigr)\,dt.
	\end{aligned}
\end{equation}
Similar arguments yield
\begin{equation}\label{est-linearized-eq-total-v-2}
	\begin{aligned}
		&\sup_{0\leq t\leq T} \norm{(z^\alpha+\eta)^{1/2} v_{z}(t)}{\hHnorm{1}}^2 + \int_0^T \biggl( \mu \norm{\nablah v_{z}}{\hHnorm{1}}^2 + \mu \norm{\dz v_{z}}{\hHnorm{1}}^2\,\biggr) dt\\
		& ~~~~ \leq \norm{(z^\alpha + \eta)^{1/2} v_{0,\eta,z}}{\hHnorm{1}}^2 + C \int_0^T \norm{\nabla \dt v}{\Lnorm{2}}^2\,dt \\
		& ~~~~ + \int_0^T \mathcal H(\mathcal{\hat E}_\eta) \bigl( \norm{\nabla \dt v}{\Lnorm{2}}^2 + \norm{\nablah \hat v}{\Hnorm{2}}^2 \bigr)\,dt \\
		& ~~~~ + \int_0^T \mathcal H(\mathcal{\hat E}_\eta, 1)  \bigl(\norm{z^{\alpha/2} v_{z}}{\hHnorm{1}}^2 + \norm{z^{\alpha/2}v_t}{\Lnorm{2}}^2 + 1 \bigr)\,dt.
	\end{aligned}
\end{equation}
\subsubsection*{The fixed point argument}
Since $ (\hat Z, \hat v) \in X_T $, \eqref{est-linearized-eq-total-v-1} and \eqref{est-linearized-eq-total-v-2} implies that for $ \omega $ small enough, there are constants $ 0 < C, C', C'', C''' < \infty $ that,
\begin{align*}
	& \sup_{0\leq t\leq T} \bigl\lbrace \norm{(z^\alpha+\eta)^{1/2} v(t)}{\hHnorm{2}}^2 + \mu \norm{\nabla v(t)}{\Lnorm{2}}^2 + \norm{(z^\alpha + \eta)^{1/2} v_t(t)}{\Lnorm{2}}^2 \\
	& ~~~~ + C' \norm{(z^\alpha+\eta)^{1/2} v_z(t)}{\hHnorm{1}}^2 \bigr\rbrace
	+  C'' \int_0^T \biggl( \norm{\nabla v}{\Hnorm{1}}^2 + \norm{\nablah v}{\Hnorm{2}}^2 \\
	& ~~~~ + \norm{\nabla v_t}{\Lnorm{2}}^2 + \norm{(z^\alpha+\eta)^{1/2}v_t}{\Lnorm{2}}^2 \biggr) \,dt \leq C''' \mathcal E_{0,\eta} + \mathcal H(\mathcal E_{0,\eta},\omega) \\
	& ~~~~ \times \bigl( C_0 \mathcal E_{0,\eta} + \int_0^T \norm{\nabla v_t}{\Lnorm{2}}^2 \,dt \bigr) + T \mathcal H(\mathcal E_{0,\eta},\omega,1/\omega,1) \sup_{0\leq t\leq T}\bigl\lbrace \norm{\nabla v}{\Lnorm{2}}^2 \\
	& ~~~~ + \norm{z^{\alpha/2} v}{\hHnorm{2}}^2 + \norm{z^{\alpha/2}v_t}{\Lnorm{2}}^2 + \norm{z^{\alpha/2} v_z}{\hHnorm{1}}^2 + 1 \bigr\rbrace  + CT.
\end{align*}
Then for $ \mathcal E_{0,\eta} $ small enough, taking $ \omega $ small and then $ T $ small implies that
\begin{align*}
	& \sup_{0\leq t\leq T} \bigl\lbrace \norm{(z^\alpha+\eta)^{1/2} v(t)}{\hHnorm{2}}^2 + \mu \norm{\nabla v(t)}{\Lnorm{2}}^2 + \norm{(z^\alpha + \eta)^{1/2} v_t(t)}{\Lnorm{2}}^2 \\
	& ~~~~ + C' \norm{(z^\alpha+\eta)^{1/2} v_z(t)}{\hHnorm{1}}^2 \bigr\rbrace
	+  C'' \int_0^T \biggl( \norm{\nabla v}{\Hnorm{1}}^2 + \norm{\nablah v}{\Hnorm{2}}^2 \\
	& ~~~~ + \norm{\nabla v_t}{\Lnorm{2}}^2 + \norm{(z^\alpha+\eta)^{1/2}v_t}{\Lnorm{2}}^2 \biggr) \,dt \leq 4C''' \mathcal E_{0,\eta}.
\end{align*}
Therefore, together with \eqref{linear:001}, \eqref{linear:002} and \eqref{linear:003}, this implies $ ( Z, v) \in X_T $ for small enough $ \mathcal E_{0,\eta}, T $ and large enough $ C_0 $. This verifies that $ \Phi: (\hat Z, \hat v) \mapsto (Z,v) $ is from $ X_T $ to itself. It is easy to verify that $ \Phi $ is continuous in the topology of $ L^2(0,T;L^2(\Omega)) \times L^2(0,T;L^2(\Omega)) $ and that $ X_T $ is convex and compact in the $ L^2(0,T;L^2(\Omega)) \times L^2(0,T;L^2(\Omega)) $ topology.  Therefore, the Tychonoff fixed point theorem implies that there is a fixed point of $ \Phi $ in $ X_T $ for small enough $ \mathcal E_{0,\eta}, T $ and large enough $ C_0 $. This is the local solution to the following equations,
\begin{equation*}
	\begin{cases}
		\dt Z + (\alpha+1) \overline{z^\alpha  v} \cdot \nablah Z + \overline{z^\alpha\dvh  v}  Z = 0, \\
		(z^\alpha { Z}^{\alpha} + \eta )\dt v + z^\alpha { Z}^\alpha \bigl(  v \cdot \nablah  v +  W \dz  v + g \nablah  Z \bigr) = \mu \deltah v \\
		~~~~ + (\mu+\lambda) \nablah \dvh v + \mu \partial_{zz} v + \mu \nablah \log Z \cdot \nablah v \\
		~~~~ + (\mu+\lambda) \dvh v \nablah \log Z,\\
		\dz Z =0,\\
		z^\alpha { Z}  W = z^{\alpha+1}\bigl( (\alpha+1) \overline{z^\alpha  v} \cdot \nablah  Z + \overline{z^\alpha \dvh  v}  Z\bigr)\\
		~~~~ - (\alpha+1) (\int_0^z \xi^\alpha  v \,d\xi) \cdot \nablah  Z - (\int_0^z \xi^\alpha \dvh  v \,d\xi)  Z.
	\end{cases}
\end{equation*}
Notice, the above estimates are independent of $ \eta $. Thus by taking $ \eta \rightarrow 0^+ $ in the above equations, we achieve the local solutions to \eqref{FB-CPE} for small initial energy $ \mathcal E_0 $ given in \eqref{STB-ttl-ini-energy}. In other words, we have shown the following proposition:
\begin{proposition}\label{prop:local-existence}
There is a constant $ \varepsilon_4 < \varepsilon_3 $, where $ \varepsilon_3 $ is given in Proposition \ref{prop:H^2-of-v}, such that if $ \mathcal E_0 < \varepsilon_4 $ small enough, that there are a $ T_{\varepsilon_4} \in (0,\infty) $ and a local strong solution $ (Z,v) \in L^\infty(0,T_{\varepsilon_4};H^2(\Omega)) \times L^\infty(0,T_{\varepsilon_4};H^2(\Omega)) $ to \eqref{FB-CPE}. Moreover, there is a constant $ 0 < C_2 < \infty $ such that
\begin{equation}
	\mathcal E(t) \leq C_2 \mathcal E_0,
\end{equation}
for $ t \in (0,T_{\varepsilon_4}) $. 
\end{proposition}

\subsection{Uniqueness and global stability theory} 
We summarize the result from Propositions \ref{prop:stability-theory}, \ref{prop:H^2-of-v}, \ref{prop:local-existence} in the following proposition. 
\begin{proposition}\label{prop:global-stability}
	Consider initial data $ (Z_0, v_0) \in H^2(\Omega) \times H^2(\Omega) $ satisfying $ \mathcal E_0 < \varepsilon' $ with some $ \varepsilon' \leq \min \lbrace \varepsilon_2/C_1, \varepsilon_2/C_2,\varepsilon_4/C_1, \varepsilon_4/C_2, \varepsilon_2 , \varepsilon_4 \rbrace $ where $ \varepsilon_2, C_1 $ are given in Proposition \ref{prop:stability-theory} and $ \varepsilon_4, C_2 $ is given in Proposition \ref{prop:local-existence}. There exists a unique global strong solution $ (Z,v) $  to \eqref{FB-CPE} satisfying estimates \eqref{apriori:total-energy} and \eqref{apriori:H^2-v}.
\end{proposition}
\begin{proof}
	The existence of the global strong solution follows by a continuity argument, using the fact that the relevant norms of the solutions remain bounded on any finite maximal interval of existence. What is left is to show the uniqueness. 
	Consider $(Z_1,v_1)$, $(Z_2,v_2)$ (in this proof, not to get confused with the initial data) being two strong solutions to \eqref{FB-CPE} with the same initial data. Then, for $ i = 1,2 $, the following equations hold,
	\begin{equation}
		\begin{cases}
			z^\alpha Z^\alpha_i (\dt v_i + v_i \cdot \nablah v_i + W_i \partial_z v_i + g \nablah Z_i) = \mu \deltah v_i 
		+ (\mu + \lambda) \nablah \dvh v_i \\
		~~~~ ~~~~ 
		+ \mu \partial_{zz} v_i + \mu \nablah \log Z_i \cdot \nablah v_i +(\mu+\lambda) \dvh v_i \nablah \log Z_i ,\\
		\dt Z_i + (\alpha + 1) \overline{z^\alpha v_i} \cdot \nablah Z_i + \overline{z^\alpha \dvh v_i} Z_i = 0,   \\
		z^\alpha W_i = z^{\alpha+1} \bigl( (\alpha+1) \overline{z^\alpha v_i} \cdot \nablah \log Z_i + \overline{z^\alpha \dvh v_i} \bigr)  \nonumber \\
		~~~~ ~~~~ - (\alpha+ 1) \int_0^z \xi^\alpha v_i\,d\xi \cdot \nablah \log Z_i - \int_0^z \xi^\alpha \dvh v_i \,d\xi,\\
		\partial_z Z_i = 0.
		\end{cases}
	\end{equation}
	Next, we denote the total energy functionals for these two solutions as $ \mathcal E_i(t) $ for $ i = 1,2 $, which are given as in \eqref{STB-total-energy},
	\begin{align*}
		& \mathcal E_i(t) := \norm{z^{\alpha/2} v_i}{\hHnorm{2}}^2 + \norm{z^{\alpha/2} v_{i,t}}{\Lnorm{2}}^2 + \norm{z^{\alpha/2}v_{i,z}}{\hHnorm{1}}^2 + \norm{v_i}{\Hnorm{1}}^2 \\
		& ~~~~ + \norm{Z_i - 1}{\Hnorm{2}}^2 + \norm{Z_{i,t}}{\Lnorm{2}}^2.
	\end{align*}
	In the meantime, we denote $ Z_{12} = Z_1 - Z_2, v_{12} = v_1-v_2 $. Thus $ (Z_{12},v_{12})|_{t=0} = 0 $. Then we have the following equations for $ (Z_{12},v_{12}) $,
	\begin{align}
			& \dt Z_{12} + (\alpha+1) \overline{z^\alpha v_1} \cdot \nablah Z_{12} + \overline{z^\alpha \dvh v_1} Z_{12} \nonumber \\
			& ~~~~ + (\alpha+1)\overline{z^\alpha v_{12}} \cdot \nablah Z_2 + \overline{z^\alpha \dvh v_{12}} Z_2, \label{uniq-001} \\
			& z^\alpha Z^\alpha_1 \dt v_{12} - \mu \deltah v_{12} - (\mu+\lambda) \nablah \dvh v_{12} - \mu \partial_{zz} v_{12} \nonumber \\
			& ~~~~   = - z^\alpha (Z^\alpha_1 - Z^\alpha_2) ( \dt v_2 + v_2 \cdot \nablah v_2 + W_2 \partial_z v_2 + g \nablah Z_2) \nonumber \\
			& ~~~~ - z^\alpha Z^\alpha_1 ( v_{12} \cdot\nablah v_1 + v_2 \cdot \nablah v_{12} + (W_{1}-W_2) \dz v_1 + W_{2} \dz v_{12} \nonumber \\
			& ~~~~ + g\nablah Z_{12} ) 
			+ \mu \nablah ( \log Z_1 - \log Z_2) \cdot \nablah v_1 + \mu \nablah \log Z_2 \cdot \nablah v_{12} \nonumber \\
			& ~~~~ + (\mu+\lambda) \dvh v_1 \nablah (\log Z_1 - \log Z_2) + (\mu+\lambda) \dvh v_{12} \nablah \log Z_2, \label{uniq-002} \\
			& z^\alpha (W_{1}-W_2) = z^{\alpha+1} \bigl( (\alpha+1) \overline{z^\alpha v_{1}} \cdot \nablah (\log Z_1 - \log Z_2) \nonumber \\
			& ~~~~ + (\alpha+1) \overline{z^\alpha v_{12}}\cdot \nablah \log Z_2 + \overline{z^\alpha\dvh v_{12}} \bigr) \nonumber \\
			& ~~~~ - (\alpha+1) \int_0^z \xi^\alpha v_1 \,d\xi \cdot \nablah (\log Z_1- \log Z_2) \nonumber \\
			& ~~~~ - (\alpha+1)\int_0^z \xi^\alpha v_{12}\,d\xi \cdot \nablah \log Z_2 - \int_0^z \xi^\alpha \dvh v_{12} \,d\xi. \nonumber
	\end{align}
	Now we deduce $ L^2 $ estimates for \eqref{uniq-001} and \eqref{uniq-002}. It follows that,
	\begin{align*}
		& \dfrac{d}{dt} \norm{Z_{12}}{\Lnorm{2}}^2 = (\alpha - 1) \inth \overline{z^\alpha \dvh v_1} \abs{Z_{12}}{2} \idxh - \inth 2(\alpha+1) \overline{z^\alpha v_{12}} \cdot\nablah Z_2 Z_{12} \\
		& ~~~~ + 2 \overline{z^\alpha \dvh v_{12}} Z_2 Z_{12} \idxh =: G_1 + G_2 ,  \\
		\shortintertext{and}
		& \dfrac{d}{dt} \norm{z^{\alpha/2} Z^{\alpha/2}_1 v_{12}}{\Lnorm{2}}^2 + 2\mu \norm{\nablah v_{12}}{\Lnorm{2}}^2 + 2(\mu+\lambda) \norm{\dvh v_{12}}{\Lnorm{2}}^2 \\
		& ~~~~ + 2\mu \norm{\dz v_{12}}{\Lnorm{2}}^2 = - \alpha \int z^\alpha Z_1^{\alpha-1} Z_{1,t} \abs{ v_{12}}{2} \idx - 2 \int z^\alpha (Z^\alpha_1 - Z^\alpha_2) ( \dt v_2 \\
		& ~~~~ + v_2 \cdot \nablah v_2 + g\nablah Z_2) \cdot v_{12} \idx - 2 \int z^\alpha ( Z_1^\alpha - Z_2^\alpha) W_2 \dz v_2 \cdot v_{12} \idx \\
		& ~~~~ - 2 \int z^\alpha Z_1^\alpha (v_{12} \cdot \nablah v_1 + v_2 \cdot \nablah v_{12} + g \nablah Z_{12}) \cdot v_{12} \idx \\
		& ~~~~ -2 \int z^\alpha Z_1^\alpha (W_1-W_2) \dz v_1 \cdot v_{12} \idx +  \int Z_1^\alpha \dz (z^\alpha W_2) \abs{v_{12}}{2} \idx \\
		& ~~~~ + 2\int \mu \nablah (\log Z_1 - \log Z_2) \cdot \nablah v_1 \cdot v_{12} \\
		& ~~~~ ~~~~ + (\mu+\lambda) \dvh v_1 \nablah (\log Z_1 - \log Z_2) \cdot v_{12} \idx \\
		& ~~~~ + 2 \int \mu \nablah \log Z_2 \cdot \nablah v_{12} \cdot  v_{12} + (\mu+\lambda) \dvh v_{12} \nablah \log Z_2 \cdot v_{12} \idx \\
		& ~~~~ =: \sum_{k=3}^{10} G_k.
	\end{align*}
	Notice that $ (Z_i, v_i) $ are the solutions that were established before, satisfying 
	\begin{align*}
 		& \sup_{0 \leq t \leq T}\mathcal E_i (t) +\int_0^T \norm{Z_i-1}{\Hnorm{2}}^2+ \norm{Z_{i,t}}{\Lnorm{2}}^2  + \norm{\nabla v_i}{\Hnorm{1}}^2 + \norm{\nablah v_i}{\Hnorm{2}}^2 \\
		& ~~~~ + \norm{z^{\alpha/2} v_{i,t}}{\Lnorm{2}}^2 + \norm{\nabla v_{i,t}}{\Lnorm{2}}^2 \,dt  < C \mathcal E_0,\\
		& \abs{Z_i-1}{} \leq 1/2.
	\end{align*}
 	Then we have that, for any $ \delta \in (0,1) $, 
	\begin{align*}
		& G_1 \lesssim \norm{\nablah v_{1}}{\Hnorm{2}} \norm{Z_{12}}{\Lnorm{2}}^2,\\
		& G_2 \lesssim \delta \norm{\nablah v_{12}}{\Lnorm{2}}^2 + C_\delta \bigl( \norm{\nablah Z_2}{\Hnorm{1}}^2 + 1 \bigr) \bigl( \norm{Z_{12}}{\Lnorm{2}}^2 + \norm{z^{\alpha/2} v_{12}}{\Lnorm{2}}^2 \bigr) , \\
		& G_3 \lesssim \delta \norm{\nabla v_{12}}{\Lnorm{2}} + \bigl( \delta + C_\delta \norm{Z_{1,t}}{\Lnorm{2}}^4 \bigr) \norm{z^{\alpha/2}v_{12}}{\Lnorm{2}}^2, \\
		& G_4 \lesssim \delta \norm{\nabla v_{12}}{\Lnorm{2}}^2 + \delta \norm{z^{\alpha/2} v_{12}}{\Lnorm{2}}^2 + C_\delta \bigl( \norm{z^{\alpha/2} \dt v_2}{\Lnorm{2}}^2 + \norm{\nablah Z_2}{\Hnorm{1}}^2\\
		& ~~~~ + \norm{\nabla \dt v_2}{\Lnorm{2}}^2 + \norm{v_2}{\Hnorm{1}}^2 \norm{\nablah v_{2}}{\Hnorm{1}}^2 \bigr) \norm{Z_{12}}{\Lnorm{2}}^2, \\
		& G_5 = -2 \int \biggl( (Z_1^\alpha - Z_2^\alpha) \dz v_2 \cdot v_{12} \bigl\lbrack z^{\alpha+1} ((\alpha+1) \overline{z^\alpha v_2} \cdot\nablah \log Z_2 \\
		& ~~~~ + \overline{z^\alpha \dvh v_2}) - (\alpha+1)\int_0^z \xi^\alpha v_2 \,d\xi \cdot \nablah \log Z_2 - \int_0^z \xi^\alpha \dvh v_2 \,d\xi  \bigr\rbrack \biggr) \idx \\
		& ~~~~ \lesssim \delta\norm{\nablah v_{12}}{\Lnorm{2}}^2 + \delta \norm{z^{\alpha/2}v_{12}}{\Lnorm{2}}^2 + C_\delta \bigl(\norm{\nablah Z_2}{\Hnorm{1}}^2 + 1\bigr) \norm{z^{\alpha/2} v_2}{\hHnorm{2}}^2 \\
		& ~~~~ \times \norm{\dz v_2}{\hHnorm{1}}^2 \norm{Z_{12}}{\Lnorm{2}}^2,\\
		& G_6 = -2 \int (z^\alpha Z_1^\alpha v_{12} \cdot \nablah v_1) \cdot v_{12} \idx + \int z^\alpha \dvh (Z_1^\alpha v_2) \abs{v_{12}}{2}\idx \\
		& ~~~~ + 2 g \int z^\alpha \bigl( Z_{12} \nablah Z_1^\alpha \cdot v_{12} + Z_1^\alpha Z_{12} \dvh v_{12} \bigr)\idx \lesssim \delta \norm{\nabla v_{12}}{\Lnorm{2}}^2 \\
		& ~~~~ + \delta \norm{z^{\alpha/2} v_{12}}{\Lnorm{2}}^2 + C_\delta \bigl( \norm{\nablah Z_1}{\Hnorm{1}}^2 \norm{v_2}{\Hnorm{1}}^2 + \norm{\nablah v_1}{\Hnorm{1}}^2 \\
		& ~~~~ + \norm{\nablah v_2}{\Hnorm{1}}^2 \bigr) \norm{z^{\alpha/2} v_{12}}{\Lnorm{2}}^2 + C_\delta \bigl( \norm{\nablah Z_1}{\Hnorm{1}}^2 + 1 \bigr) \norm{Z_{12}}{\Lnorm{2}}^2, \\
		& G_{7} = - 2 \int \biggl( Z_1^\alpha \dz v_1 \cdot v_{12} \bigl\lbrack z^{\alpha+1} ((\alpha+1) \overline{z^\alpha v_1} \cdot \nablah (\log Z_1 - \log Z_2) \\
		& ~~~~ + (\alpha+1) \overline{z^\alpha v_{12}} \cdot \nablah \log Z_2 + \overline{z^\alpha \dvh v_{12}})\\
		& ~~~~ - (\alpha+1 ) \int_0^z \xi^\alpha v_1 \,d\xi \cdot \nablah (\log Z_1 - \log Z_2)  \\
		& ~~~~ - (\alpha+1) \int_0^z \xi^\alpha v_{12} \,d\xi \cdot \nablah \log Z_2 - \int_0^z \xi^\alpha \dvh v_{12}\,d\xi   \bigr\rbrack \biggr) \idx \\
		& ~~~~ = - 2 \int \biggl( Z_1^\alpha \dz v_1 \cdot v_{12} \bigl\lbrack z^{\alpha+1}((\alpha+1)\overline{z^\alpha v_{12}} \cdot \nablah \log Z_2 + \overline{z^\alpha \dvh v_{12}}) \\
		& ~~~~ - (\alpha+1) \int_0^z \xi^\alpha v_{12} \,d\xi \cdot \nablah \log Z_2 - \int_0^z \xi^\alpha \dvh v_{12} \,d\xi \bigr\rbrack \biggr) \idx \\
		& ~~~~ + 2 (\alpha+1) \int \biggl( z^{\alpha+1}  (\log Z_1 - \log Z_2)   \dvh (Z_1^\alpha (\dz v_1 \cdot v_{12}) \overline{z^\alpha v_1}) \biggr) \idx \\
		& ~~~~ - 2 (\alpha + 1) \int \biggl(  (\log Z_1 - \log Z_2)   \dvh \lbrack Z_1^\alpha (\dz v_1 \cdot v_{12}) \int_0^z \xi^\alpha v_1\,d\xi \rbrack \biggr) \idx \\
		& ~~~~ \lesssim \delta \norm{\nabla v_{12}}{\Lnorm{2}}^2 + \delta \norm{z^{\alpha/2} v_{12}}{\Lnorm{2}}^2 + C_\delta \bigl( \norm{\nablah Z_2}{\Hnorm{1}}^2 + 1 \bigr) \norm{\nablah v_1}{\Hnorm{2}}^2 \\
		& ~~~~ \times \norm{z^{\alpha/2} v_{12}}{\Lnorm{2}}^2 + C_\delta \bigl( \norm{\nablah Z_1}{\Hnorm{1}}^2 + 1\bigr) \norm{z^{\alpha/2}v_1}{\hHnorm{2}}^2 \norm{\nablah v_1}{\Hnorm{2}}^2 \\
		& ~~~~ \times \norm{Z_{12}}{\Lnorm{2}}^2  ,\\
		& G_8 = \int \biggl( Z_1^\alpha \abs{v_{12}}{2} \bigl\lbrack (\alpha+1)z^{\alpha}((\alpha+1) \overline{z^\alpha v_2} \cdot \nablah \log Z_2 + \overline{z^\alpha \dvh v_2}) \\
		& ~~~~ -(\alpha+1) z^\alpha v_2 \cdot \nablah \log Z_2 - z^\alpha \dvh v_2  \bigr\rbrack \biggr) \idx \lesssim \delta \norm{\nabla v_{12}}{\Lnorm{2}}^2 \\
		& ~~~~ + \delta \norm{z^{\alpha/2}v_{12}}{\Lnorm{2}}^2 + C_\delta
		 \bigl( \norm{\nablah Z_2}{\Hnorm{1}}^2 \norm{v_2}{\Hnorm{1}}^2 + \norm{\nablah v_2}{\Hnorm{1}}^2\bigr)\\
		 & ~~~~ \times \norm{z^{\alpha/2} v_{12}}{\Lnorm{2}}^2, \\
		& G_9 = -2 \mu \int \biggl( (\log Z_1 -\log Z_2)  ( \deltah v_1 \cdot v_{12} + \nablah v_1 \otimes \nablah v_{12} )  \biggr) \idx \\
		& ~~~~ - 2(\mu+\lambda) \int \biggl( (\log Z_1 -\log Z_2)  (v_{12} \cdot \nablah \dvh v_1 \\
		& ~~~~ + \dvh v_1 \dvh v_{12} ) \biggr) \idx \lesssim \delta \norm{\nabla v_{12}}{\Lnorm{2}}^2 + \delta \norm{z^{\alpha/2}v_{12}}{\Lnorm{2}}^2 \\
		& ~~~~ + C_\delta \norm{\nablah v_1}{\Hnorm{2}}^2 \norm{Z_{12}}{\Lnorm{2}}^2. 
	\end{align*}
	Moreover, for $ \omega > 0 $, small enough,
	\begin{align*}
		& G_{10} \lesssim \norm{\nablah Z_2}{\Lnorm{6}}\norm{\nablah v_{12}}{\Lnorm{2}} \norm{v_{12}}{\Lnorm{3}} \lesssim \norm{\nablah Z_2}{\Hnorm{1}}\norm{\nablah v_{12}}{\Lnorm{2}} \norm{v_{12}}{\Lnorm{2}}^{1/2} \\
		& ~~~~ \times \bigl( \norm{z^{\alpha/2}v_{12}}{\Lnorm{2}}^{1/2} + \norm{\nabla v_{12}}{\Lnorm{2}}^{1/2} \bigr) \lesssim \norm{\nablah Z_2}{\Hnorm{1}}\norm{\nablah v_{12}}{\Lnorm{2}} \\
		& ~~~~ \times \bigl( ( \omega^{-1- \max\lbrace [\alpha/2] - 1, 0 \rbrace/2 } + \omega^{-\alpha/4} )  \norm{z^{\alpha/2}v_{12}}{\Lnorm{2}}^{1/2} + \omega^{1/2} \norm{\nabla v_{12}}{\Lnorm{2}}^{1/2} \bigr) \\
		& ~~~~ \times \bigl( \norm{z^{\alpha/2}v_{12}}{\Lnorm{2}}^{1/2} + \norm{\nabla v_{12}}{\Lnorm{2}}^{1/2} \bigr) \lesssim \delta \norm{\nabla v_{12}}{\Lnorm{2}}^{2} \\
		& ~~~~ + C_\delta \mathcal H(\norm{\nablah Z_2}{\Hnorm{1}}) \norm{z^{\alpha/2}v_{12}}{\Lnorm{2}}^2,
	\end{align*}
	where we have applied the following inequality with small enough $ \omega = \omega(\norm{\nablah Z_2}{\Hnorm{1}}^{-1}, \delta ) $ in $ G_{10} $,
	\begin{align*}
		& \norm{v_{12}}{\Lnorm{2}}^2 = \int_0^{\omega} \hnorm{v_{12}}{\Lnorm{2}}^2\,dz +  \int_{\omega}^1 \hnorm{v_{12}}{\Lnorm{2}}^2\,dz \lesssim \bigl( \omega^{-4- 2\max\lbrace [\alpha/2] - 1, 0 \rbrace } \\
		& ~~~~ + \omega^{-\alpha} \bigr)  \norm{z^{\alpha/2}v_{12}}{\Lnorm{2}}^2 + \omega^2 \norm{\nabla v_{12}}{\Lnorm{2}}^2, 
	\end{align*}
	in which we use Hardy's inequality as follows:
	\begin{align*}
		& \int_0^\omega \hnorm{v_{12}}{\Lnorm{2}}^2 \,dz \lesssim \omega^{-2} \int_0^\omega z^2 \hnorm{v_{12}}{\Lnorm{2}}^2 \,dz + \omega^2 \int_0^\omega  \hnorm{\dz v_{12}}{\Lnorm{2}}^{2} \,dz \\
		& ~~~~ \lesssim \omega^{-4} \int_0^\omega z^{4} \hnorm{v_{12}}{\Lnorm{2}}^2 \,dz + \omega^{2} \int_0^\omega \hnorm{\dz v_{12}}{\Lnorm{2}}^2 \,dz \\
		& \underbrace{ \lesssim \cdots \lesssim }_{\max\lbrace [\alpha/2] - 1, 0 \rbrace ~ \text{times}} \omega^{-4- 2\max\lbrace [\alpha/2] - 1, 0 \rbrace } \norm{z^{\alpha/2}v_{12}}{\Lnorm{2}}^2 + \omega^2 \norm{\nabla v_{12}}{\Lnorm{2}}^2.
	\end{align*}
	Therefore, summing up the above inequalities and equations with small enough $ \delta > 0 $ yields that for some constant $ C_{\mu,\lambda} > 0 $,
	\begin{align*}
		& \dfrac{d}{dt} \bigl\lbrace \norm{Z_{12}}{\Lnorm{2}}^2 + C_{\mu,\lambda} \norm{z^{\alpha/2}Z_1^{\alpha/2} v_{12}}{\Lnorm{2}}^2  \bigr\rbrace + \mu \norm{\nablah v_{12}}{\Lnorm{2}}^2 + \mu \norm{\dz v_{12}}{\Lnorm{2}}^2 \\
		& ~~~~ \lesssim \mathcal H(\mathcal E_1(t), \mathcal E_2(t),1) \bigl(  \norm{\nablah v_1}{\Hnorm{2}}^2 + \norm{\nablah v_2}{\Hnorm{2}}^2 + \norm{\nabla v_1}{\Hnorm{1}}^2\\
		& ~~~~ + \norm{\nabla v_2}{\Hnorm{1}}^2 + \norm{\nabla \dt v_2}{\Lnorm{2}}^2 + 1\bigr) \bigl( \norm{Z_{12}}{\Lnorm{2}}^2 + \norm{z^{\alpha/2} Z_1^{\alpha/2} v_{12}}{\Lnorm{2}}^2 \bigr).
	\end{align*}
	Then after applying Gr\"onwall's inequality to the above inequality, one has 
	$$ \norm{Z_{12}}{\Lnorm{2}} \equiv 0 \quad \text{and} \quad   \norm{z^{\alpha/2} Z_1^{\alpha/2} v_{12}}{\Lnorm{2}} \equiv 0.$$
	Therefore $ (Z_1, v_1 ) \equiv (Z_2, v_2 ) $. This finishes the proof of the proposition. 
\end{proof}

\section{Decay estimates and asymptotic stability}\label{sec:decay}

In this section, we will establish the asymptotic stability theory. In fact, we will show the following proposition:
\begin{proposition}\label{prop:decay-est}
	Under the same assumptions as in Proposition \ref{prop:stability-theory} and $ 0 < \alpha < 3 $, there is a constant $ 0 < \varepsilon_5 < \varepsilon_2 $ where $ \varepsilon_2 $ is given in Proposition \ref{prop:stability-theory}, such that if $ \mathcal E(t) < \varepsilon_5 $, for $ t \in (0,T) $, the following inequality holds,
	\begin{equation}\label{decay-est}
		\mathcal E(t) \leq e^{-C_1t} C_2 \mathcal E_0,
	\end{equation}
	for some positive constants $ 0 < C_1, C_2 < \infty $. 
\end{proposition}
In order to show Proposition \ref{prop:decay-est}, we need the following lemma concerning the Poincar\'e type inequality.
\begin{lemma}\label{lm:poincare-ineq}
	Under the same assumption as in Proposition \ref{lm:basic-energy} and $ 0 < \alpha < 3 $,
	we have
	\begin{equation}\label{poincare-ineq}
	\begin{aligned}
		& \int z^\alpha \abs{v}{2} \idx \leq C \int z^{\beta}  \bigl(\abs{\dz v}{2} + \abs{\partial_{zz} v}{2} + \abs{\partial_{zzz} v}{2} \bigr) \idx \\
		& ~~~~ + C \norm{z^{\alpha/2} \nablah v}{L^2}^2 + C \norm{z^{\alpha/2} v}{\Lnorm{2}}^2 \norm{Z^{\alpha+1}-1}{\Hnorm{2}}^2,
	\end{aligned}
	\end{equation}
	for some positive constant $ 0 < C < \infty $ and $ \beta $ satisfying,
	\begin{equation*}
		0 \leq \beta < 4 - \alpha.
	\end{equation*}
	In particular, one can choose $ \beta = 1 $. 
\end{lemma}
\begin{proof}
	First, as we have mentioned in the introduction, from \subeqref{FB-CPE}{2}, \eqref{eq:density}, the conservation of momentum holds,
	\begin{equation*}
		\dfrac{d}{dt} \int z^\alpha Z^{\alpha+1} v \idx  = 0.
	\end{equation*}
	Therefore, we have from \eqref{vanish-of-initial-momentum},
	\begin{equation*}
		\int z^\alpha Z^{\alpha+1} v \idx = \int z^\alpha Z_0^{\alpha+1} v_0 \idx = 0,
	\end{equation*}
	or
	\begin{equation}\label{DC-001}
		\int z^\alpha  v \idx = \int z^\alpha \bigl( 1 - Z^{\alpha+1} \bigr) v \idx.
	\end{equation}
	Next, we introduce a new coordinate system $ (x,y,z') $ given by 
	\begin{equation}
		x = x, y = y, z' = \dfrac{1}{\alpha+1} z^{\alpha+1}.
	\end{equation}
	Then $ \partial_{z'} =  z^{-\alpha} \partial_z $, and $ \int_\Omega z^\alpha \cdot \idx = \int_{\Omega'} \cdot \,dxdydz' $ where $ \Omega' = \Omega_h \times (0, \frac{1}{\alpha+1} ) $ and $ |\Omega'| < \infty $. Then we apply the standard Poincar\'e inequality as follows, 
	\begin{align*}
		& \int_\Omega z^\alpha \abs{v}{2} \idx = \int_{\Omega'} \abs{v}{2} \,dxdydz' \lesssim \int_{\Omega'} \abs{v - \int_{\Omega'} v \,dxdydz'}{2} \,dxdydz' \\
		& ~~~~ + \abs{\int_{\Omega'} v \,dxdydz'}{2} \lesssim \int_{\Omega'} \abs{\partial_{z'} v}{2} \,dxdydz' + \int_{\Omega'} \abs{\nablah v}{2} \,dxdydz' + \abs{\int_{\Omega} z^\alpha v \idx  }{2} \\
		& ~~~~ = \int_\Omega z^{-\alpha} \abs{\dz v}{2} \idx  + \int_{\Omega} z^\alpha \abs{\nablah v}{2} \,dxdydz + \abs{\int_\Omega z^\alpha(1-Z^{\alpha+1}) v \idx }{2}.
	\end{align*}
	Therefore, for 
	$ 0 < \alpha < 3 $, applying Hardy's inequality to the right of the above inequality yields,
	\begin{align*}
		& \int z^\alpha \abs{v}{2} \idx \lesssim \int z^{-\alpha - \sigma} \abs{\dz v}{2} \idx  + \int_{\Omega} z^\alpha \abs{\nablah v}{2} \,dxdydz + \int z^{2\alpha} \abs{1-Z^{\alpha+1}}{2} \abs{v}{2} \idx\\
		& ~~~~ \lesssim \int z^{2-\alpha-\sigma} \bigl( \abs{\dz v}{2} + \abs{\partial_{zz} v}{2} \bigr) \idx  + \int_{\Omega} z^\alpha \abs{\nablah v}{2} \,dxdydz+ \norm{Z^{\alpha+1}-1}{\Lnorm{\infty}}^2  \int z^{2\alpha}\abs{v}{2} \idx\\
		& ~~~~\lesssim \int z^{4-\alpha - \sigma}  \bigl(\abs{\dz v}{2} + \abs{\partial_{zz} v}{2} + \abs{\partial_{zzz} v}{2} \bigr) \idx \\
		& ~~~~ ~~~~ + \int_{\Omega} z^\alpha \abs{\nablah v}{2} \,dxdydz +  \norm{z^{\alpha/2} v}{\Lnorm{2}}^2 \norm{Z^{\alpha+1}-1}{\Hnorm{2}}^2,
	\end{align*}
	where $ \sigma > 0 $ is a small number such that the following conditions hold
	\begin{gather*}
		- \alpha - \sigma \neq -1, ~ 2-\alpha-\sigma > -1.
	\end{gather*}
	For $ 0 < \alpha < 3 $, such $ \sigma $ alway exists. 
	This finishes the proof. 
\end{proof}
Now we will show Proposition \ref{prop:decay-est}.
\begin{proof}[Proof of Proposition \ref{prop:decay-est}]
	From \eqref{apriori-ineq-001}, \eqref{BE-004}, \eqref{apriori-ineq-002}, \eqref{apriori-ineq-003}, \eqref{apriori-ineq-004}, \eqref{apriori-ineq-005}, \eqref{apriori-ineq-006},  \eqref{SPE-003}, \eqref{apriori-ineq-007}, there are constants $ c_1, c_2 \cdots c_9 $ such that the quantities defined by
	\begin{align*}
		& \mathfrak E_d(t) : = \dfrac{1}{2} \int z^\alpha Z^{\alpha+1} \abs{v}{2} \idx + \dfrac{g}{\alpha+2} \int \biggl( \bigl(Z^{\alpha+2}-1\bigr) \\
		& ~~~~ ~~~~ - \dfrac{\alpha+2}{\alpha+1}\bigl( Z^{\alpha+1}-1\bigr) \biggr) \idx 
		 + c_2\int \biggl( \dfrac{\mu}{2} \abs{\nablah v}{2} + \dfrac{\mu+\lambda}{2} \abs{\dvh v}{2} \\
		& ~~~~ ~~~~ + \dfrac{\mu}{2}\abs{\dz v}{2} \biggr) \idx
		 - \dfrac{c_2}{\alpha+1} \int g z^\alpha \bigl( Z^{\alpha+1}-1\bigr) \dvh v \idx \\
		& ~~~~ + \dfrac{c_3}{2} \int z^\alpha Z^{\alpha+1} \abs{v_t}{2} \idx 
		 + \dfrac{c_3 g}{2} \int Z^\alpha \abs{Z_t}{2} \idx + \dfrac{c_4}{2} \int z^\alpha Z^{\alpha+1} \abs{\nablah v}{2} \idx \\
		& ~~~~ + \dfrac{c_4 g}{2} \int Z^\alpha \abs{\nablah Z}{2} \idx  
		 + \dfrac{c_5}{2} \int z^\alpha Z^{\alpha+1} \abs{\nablah^2 v}{2} \idx + \dfrac{c_5 g}{2} \int \abs{\nablah^2 Z}{2} \idx \\
		& ~~~~ + \dfrac{c_7}{2} \int z^\alpha Z^{\alpha+1} \abs{v_z}{2}\idx 
		 + \dfrac{c_9}{2} \int z^\alpha Z^{\alpha+1} \abs{\nablah v_{z}}{2} \idx
		, \\
		& \mathfrak G_d(t) : = \mu/2 \int Z \abs{\nabla v}{2} \idx + c_1 \norm{Z_t}{\Lnorm{2}}^2 + c_2/2 \int z^\alpha Z^\alpha \abs{v_t}{2} \idx \\
		& ~~~~ + c_3\mu/2 \int Z \abs{\nabla v_t}{2} \idx + c_4\mu/2\int Z \abs{\nabla \nablah v}{2} \idx \\
		& ~~~~ + c_5 \mu/2\int Z\abs{\nabla \nablah^2 v}{2}\idx 
		 + c_6 \norm{\nablah Z}{\Lnorm{2}}^2 + c_7\mu/4 \int Z\abs{\nabla v_{z}}{2} \idx \\
		& ~~~~ + c_8 \norm{\nablah^2 Z}{\Lnorm{2}}^2 
		 + c_9 \mu/4 \int Z \abs{\nabla \nablah v_z}{2} \idx,
	\end{align*}
	satisfy
	\begin{equation*}
		\dfrac{d}{dt} \mathfrak E_d(t)  + \mathfrak G_d(t) \leq \mathcal H(\mathcal E(t)) \mathfrak G_d(t).
	\end{equation*}
	Notice that, following similar arguments as in Proposition \ref{lm:equal-of-energies}, one can choose $ c_i's ~ (i = 1,2 \cdots 9) $ such that there is a constant $ C < \infty $ such that 
	\begin{gather*}
		C^{-1}\mathcal E(t)  \leq \mathfrak E_d(t) \leq C \mathcal E(t), \\
		C^{-1} \mathfrak D(t) \leq \mathfrak G_d (t) \leq C \mathfrak D(t) .
	\end{gather*}
	Moreover, from Proposition \ref{lm:diss-dzzz-v}, Lemma \ref{lm:poincare-ineq}, 
	and \eqref{aprasm:dissipation}, the above inequality can be extended to 
	\begin{equation*}
		\dfrac{d}{dt} \mathfrak E_d(t)  + \mathfrak G_d(t) + c_{10}\norm{z^{\alpha/2} v}{\Lnorm{2}}^2 \leq \mathcal H(\mathcal E(t)) \mathfrak G_d(t),
	\end{equation*}
	for some positive constant $ c_{10} $. Then for $ \mathcal E(t) $ small enough, one has 
	\begin{equation*}
		\dfrac{d}{dt} \mathfrak E_d(t) + C_1\mathfrak E_d(t) \leq  \dfrac{d}{dt} \mathfrak E_d(t)  + \dfrac{1}{2} \mathfrak G_d(t) + c_{10}\norm{z^{\alpha/2} v}{\Lnorm{2}}^2 \leq 0,
	\end{equation*}
	for some positive constant $ C_1 <\infty $.
	Therefore, one has
	\begin{equation*}
		\mathfrak E_d(t) \leq e^{-C_1t} \mathfrak E_d(0),
	\end{equation*}
	which finishes the proof. 
\end{proof}

\section{Higher regularity}\label{sec:regularity-2}

In this section, we show the regularity of the global solution with more regular initial data. In particular, this shows that the change of coordinates in \eqref{new-coordinates} is regular, and it justifies the equivalence of \eqref{isen-CPE-fb} and \eqref{rfeq:isen-CPE-fb} (or \eqref{FB-CPE} with the simplified viscosity tensor). Now we write down the equations after differentiating \subeqref{eq:dh-FBCPE}{1}, \subeqref{eq:dh-FBCPE}{3} horizontally once and \subeqref{eq:dh-FBCPE}{2} horizontally twice . It holds,
\begin{equation}\label{eq:dhhh-FBCPE}
	\begin{cases}
		z^\alpha Z^\alpha\bigl( \dt v_{hh} + v\cdot \nablah v_{hh} + W \dz v_{hh}\bigr) + z^\alpha Z^\alpha\bigl( 2 v_h \cdot \nablah v_{h} + v_{hh} \cdot \nablah v \\
		~~~~ ~~~~ + 2 W_h \dz v_{h} + W_{hh} \dz v\bigr) + 2z^\alpha (Z^\alpha)_h \bigl( \dt v + v\cdot \nablah v + W \dz v\bigr)_{h}\\
		~~~~ ~~~~ + z^\alpha (Z^\alpha)_{hh} \bigl( \dt v + v\cdot \nablah v + W \dz v\bigr) + g(z^\alpha Z^\alpha \nablah Z)_{hh}\\
		~~~~ = \mu \deltah v_{hh} + (\mu+\lambda) \nablah \dvh v_{hh} + \mu \partial_{zz} v_{hh} + \bigl( \mu \nablah \log Z \cdot \nablah v \\
		~~~~ ~~~~ + (\mu+\lambda) \dvh v \nablah \log Z \bigr)_{hh},
		\\
		\dt Z_{hhh} + (\alpha+1) \overline{z^\alpha v} \cdot \nablah Z_{hhh} + \overline{z^\alpha \dvh v} Z_{hhh} + 3 (\alpha+1) \overline{z^\alpha v_h} \cdot \nablah Z_{hh} \\
		~~~~ ~~~~ + 3\overline{z^\alpha \dvh v_h} Z_{hh} +  3 (\alpha+1) \overline{z^\alpha v_{hh}} \cdot \nablah Z_{h}  + 3\overline{z^\alpha \dvh v_{hh}} Z_{h} \\
		~~~~ ~~~~ + (\alpha+1) \overline{z^\alpha v_{hhh}}\cdot\nablah Z + \overline{z^\alpha \dvh v_{hhh}} Z = 0,\\
		z^\alpha W_{hh} = z^{\alpha+1} \bigl( (\alpha+1) \overline{z^\alpha v} \cdot \nablah (\log Z)_{hh} + 2(\alpha+1) \overline{z^\alpha v_{h}} \cdot \nablah (\log Z)_{h} \\
		~~~~ ~~~~ + (\alpha +1) \overline{z^\alpha v_{hh} } \cdot \nablah \log Z + \overline{z^\alpha \dvh v_{hh}}\\
		~~~~ ~~~~ - (\alpha+1) \int_0^z \xi^\alpha v \,d\xi \cdot \nablah (\log Z)_{hh}\\
		~~~~ ~~~~ - 2(\alpha+1) \int_0^z \xi^\alpha v_h \,d\xi \cdot \nablah (\log Z)_{h}\\
		~~~~ ~~~~ - (\alpha+1) \int_0^z \xi^\alpha v_{hh} \,d\xi \cdot \nablah \log Z - \int_0^z \xi^\alpha \dvh v_{hh} \,d\xi.
	\end{cases}
\end{equation}

\begin{proposition}\label{prop:Higher-order}
Suppose in addition to the assumptions as in Proposition \ref{prop:global-stability}, $ (Z_0, v_0) $ satisfies \eqref{high-regular-initial}. That is,
\begin{equation*}
	\norm{Z_0}{\Hnorm{3}} + \norm{z^{\alpha/2} \nablah^3 v_0}{\Lnorm{2}} < \infty. 
\end{equation*}
Then for any $ T > 0 $, there is a positive constant $ 0< C_T < \infty $ such that,
\begin{equation}\label{higher-regularity-001}
	\begin{aligned}
	& \sup_{0<t<T} \bigl\lbrace \norm{z^{\alpha/2} \nablah^3 v}{\Lnorm{2}}^2 + \norm{\nablah^3 Z}{\Lnorm{2}}^2 \bigr\rbrace + \int_0^T \norm{\nabla \nablah^3 v}{\Lnorm{2}}^2 \,dt \\
	& ~~~~ \leq C_T \bigl( \norm{z^{\alpha/2} \nablah^3 v_0}{\Lnorm{2}}^2 
	 + \norm{\nablah^3 Z_0}{\Lnorm{2}}^2 + 1 \bigr).
	\end{aligned}
\end{equation}
Moreover, 
$ \norm{\dt Z}{\Lnorm{\infty}}, \norm{\nablah Z}{\Lnorm{\infty}} < \infty $ for any finite $ t< \infty $. 
\end{proposition}
\begin{proof}
	Taking inner product of \subeqref{eq:dhhh-FBCPE}{1} with $ - (Z \partial_{hhh}v)_h  $ yields,
	\begin{align*}
		& \dfrac{d}{dt} \biggl\lbrace \dfrac{1}{2} \int z^\alpha Z^{\alpha+1} \abs{v_{hhh}}{2} \idx + \dfrac{g}{2} \int \abs{Z_{hhh}}{2} \idx  \biggr\rbrace \\
		& ~~~~ + \int Z \bigl( \mu \abs{\nablah v_{hhh}}{2} + (\mu+\lambda) \abs{\dvh v_{hhh}}{2} + \mu \abs{\dz v_{hhh}}{2} \bigr) \idx \\
		& = \int \bigl( \mu \nablah (\log Z)_h \cdot\nablah v_{hh} +(\mu+\lambda) \dvh v_{hh} \nablah (\log Z)_h \bigr) \cdot Z v_{hhh}\\
		& ~~~~ - \bigl(2 \mu \nablah (\log Z)_{h} \cdot \nablah v_{h} + 2(\mu+\lambda) \dvh v_h \nablah (\log Z)_h \\
		& ~~~~ ~~~~ + \mu \nablah (\log Z)_{hh} \cdot \nablah v + (\mu+\lambda) \dvh v \nablah (\log Z)_{hh} \bigr) \cdot (Z v_{hhh})_h \idx   \\
		& ~~~~ - \int z^\alpha Z \bigl( (Z^\alpha)_h  (\dt v_{hh} + v \cdot \nablah v_{hh}) + Z^{\alpha} v_h \cdot \nablah v_{hh} \bigr)\cdot v_{hhh} \idx \\
		& ~~~~ - \int z^\alpha Z \bigl( (Z^\alpha)_h W \dz v_{hh} + Z^\alpha W_h \dz v_{hh} \bigr)\cdot v_{hhh} \idx \\
		& ~~~~ + \int z^\alpha \bigl( Z^\alpha (2 v_h \cdot \nablah v_h + v_{hh} \cdot \nablah v) + 2(Z^\alpha)_h (\dt v + v \cdot \nablah v)_h \\
		& ~~~~ ~~~~ + (Z^\alpha)_{hh}(\dt v + v\cdot \nablah v)  \bigr) \cdot \bigl( Z v_{hhh} \bigr)_h \idx \\
		& ~~~~ + \int z^\alpha \bigl( Z^\alpha (2 W_h \dz v_h + W_{hh} \dz v) + 2 (Z^\alpha)_h (W \dz v)_h \\
		& ~~~~ ~~~~ + (Z^\alpha)_{hh} W \dz v \bigr) \cdot \bigl( Z v_{hhh} \bigr)_h \idx \\
		& ~~~~ + g\int z^\alpha \bigl( (Z^\alpha-1) Z_{hhh} + 2 (Z^\alpha)_h Z_{hh} + (Z^\alpha)_{hh} Z_h \bigr) \dvh (Z v_{hhh}) \idx \\ 
		& ~~~~ - g \int Z_{hhh} \bigl( \dfrac{1-\alpha}{2} \overline{z^\alpha\dvh v} Z_{hhh} + 3(\alpha+1) \overline{z^\alpha v_h} \cdot \nablah Z_{hh} \\
		& ~~~~ ~~~~ + 3 \overline{z^\alpha \dvh v_h} Z_{hh} + 3 (\alpha+1) \overline{z^\alpha v_{hh}} \cdot \nablah Z_h \\
		& ~~~~ ~~~~ + 3 \overline{z^\alpha \dvh v_{hh}} Z_h + \alpha \overline{z^\alpha v_{hhh}} \cdot \nablah Z \bigr) \idx =: \sum_{k=1}^{7} H_k.
	\end{align*}
	Now we show the estimates for $ H_k $'s in the following. Similarly, applying integration by parts and making use of  H\"older's, Minkowski's, the Sobolev embedding, and Young's inequalities yield that,
	\begin{align*}
		& H_1 \lesssim \int_0^1 \bigl( \hnorm{\nablah^2 Z}{\Lnorm{2}} +\hnorm{\nablah Z}{\Lnorm{4}}^2 \bigr) \hnorm{v_{hhh}}{4}^2 + \bigl( (\hnorm{\nablah^2 Z}{\Lnorm{4}} + \hnorm{\nablah Z}{\Lnorm{8}}^2) \hnorm{\nablah^2 v}{\Lnorm{4}} \\
		& ~~~~ + ( \hnorm{\nablah^3 Z}{\Lnorm{2}} + \hnorm{\nablah^2 Z}{\Lnorm{4}} \hnorm{\nablah Z}{\Lnorm{4}} + \hnorm{\nablah Z}{\Lnorm{6}}^3) \hnorm{\nablah v}{\Lnorm{\infty}} \bigr) \bigl( \hnorm{v_{hhhh}}{\Lnorm{2}} \\
		& ~~~~ + \hnorm{Z_h}{\Lnorm{4}} \hnorm{v_{hhh}}{\Lnorm{4}} \bigr) \,dz \lesssim \delta \norm{\nablah v_{hhh}}{\Lnorm{2}}^2 + C_\delta \bigl( \norm{\nablah Z}{\Hnorm{1}}^8 + 1\bigr) \norm{\nablah v}{\hHnorm{2}}^2 \\
		& ~~~~ + C_\delta \bigl( \norm{\nablah Z}{\Hnorm{1}}^2 + 1\bigr) \norm{\nablah v}{\hHnorm{2}}^2\norm{\nablah Z}{\Hnorm{2}}^2,\\
		& H_2 = \dfrac{\alpha}{\alpha+1} \int z^\alpha \bigl( (Z^{\alpha+1})_{hh} \dt v_h \cdot v_{hhh} + (Z^{\alpha+1})_h \dt v_h \cdot v_{hhhh} \bigr) \idx \\
		& ~~~~ - \int z^\alpha Z \bigl( (Z^\alpha)_h v \cdot \nablah v_{hh} + Z^\alpha v_h \cdot \nablah v_{hh} \bigr) \cdot v_{hhh} \idx \lesssim \delta \norm{\nabla v_{hhh}}{\Lnorm{2}}^2\\
		& ~~~~ + \delta \norm{\nablah v}{\Hnorm{2}}^2 + C_\delta \norm{\nablah Z}{\Hnorm{1}}^4  \norm{\nabla \dt v}{\Lnorm{2}}^2  + C_\delta \norm{\nabla \dt v}{\Lnorm{2}}^2 \norm{\nablah Z}{\Hnorm{2}}^2\\
		& ~~~~ + C_\delta \bigl( \norm{\nablah Z}{\Hnorm{1}}^4 + 1 \bigr)\norm{v}{\Hnorm{1}}^4 \norm{\nablah v}{\Hnorm{2}}^2,\\
		& H_4 \lesssim \delta \norm{\nabla v_{hhh}}{\Lnorm{2}}^2 + C_\delta \norm{\nablah Z}{\Hnorm{1}}^4 \norm{v}{\Hnorm{1}}^2 \norm{\nablah v}{\Hnorm{2}}^2 \\
		& ~~~~ + C_\delta \norm{\nablah Z}{\Hnorm{1}}^4 \bigl( \norm{z^{\alpha/2}v_t}{\Lnorm{2}}^2 + \norm{\nabla v_t}{\Lnorm{2}}^2 \bigr)  + C_\delta \bigl( \norm{z^{\alpha/2} v}{\hHnorm{1}}^2 \norm{\nablah v}{\Hnorm{2}}^2\\
		& ~~~~ + \norm{v}{\Hnorm{1}}^2 \norm{\nablah v}{\Hnorm{2}}^2 + \norm{\nablah v}{\Hnorm{2}}^2 + \norm{z^{\alpha/2} v_t}{\Lnorm{2}}^2 + \norm{\nabla \dt v}{\Lnorm{2}}^2 \bigr) \\
		& ~~~~ \times \norm{\nablah Z}{\Hnorm{2}}^2 + C_\delta \norm{z^{\alpha/2} v}{\hHnorm{2}}^2 \norm{\nablah v}{\Hnorm{2}}^2, \\
		& H_6 \lesssim \delta \norm{\nabla v_{hhh}}{\Lnorm{2}}^2 + \delta \norm{v_{hhh}}{\Lnorm{2}}^2 + C_\delta \bigl( \norm{Z-1}{\Hnorm{2}}^4 + 1\bigr) \norm{\nablah Z}{\Hnorm{2}}^2,\\
		& H_7 \lesssim \bigl( \norm{\nablah v}{\Hnorm{2}}^2 + 1 \bigr) \norm{\nablah Z}{\Hnorm{2}}^2.
	\end{align*}
	On the other hand, 
	\begin{align*}
		& H_{3} + H_5 = \int \bigl(Z(Z^\alpha)_h \dz v_{hh} \cdot v_{hhh} + 2 (Z^\alpha)_h \dz v_h \cdot (Z v_{hhh})_h \\
		& ~~~~ + (Z^\alpha)_{hh} \dz v \cdot (Z v_{hhh})_h \bigr) z^\alpha W \idx + \int \bigl( Z^{\alpha+1} \dz v_{hh} \cdot v_{hhh} \\
		& ~~~~ + 2 Z^\alpha\dz v_h \cdot (Z v_{hhh})_h + 2(Z^\alpha)_h \dz v \cdot (Z v_{hhh})_h \bigr) z^\alpha W_h \idx\\
		& ~~~~ + \int Z^\alpha \dz v \cdot (Z v_{hhh})_h z^\alpha W_{hh} \idx \lesssim \int_0^1 \bigl( \hnorm{Z_h}{\Lnorm{4}} \hnorm{\dz v_{hh}}{\Lnorm{2}} \hnorm{v_{hhh}}{\Lnorm{4}} \\
		& ~~~~ + (\hnorm{Z_h}{\Lnorm{4}} \hnorm{\dz v_h}{\Lnorm{4}} + \hnorm{Z_{hh}}{\Lnorm{2}} \hnorm{\dz v}{\Lnorm{\infty}} + \hnorm{Z_{h}}{\Lnorm{4}}^2 \hnorm{\dz v}{\Lnorm{\infty}} ) ( \hnorm{v_{hhhh}}{\Lnorm{2}} \\
		& ~~~~ + \hnorm{Z_h}{\Lnorm{4}}\hnorm{v_{hhh}}{\Lnorm{4}})  \bigr)  \hnorm{z^\alpha W}{\Lnorm{\infty}} \,dz + \int_0^1 \bigl( \hnorm{\dz v_{hh}}{\Lnorm{2}} \hnorm{v_{hhh}}{\Lnorm{4}} \\
		& ~~~~ + (\hnorm{\dz v_h}{\Lnorm{4}} + \hnorm{Z_h}{\Lnorm{8}} \hnorm{\dz v}{\Lnorm{8}}  ) (\hnorm{v_{hhhh}}{\Lnorm{2}} + \hnorm{Z_h}{\Lnorm{4}} \hnorm{v_{hhh}}{\Lnorm{4}})    \bigr) \hnorm{z^\alpha W_h }{\Lnorm{4}} \,dz\\
		& ~~~~ + \int_0^1  \hnorm{\dz v}{\Lnorm{\infty}}\bigl( \hnorm{v_{hhhh}}{\Lnorm{2}} + \hnorm{Z_h}{\Lnorm{4}} \hnorm{v_{hhh}}{\Lnorm{4}} \bigr) \hnorm{z^\alpha W_{hh}}{\Lnorm{2}} \,dz \\
		& ~~~~ \lesssim \delta \norm{\nablah v_{hhh}}{\Lnorm{2}}^2 + \delta \norm{\nabla v_{hh}}{\Lnorm{2}}^2 + C_\delta \bigl(\mathcal H(\mathcal E(t)) + 1\bigr) \\
		& ~~~~ \times \bigl( \norm{\nablah v}{\Hnorm{2}}^2 + \norm{\nabla v}{\Lnorm{2}}^2 \bigr) \bigl( \norm{z^{\alpha/2}\nablah^3 v}{\Lnorm{2}}^2 + \norm{\nablah^3 Z}{\Lnorm{2}}^2  + 1\bigr),
	\end{align*}
	where we have applied the following facts
	\begin{align*}
		& \hnorm{z^\alpha W}{\Lnorm{\infty}} \lesssim \norm{z^{\alpha/2} v}{\hHnorm{3}} + \norm{z^{\alpha/2} v}{\hHnorm{2}} \norm{\nablah Z}{\Hnorm{2}}, \\
		& \hnorm{z^\alpha W_h}{\Lnorm{4}} \lesssim \norm{z^{\alpha/2} v}{\hHnorm{3}} + \norm{z^{\alpha/2} v}{\hHnorm{2}}\bigl( \norm{\nablah Z}{\Hnorm{2}}+ \norm{\nablah Z}{\Hnorm{1}}^2  \bigr),\\
		& \hnorm{z^\alpha W_{hh}}{\Lnorm{2}} \lesssim \norm{z^{\alpha/2} v}{\hHnorm{3}} + \norm{z^{\alpha/2} v}{\hHnorm{2}}\bigl( \norm{\nablah Z}{\Hnorm{2}}+ \norm{\nablah Z}{\Hnorm{1}}^3 \bigr).
	\end{align*}
	Summing up the above inequalities with a small enough $ \delta > 0 $ yields,
	\begin{align*}
		& \dfrac{d}{dt} \biggl\lbrace \dfrac{1}{2} \int z^\alpha Z^{\alpha+1} \abs{\nablah^3 v}{2} \idx + \dfrac{g}{2} \int \abs{\nablah^3 Z}{2} \idx  \biggr\rbrace \\
		& ~~~~ + \mu/2 \int Z \bigl( \abs{\nablah^4 v}{2} + \abs{\dz \nablah^3 v}{2} \bigr) \idx\\
		& \lesssim \bigl(\mathcal H(\mathcal E(t)) + 1\bigr)\bigl( \norm{\nablah v}{\Hnorm{2}}^2 + \norm{\nabla v}{\Lnorm{2}}^2 + \norm{\nabla v_t}{\Lnorm{2}}^2 + 1 \bigr) \\
		& ~~~~ \times \bigl( \norm{z^{\alpha/2} Z^{\alpha+1} \nablah^3 v}{\Lnorm{2}}^2 + \norm{\nablah^3 Z}{\Lnorm{2}}^2  + 1\bigr).
	\end{align*}
	Then applying the Gr\"onwall's inequality to above inequality will show \eqref{higher-regularity-001}. On the other hand, notice \eqref{eq:movingboundary}. One has
	\begin{equation*}
		\norm{\dt Z}{\Lnorm{\infty}} \leq \norm{z^{\alpha/2}v}{\hHnorm{2}}\norm{\nablah Z}{\Hnorm{2}} + \norm{z^{\alpha/2} v}{\hHnorm{3}}. 
	\end{equation*}
	This finishes the proof.
\end{proof}

\section{Global stability theory of a non-conservative system}\label{sec:non-conservative}

System \eqref{FB-CPE}, as discussed in Remark \ref{rm:viscosity}, satisfies the conservations of mass, momentum, and energy; see Proposition \refeq{lm:basic-energy}. To provide a foundation for application, e.g. simulation etc., we introduce the following non-conservative system with constant viscosity:
\begin{equation}\label{FB-CPE-non-conservative}
	\begin{cases}
		(\alpha + 1) z(\dt Z + v \cdot \nablah Z) + z Z ( \dvh v + \partial_z W) \\
		~~~~ ~~~~ + \alpha Z W = 0 & \text{in} ~ \Omega,\\
		\rho (\dt v + v \cdot \nablah v + W \partial_z v + g \nablah Z) = \Delta_h v + \partial_{zz} v & \text{in} ~ \Omega,\\
		\partial_z Z = 0 & \text{in} ~ \Omega,
	\end{cases}
\end{equation}
with the boundary condition \eqref{FB-BC-CPE}. The same identities \eqref{eq:movingboundary} and \eqref{eq:verticalvelocity} still hold. For the initial data, we assume again
\begin{equation}\label{initial-mass-non-conserv}
	\int Z^{\alpha+1} \idx\vert_{t=0} = \int Z_0^{\alpha+1}\idx = 1.
\end{equation}
Then one can check, for the non-conservative system \eqref{FB-CPE-non-conservative}, the conservation of mass still holds, i.e.,
\begin{equation}\label{mass-non-conserv}
	\int Z^{\alpha+1}\idx = \int Z_0^{\alpha+1}\idx = 1
\end{equation}
for a solution. 

Then one can easily check the calculation, and conclude that, the same global stability theory holds for solutions $ (v, Z) $ to \eqref{FB-CPE-non-conservative}. That is
\begin{theorem}[Global Stability for the Non-conservative System]\label{thm:global-non-conserv}
	Under the same assumptions as in Theorem \ref{thm:global}, there exists a unique global solution $ (v,Z) $ to system \eqref{FB-CPE-non-conservative}, which satisfies the same regularity and estimates as in Theorem \ref{thm:global}.
\end{theorem}

On the other hand, to describe the asymptotic stability, let
\begin{equation}\label{def:momentum-non-conserv}
	\delta(t):= \dfrac{\int \rho v Z\idx}{\int \rho Z \idx} = (\alpha+1) \int z^\alpha Z^{\alpha + 1} v \idx.
\end{equation}
We remind readers that we have assumed $ A_\alpha = 1 $ for the sake of clear presentation. 
Then from \subeqref{FB-CPE-non-conservative}{2}, one can verify
\begin{equation}\label{eq:momemtum-non-conserv}
	\dfrac{d}{dt}\delta(t) = \delta'(t) = - (\alpha+1)\int (\nablah Z \cdot \nablah) v \idx \lesssim\norm{\nablah Z}{L^2} \norm{\nablah v}{L^2}.
\end{equation}
Then from Theorem \refeq{thm:global-non-conserv}, we have that there exists $ \delta_0 > 0 $ such that 
\begin{equation}\label{lim:momentum-non-conserv}
	\lim_{t\rightarrow \infty }\delta(t) = \delta_\infty.
\end{equation}

Moreover, the energy functional for the non-conservative system is defined by replacing $ v $ with $ v - \delta(t) $, i.e.,
\begin{equation}\label{def:total-energy-non-conserv}
	\begin{aligned}
		& \mathcal E_n(t) := \norm{z^{\alpha/2} (v-\delta(t))}{\hHnorm{2}}^2 + \norm{z^{\alpha/2} (v-\delta(t))_t}{\Lnorm{2}}^2 \\
		& ~~~~ + \norm{z^{\alpha/2}(v-\delta(t))_z}{\hHnorm{1}}^2 + \norm{(v-\delta(t))}{\Hnorm{1}}^2 \\
		& ~~~~ + \norm{Z - 1}{\Hnorm{2}}^2 + \norm{Z_t}{\Lnorm{2}}^2,
	\end{aligned}
\end{equation}
with the initial energy $ \mathcal E_{n,0} $ defined similarly. 

Rewrite \subeqref{FB-CPE-non-conservative}{2} as
\begin{equation}\label{eq:momentum-non-conserv}
	\begin{aligned}
	& z^\alpha Z^\alpha (\dt (v-\delta(t)) + v \cdot \nablah (v-\delta(t)) + W\dz (v-\delta(t)) + g\nablah Z) \\
	& \qquad = \Delta_h(v-\delta(t)) + \partial_{zz}(v-\delta(t)) - z^\alpha Z^\alpha \delta'(t).
	\end{aligned}
\end{equation}
Then taking the $ L^2 $-inner product of \eqref{eq:momemtum-non-conserv} with $ Z (v-\delta(t)) $ and following the same calculation as in Proposition \refeq{lm:basic-energy} lead to
\begin{equation}\label{energy-non-conserv}
	\begin{aligned}
	& \dfrac{d}{dt}\biggl\lbrace \dfrac{1}{2} \int z^\alpha Z^{\alpha+1}|v-\delta(t)|^2 \idx + \dfrac{g}{\alpha+2} \int (Z^{\alpha+2}-1) - \dfrac{\alpha+2}{\alpha+1} (Z^{\alpha+1} - 1) \idx \biggr\rbrace\\
	& \quad + \int Z (|\nablah(v-\delta(t))|^2 + |\partial_z(v-\delta(t))|^2) \idx \\
	& = - \int \nablah Z \cdot \nablah (v-\delta(t)) \cdot (v-\delta(t)) \idx - \int z^\alpha Z^{\alpha + 1} \delta'(t)(v-\delta(t)) \idx\\
	& \leq \norm{\nablah Z}{L^6} \norm{\nablah(v-\delta(t))}{L^3} \norm{v-\delta(t)}{L^2}\\
	& \lesssim \norm{v-\delta(t)}{H^1}\norm{\nablah Z}{H^1} \norm{\nablah(v-\delta(t))}{H^1}.
	\end{aligned}
\end{equation}

Then from \eqref{energy-non-conserv}, \eqref{BE-004} \eqref{apriori-ineq-002}, \eqref{apriori-ineq-003}, \eqref{apriori-ineq-004}, \eqref{apriori-ineq-005}, \eqref{apriori-ineq-006}, \eqref{SPE-003}, \eqref{apriori-ineq-007}, there exist constants $ d_1, d_2 \cdots d_9 $ such that the quantities defined by
\begin{align*}
	& \mathfrak E_{n,d}(t) : = \dfrac{1}{2} \int z^\alpha Z^{\alpha+1} \abs{v-\delta(t)}{2} \idx + \dfrac{g}{\alpha+2} \int \biggl( \bigl(Z^{\alpha+2}-1\bigr) \\
	& ~~~~ ~~~~ - \dfrac{\alpha+2}{\alpha+1}\bigl( Z^{\alpha+1}-1\bigr) \biggr) \idx 
	 + d_2\int \biggl( \dfrac{1}{2} \abs{\nablah (v-\delta(t))}{2}  + \dfrac{1}{2}\abs{\dz (v-\delta(t))}{2} \biggr) \idx \\
	& ~~~~ - \dfrac{d_2}{\alpha+1} \int g z^\alpha \bigl( Z^{\alpha+1}-1\bigr) \dvh (v-\delta(t)) \idx \\
	& ~~~~ + \dfrac{d_3}{2} \int z^\alpha Z^{\alpha+1} \abs{(v-\delta(t))_t}{2} \idx 
	 + \dfrac{d_3 g}{2} \int Z^\alpha \abs{Z_t}{2} \idx \\
	& ~~~~ + \dfrac{d_4}{2} \int z^\alpha Z^{\alpha+1} \abs{\nablah (v-\delta(t))}{2} \idx + \dfrac{d_4 g}{2} \int Z^\alpha \abs{\nablah Z}{2} \idx \\
	& ~~~~ 
	 + \dfrac{d_5}{2} \int z^\alpha Z^{\alpha+1} \abs{\nablah^2 (v-\delta(t))}{2} \idx + \dfrac{d_5 g}{2} \int \abs{\nablah^2 Z}{2} \idx \\
	& ~~~~ + \dfrac{d_7}{2} \int z^\alpha Z^{\alpha+1} \abs{(v-\delta(t))_z}{2}\idx 
	 + \dfrac{d_9}{2} \int z^\alpha Z^{\alpha+1} \abs{\nablah (v-\delta(t))_{z}}{2} \idx
	, \\
	& \mathfrak G_{n,d}(t) : = \dfrac{1}{2} \int Z \abs{\nabla (v-\delta(t))}{2} \idx + d_1 \norm{Z_t}{\Lnorm{2}}^2 + \dfrac{d_2}{2} \int z^\alpha Z^\alpha \abs{(v-\delta(t))_t}{2} \idx \\
	& ~~~~ + \dfrac{d_3}{2} \int Z \abs{\nabla (v-\delta(t))_t}{2} \idx + \dfrac{d_4}{2}\int Z \abs{\nabla \nablah (v-\delta(t))}{2} \idx \\
	& ~~~~ + \dfrac{d_5}{2}\int Z\abs{\nabla \nablah^2 (v-\delta(t))}{2}\idx 
	 + d_6\norm{\nablah Z}{\Lnorm{2}}^2 + \dfrac{d_7}{4} \int Z\abs{\nabla v_{z}}{2} \idx \\
	& ~~~~ + d_8 \norm{\nablah^2 Z}{\Lnorm{2}}^2 
	 + \dfrac{d_9}{4} \int Z \abs{\nabla \nablah (v-\delta(t))_z}{2} \idx,
\end{align*}
satisfy
\begin{equation*}
	\dfrac{d}{dt} \mathfrak E_{n,d}(t)  + \mathfrak G_{n,d}(t) \leq \mathcal H(\mathcal E_{n}(t)) \mathfrak G_{n,d}(t).
\end{equation*}
Then similar arguments as in Proposition \refeq{lm:equal-of-energies}, one can choose $d_i$'s($i=1,2\cdots 9$) such that there is a constant $ C \in (0,\infty) $ such that
\begin{equation*}
	C^{-1} \mathcal E_n(t) \leq \mathfrak E_{n,d}(t) \leq C\mathcal E_n(t),
\end{equation*}
and similar Poincar\'e inequality in 
Proposition \refeq{lm:diss-dzzz-v} and Lemma \refeq{lm:poincare-ineq} hold for $ v $ replaced by $ v - \delta(t) $. Therefore, for $ \mathcal E_n(t) $ small enough, we arrive at
\begin{equation}
	\begin{gathered}
	\dfrac{d}{dt}\mathfrak E_{n,d}(t) + C_{n,1}\mathfrak E_d(t) \leq \dfrac{d}{dt} \mathfrak E_{n,d}(t) + \dfrac{1}{2} \mathfrak G_{n,d}(t) \\
	+ d_{10} \norm{z^{\alpha/2}(v-\delta(t))}{L^2}^2 \leq 0,
	\end{gathered}
\end{equation}
for some positive constants $ d_{10} $ and $ C_{n,1} <\infty $. Therefore, one has that
\begin{equation}
	\mathcal E_n(t) \leq C \mathfrak E_{n,d}(t) \leq C e^{-C_{n,1}t} \mathfrak E_{n,d}(0).
\end{equation}
Moreover, from \eqref{eq:momentum-non-conserv}, one has that
\begin{equation}
	|\delta'(t)| \leq \mathcal E_n(t) \leq C e^{-C_{n,1}t}\mathfrak E_{n,d}(0).
\end{equation}
Consequently, it follows that
\begin{equation}
	| \delta(t) - \delta_\infty | = |\int_t^\infty \delta'(s) \,ds| \leq \dfrac{C}{C_{n,1}} e^{-C_{n,1}t}\mathfrak E_{n,d}(0).
\end{equation}
Moreover, Proposition \ref{prop:H^2-of-v} and, in particular \eqref{apriori:H^2-v}, still hold with $ v $ and $ \mathcal E(t) $ replaced by $ v-\delta(t) $ and $ \mathcal E_n(t) $, respectively.

Therefore, we have proved the following theorem:
\begin{theorem}[Asymptotic Stability for the Non-conservative System]\label{thm:stability-non-conserv} Suppose that in addition to the conditions given in Theorem \ref{thm:global-non-conserv}, $ 0 < \alpha < 3 $(or equivalently $ \gamma > \frac{4}{3}$). 
Then there are constants $ \varepsilon ''' \in (0,1) $ small enough and $ C_{n,1}, C_{n,2} \in (0,\infty) $ such that if $ \mathcal E_{n,0} < \varepsilon''' $, we have
\begin{equation}
	\mathcal E_{n}(t) \leq e^{-C_{n,1}}C_{n,2} \mathcal E_{n,0}.
\end{equation}
Moreover, there exists $ \delta_\infty $ as defined in \eqref{lim:momentum-non-conserv} such that
\begin{equation}
	|Z-1| + |v-\delta_\infty| \leq \norm{Z-1}{H^2} + \norm{v-\delta_\infty}{H^2} \leq C_ne^{-C_nt},
\end{equation}
for some constant $ C_n \in (0,\infty ) $. 
	
\end{theorem}

\section*{Acknowledgements}
	The research of Z.X. is partially supported by Zheng Ge Ru Foundation, Hong Kong RGC Earmarked Research Grants CUHK-14301421, CUHK-14302819, and CUHK-14300819. The research of E.S.T. has benefited from the inspiring environment of the CRC 1114 ``Scaling Cascades in Complex Systems'', Project Number 235221301, Project C06, funded by Deutsche Forschungsgemeinschaft (DFG). Moreover, the work of E.S.T. was also supported in part by the DFG Research Unit FOR 5528 on Geophysical Flows.

\bibliographystyle{plain}

\end{document}